\documentclass[12pt]{amsart}  	

\usepackage{geometry}  

\usepackage[all]{xy}						
\usepackage{amssymb}
\usepackage{amscd,latexsym,amsthm,amsfonts,amssymb,amsmath,amsxtra}
\usepackage[mathscr]{eucal}
\usepackage[colorlinks]{hyperref}
\usepackage{mathtools}

\usepackage{chngcntr}
\numberwithin{equation}{subsection}
\newcounter{keepeqno}
\newenvironment{num}
 {\setcounter{keepeqno}{\value{equation}}%
  \begin{list}{(\theequation)}{\usecounter{equation}}%
  \setcounter{equation}{\value{keepeqno}}}
 {\end{list}}

\hypersetup{linkcolor=black, citecolor=black}

\newcommand{\BA}{{\mathbb {A}}}

\newcommand{\BC}{{\mathbb {C}}}

\newcommand{\BG}{{\mathbb {G}}}
\newcommand{\BH}{{\mathbb {H}}}

\newcommand{\BN}{{\mathbb {N}}}

\newcommand{\BQ}{{\mathbb {Q}}}
\newcommand{\BR}{{\mathbb {R}}}

\newcommand{\BZ}{{\mathbb {Z}}}

\newcommand{\CA}{{\mathcal {A}}}
\newcommand{\CB}{{\mathcal {B}}}
\newcommand{\CC}{{\mathcal {C}}}
\newcommand{\CE}{{\mathcal {E}}}
\newcommand{\CF}{{\mathcal {F}}}

\newcommand{\CH}{{\mathcal {H}}}

\newcommand{\CK}{{\mathcal {K}}}
\newcommand{\CL}{{\mathcal {L}}}

\newcommand{\CN}{{\mathcal {N}}}
\newcommand{\CO}{{\mathcal {O}}}
\newcommand{\CP}{{\mathcal {P}}}

\newcommand{\CS}{{\mathcal {S}}}
\newcommand{\CT}{{\mathcal {T}}}

\newcommand{\CW}{{\mathcal {W}}}
\newcommand{\CX}{{\mathcal {X}}}
\newcommand{\CY}{{\mathcal {Y}}}

\newcommand{\Fa}{{\mathfrak {a}}}
\newcommand{\Fb}{{\mathfrak {b}}}

\newcommand{\Fg}{{\mathfrak {g}}}
\newcommand{\Fh}{{\mathfrak {h}}}

\newcommand{\Fl}{{\mathfrak {l}}}

\newcommand{\Fn}{{\mathfrak {n}}}
\newcommand{\Fo}{{\mathfrak {o}}}
\newcommand{\Fp}{{\mathfrak {p}}}

\newcommand{\Fs}{{\mathfrak {s}}}
\newcommand{\Ft}{{\mathfrak {t}}}
\newcommand{\Fu}{{\mathfrak {u}}}

\newcommand{\RI}{{\mathrm {I}}}

\newcommand{\RM}{{\mathrm {M}}}
\newcommand{\RN}{{\mathrm {N}}}

\newcommand{\Ad}{{\mathrm{Ad}}}

\newcommand{\Cent}{{\mathrm{Cent}}}

\newcommand{\End}{{\mathrm{End}}}

\newcommand{\Gal}{{\mathrm{Gal}}}
\newcommand{\GL}{{\mathrm{GL}}}

\newcommand{\Hom}{{\mathrm{Hom}}}

\renewcommand{\Im}{{\mathrm{Im}}}

\newcommand{\inv}{{\mathrm{inv}}}

\newcommand{\Id}{{\mathrm{Id}}}

\newcommand{\Lie}{{\mathrm{Lie}}}

\renewcommand{\Re}{{\mathrm{Re}}}
\newcommand{\reg}{{\mathrm{reg}}}

\newcommand{\SL}{{\mathrm{SL}}}

\newcommand{\SO}{{\mathrm{SO}}}

\newcommand{\sgn}{{\mathrm{sgn}}}

\newcommand{\tr}{{\mathrm{tr}}}

\newcommand{\ud}{\,\mathrm{d}}
\newcommand{\vol}{{\mathrm{vol}}}

\newcommand{\udl}{\underline}
\newcommand{\wt}{\widetilde}
\newcommand{\wh}{\widehat}

\newcommand{\bs}{\backslash}

\def\alp{{\alpha}}
\def\bet{{\beta}}

\def\del{{\delta}}
\def\Del{{\Delta}}
\def\diag{{\rm diag}}
\def\eps{{\epsilon}}

\def\sig{{\sigma}}

\def\nil{\mathrm{Nil}}

\def\zet{{\zeta}}

\def\ome{{\omega}}
\def\Ome{{\Omega}}
\def\lam{{\lambda}}
\def\Lam{{\Lambda}}

\def\Sig{{\Sigma}}
\def\gam{{\gamma}}
\def\Gam{{\Gamma}}

\def\wb{\overline} 

\def\vpi{\varpi}
\def\LG{{}^{L}G}

\def\vphi{\varphi}

\def\p{\prime}
\def\rss{\mathrm{rss}}
\newcommand{\sslash}{\mathbin{/\mkern-6mu/}}
\def\geom{\mathrm{geom}}
\def\scusp{\mathrm{scusp}}
\def\Temp{\mathrm{Temp}}
\def\ss{\mathrm{ss}}
\def\ani{\mathrm{ani}}
\def\an{\mathrm{an}}
\def\qc{\mathrm{qc}}
\def\qd{\mathrm{qd}}
\def\stab{\mathrm{stab}}

\newtheorem{thm}{Theorem}[subsection]
\newtheorem{defin}[thm]{Definition}
\newtheorem{rmk}[thm]{Remark}
\newtheorem{ex}[thm]{Example}
\newtheorem{pro}[thm]{Proposition}
\newtheorem{lem}[thm]{Lemma}
\newtheorem{cor}[thm]{Corollary}

\newtheorem{conjec}[thm]{Conjecture}

\makeatletter

\newcommand{\Rmnum}[1]{\expandafter\@slowromancap\romannumeral #1@}
\makeatother

\begin{document}

\title[A local TF for the local GGP for Special Orthogonal Groups]{A Local Trace Formula for the Local Gan-Gross-Prasad Conjecture for Special Orthogonal Groups}
\author{Zhilin Luo}
\address{School of Mathematics\\
University of Minnesota\\
Minneapolis, MN 55455, USA}
\email{luoxx537@umn.edu}

\begin{abstract}
Through combining the work of Jean-Loup Waldspurger (\cite{waldspurger10} \cite{waldspurgertemperedggp}) and Raphaël Beuzart-Plessis (\cite{beuzart2015local}), we give a proof for the tempered part of the local Gan-Gross-Prasad conjecture (\cite{ggporiginal}) for special orthogonal groups over any local fields of characteristic zero, which was already proved by Waldspurger over $p$-adic fields. 
\end{abstract}

\maketitle

\tableofcontents

\section{Introduction}

The restriction and branching problems can be traced back to the theory of sphere harmonics. Let $\SO_{2}(\BR)\hookrightarrow \SO_{3}(\BR)$ be a pair of compact special orthogonal groups associated to $2$, resp. $3$ dimensional anisotropic quadratic spaces over $\BR$. For any finite dimensional irreducible representation $\pi$ of $\SO_{3}(\BR)$, the spectral decomposition of the restriction of $\pi$ to $\SO_{2}(\BR)$ can be detected by the spectral decomposition of the canonical action of $\SO_{3}(\BR)$ on $L^{2}(\SO_{3}(\BR)/\SO_{2}(\BR))\simeq L^{2}(S^{2})$ where $S^{2}$ is the $2$-dimensional real sphere, which is essentially provided by the classical theory of spherical harmonics.

The local Gan-Gross-Prasad conjecture aims to generalize the phenomenon to non-compact classical groups over any local field $F$ of characteristic zero. Let $(W,V)$ be a pair of quadratic space defined over $F$ with $W$ a subspace of $V$ and the orthogonal complement $W^{\perp}$ of $W$ in $V$ is odd dimensional and split. Let $G=\SO(W)\times \SO(V)$ be the product of two special orthogonal groups, and $H$ the subgroup of $G$ defined by the semi-direct product of $\SO(W)$ with the unipotent subgroup $N$ of $\SO(V)$ defined by the full isotropic flag determined by $W^{\perp}$. Fix a generic character $\xi$ of $N$ which extends to $H$. For any irreducible smooth admissible representation $\pi$ of $G(F)$, consider the dimension of the following space
$$
m(\pi) = 
\dim \Hom_{H}(\pi,\xi).
$$
By \cite{agrsmut1} \cite{ggporiginal} when $F$ is $p$-adic, and \cite{szmut1} \cite{jszmut1bessel} when $F$ is archimedean, $m(\pi)\leq 1$.

The local Gan-Gross-Prasad conjecture speculates a refinement for the behavior of $m(\pi)$. The pure inner inner forms of the special orthogonal group $\SO(W)$ are parametrized by the set $H^{1}(F,\SO(W))\simeq H^{1}(F,H)$. For each $\alp\in H^{1}(F,H)$, one can associate a pair of quadratic space $(W_{\alp},V_{\alp})$ that enjoys similar properties as $(W,V)$. Moreover, $G_{\alp} = \SO(W_{\alp})\times \SO(V_{\alp})$ has the same Langlands dual group as $G$, which is denoted as $\LG$. For any generic $L$-parameter $\vphi:\CW_{F}\to \LG$ where $\CW_{F}$ is the Weil-Deligne group of $F$, let $\Pi^{G_{\alp}}(\vphi)$ be the associated $L$-packet of $G_{\alp}$.

\begin{conjec}[\cite{ggporiginal}]\label{ggpconjec}
The following identity holds for any generic $L$-parameter $\vphi:\CW_{F}\to \LG$.
$$
\sum_{\alp\in H^{1}(F,H)}
\sum_{\pi\in \Pi^{G_{\alp}}(\vphi)}
m(\pi)  =1.
$$
\end{conjec} 
Moreover, there is a refinement determining when $m(\pi)$ is not equal to zero. It can be detected by the representation of the component group $A_{\vphi}$ associated to $\vphi$, which in particular can be detected by the sign of the relevant symplectic root number.

Over $p$-adic fields, in a series of papers (\cite{waldspurger10}, \cite{waldspurgertemperedggp}, \cite{MR3155344}), Waldspurger proved a multiplicity formula for any tempered representation of $G(F)$ and twisted general linear groups, from which he could deduce Conjecture \ref{ggpconjec} and the refinement for any tempered $L$-parameters of $G$. Later the work of M. Moeglin and Waldspurger (\cite{MR3155346}) established the full conjecture for any generic $L$-parameters of $G$.

There is a similar conjecture for unitary groups associated to a pair of hermitian spaces with parallel descriptions. Over $p$-adic fields, following the approach of Waldspurger, Beuzart-Plessis established Conjecture \ref{ggpconjec} and the refinement for any tempered $L$-parameters of $G$ over $p$-adic fields (\cite{raphaelpadic}, \cite{MR3251763}). Later, after combining the work of Beuzart-Plessis, Moeglin and Waldspurger, W. Gan and A. Ichino established the full conjecture for any generic $L$-parameters of $G$ (\cite{MR3573972}).

The cases above are usually called the \emph{Bessel case}. There are also parallel conjectures called the \emph{Fourier-Jacobi case}, which treats unitary groups (skew-hermitian) and symplectic-metaplectic groups. Over $p$-adic fields, they were established by Gan and Ichino for unitary groups case (skew-hermitian) (\cite{MR3573972}), and H. Atobe for symplectic-metaplectic case (\cite{MR3788848}), using the techniques of theta correspondence.

Over archimedean local fields, Beuzart-Plessis established Conjecture \ref{ggpconjec} for tempered $L$-parameters of unitary groups in the Bessel case (\cite{beuzart2015local}). The method works uniformly for $p$-adic fields as well.

When $F=\BC$, Conjecture \ref{ggpconjec} for generic $L$-parameters of special orthogonal groups (\cite{mollers2017symmetry}) is proved using different methods.

The goal of the present paper is to prove Conjecture \ref{ggpconjec} for tempered $L$-parameters of special orthogonal groups. In particular, the result is new only when $F$ is archimedean. We are going to combine the work of Waldspurger (\cite{waldspurger10} \cite{waldspurgertemperedggp}) over $p$-adic fields and Beuzart-Plessis (\cite{beuzart2015local}) over archimedean fields to prove the conjecture. 
\newline

\paragraph{\textbf{The proof}}

In the following, we will give a brief introduction to the proof.

Fix a tempered representation $\pi$ of $G(F)$. To study the multiplicity space $\Hom_{H}(\pi,\xi)$, by Frobenius reciprocity, it is equivalent to study the spectral decomposition of $L^{2}(H\bs G,\xi)$ as a unitary representation $R$ of $G(F)$ via right translation, where $L^{2}(H\bs G,\xi)$ is the $L^{2}$-induction from the character $\xi$ of $H(F)$ to $G(F)$. From the spirit of J. Arthur (\cite{ltfjarthur}), one may use the convolution action of $\CC^{\infty}_{c}(G)$ on $L^{2}(H\bs G, \xi)$, and study the spectral and geometric expansion of the trace of the operator $R(f)$. 

More precisely, fix $f\in \CC^{\infty}_{c}(G)$, $x\in G(F)$ and $\vphi\in L^{2}(H\bs G,\xi)$, 
$$
(R(f)\vphi)(x) = 
\int_{G(F)}
f(g)\vphi(xg)\ud g = \int_{H(F)\bs G(F)}
K_{f}(x,y)\vphi(y)\ud y,
$$
where 
$$
K_{f}(x,y) = \int_{H(F)}f(x^{-1}hy)\xi(h)\ud h, \quad x,y\in G(F).
$$
In other words, the operator $R(f)$ has integral kernel $K_{f}(x,y)$. 

Formally, the trace of $R(f)$ can be computed via integrating $K_{f}(x,y)$ along the diagonal $H(F)\bs G(F)$. Unfortunately the integral is usually not absolutely convergent. To rectify the issue one works instead with the space of \emph{strongly cuspidal} functions. Recall that a function $f\in \CC^{\infty}_{c}(G)$ is called strongly cuspidal if 
$$
\int_{U(F)}f(mu)\ud u=0, \quad m\in M(F)
$$
for any nontrivial parabolic subgroup $P$ of $G$ with Levi decomposition $P=MU$. Similarly, one can also introduce the notion of strongly cuspidal functions in the \emph{Harish-Chandra Schwartz space} of $G$, which is denoted by $\CC_\scusp(G)$.

It turns out that the following theorem holds .

\begin{thm}[Section~\ref{sec:distribution}]
For any $f\in \CC_{\scusp}(G)$, the following integral is absolutely convergent 
$$
J(f) = 
\int_{H(F)\bs G(F)}
K_{f}(x,x)\ud x.
$$
\end{thm}

The next step is to establish the spectral and geometric expansions for $J(f)$.

\subsubsection{Spectral expansion}
The spectral expansion is similar to the situation of unitary groups (\cite{beuzart2015local}), except mild structural and combinatorial difference between special orthogonal groups and unitary groups. 

More precisely, let $\CX(G)$ be the set of virtual tempered representations built from Arthur's elliptic representations of $G$ and its Levi subgroups (for details see \cite[2.7]{beuzart2015local}). For any $\pi\in \CX(G)$, Arthur defined a weighted character 
$$
f\in \CC(G) \to J_{M(\pi)}(\pi,f)
$$
where $M(\pi)$ is the Levi such that $\pi$ is induced from an elliptic representation of $M(\pi)$.

For any $f\in \CC_{\scusp}(G)$, define $\wh{\theta}_{f}$ on $\CX(G)$ via 
$$
\wh{\theta}_{f}(\pi) = 
(-1)^{a_{M}(\pi)}
J_{M(\pi)}(\pi,f),\quad \pi\in \CX(G),
$$
where $a_{M(\pi)}$ is the dimension of the maximal central split sub-torus of $M(\pi)$. Let $D(\pi)$ be the discriminant coming from Arthur's definition of elliptic representations. We have the following spectral expansion for $J(f)$.

\begin{thm}[Section~\ref{sec:spectralexp}]
For any $f\in \CC_{\scusp}(G)$, 
$$
J(f) = \int_{\CX(G)}D(\pi)\wh{\theta}_{f}(\pi)m(\pi)\ud \pi.
$$
\end{thm}

\subsubsection{Geometric expansion}

The geometric expansion turns out to be different from the situation of unitary groups. We first give the statement of the geometric expansion.
\begin{thm}[Section~\ref{sec:7.4}]
For any $f\in \CC_{\scusp}(G)$, 
$$
J(f)  = \lim_{s\to 0^+}\int_{\Gam(G,H)}
c_{f}(x)D^{G}(x)^{1/2}\Del(x)^{s-1/2}\ud x.
$$
\end{thm}
Here $\Gam(G,H)$ is a subset in the set of semi-simple conjugacy classes of $H$, equipped with topology (Section \ref{sec:linearformgeom}).

The definition of $c_{f}$, which originally appears in \cite{waldspurger10}, is the key ingredient different from the unitary group case appearing in \cite{beuzart2015local}. We explicate in more detail below. 

Recall that Arthur introduced the notion of weighted orbital integrals for any Levi subgroup $M$ of $G$ and $x\in M(F)\cap G^{\rss}(F)$
$$
f\in \CC(G)\to J_{M}(x,f).
$$
For $f\in \CC_{\scusp}(G)$, define (\cite[5.2]{beuzart2015local})
$$
\theta_{f}(x) = 
(-1)^{a_{G}-a_{M(x)}}\nu(G_{x})^{-1}
D^{G}(x)^{-1/2}J_{M(x)}(x,f).
$$
The term $\theta_{f}(x)$ is conjugation invariant, and it is a \emph{quasi-character} (\cite[\S 4]{beuzart2015local}) in the sense that it has germ expansion as the distribution character of admissible representations. The term $c_{f}(x)$ is equal to the regular germ associated to a particular regular nilpotent orbit in $\Fg(F)$. The detailed construction will be specified in Section \ref{sec:linearformgeom}. 

A key difference between the special orthogonal groups and unitary groups is that, for unitary groups, the regular nilpotent orbits can be permuted by scaling, which is not the case for special orthogonal groups. Hence, for unitary groups, the term $c_f$ is taken to be the average of all regular nilpotent germs, which turns out to be not the case for special orthogonal groups. We need to choose a particular regular nilpotent orbit (Section \ref{sec:linearformgeom}). 

In order to determine which regular nilpotent orbit shows up in the geometric expansion, we follow the idea of Waldspurger. After localization (Proposition \ref{pro:7.6.1} and Proposition \ref{pro:7.7.1}), the problem can be reduced to Lie algebra. To adapt the strategy of Waldspurger (\cite[\S 11.4-\S11.6]{waldspurger10}), in appendix \ref{sec:germformula}, we express the regular germs of orbital integrals for any quasi-split connected reductive algebraic groups in turns of endoscopic invariants, which was proved by D. Shelstad over $p$-adic fields (\cite{regulargermshelstad}). Then in appendix \ref{sec:germcal}, we review the work of Walspurger (\cite{wald01nilpotent}) computing the transfer factors over $p$-adic fields, which pass without hard effort to archimedean local fields. Finally we are able to prove the geometric expansion following \cite[\S 11.4-\S11.6]{waldspurger10}.

After establishing the spectral and geometric expansions for the distribution $J(f)$, following \cite[Section~11]{beuzart2015local}, we obtain the geometric multiplicity formula $m_\geom(\pi)$ for any tempered representation $\pi$ of $G$ (Section \ref{sec:7.4}). Then following the strategy of \cite[Section~12]{beuzart2015local}, Conjecture \ref{ggpconjec} for tempered $L$-parameters is established. An important step is Theorem \ref{thm:replacegermstable}, which shows that for a \emph{stable} quasi-character $\theta$ on $G(F)$, $c_\theta$ defined in Section \ref{sec:linearformgeom} is actually the same as the one introduced in \cite[Section~5]{beuzart2015local}.

When the ground field is $\BC$, Conjecture \ref{ggpconjec} for tempered $L$-parameters follows easily from the invariance of the multiplicity under parabolic induction (Corollary \ref{cor:3.6.1}), which reduce the conjecture to trivial case. Therefore when working with trace formula, we will assume that the ground field is not $\BC$. More precisely, starting from Section \ref{sec:spectralexp}, the ground field is assumed to be either $p$-adic or real.

To save the length of the paper, many technical details similar to \cite{waldspurger10}, \cite{waldspurgertemperedggp}, \cite{beuzart2015local} are omitted. The reader is referred to the upcoming thesis of the author \cite{thesis_zhilin} for full detail.
\newline

\paragraph{\textbf{Notation and conventions}}
We will freely use notation from the first five sections in \cite{beuzart2015local}. For convenience we list them below.
\begin{itemize}
\item $F$ is a local field of characteristic zero with fixed norm $|\cdot|$;

\item For a connected reductive algebraic group $G$ defined over $F$, $\Temp(G)$ is denoted as the set of tempered representations of $G$ (\cite[2.2]{beuzart2015local});

\item Gothic letters are used to denote the Lie algebras of the associated algebraic groups;

\item For a locally compact Hausdorff totally disconnected topological space (resp. real smooth manifold) $M$, $\CC^\infty_c(M)$ is the space of smooth functions on $M$, and $C_c(M)$ is the space of continuous functions on $M$ (\cite[1.4]{beuzart2015local});

\item $\CC(G)$ is the space of Harish-Chandra Schwartz functions on $G(F)$, $\Xi^G$ is the associated Harish-Chandra $\Xi$-function (\cite[1.5]{beuzart2015local}), and $\CC_\scusp(G)$ is the space of strongly cuspidal Harish-Chandra Schwartz functions on $G(F)$ (\cite[5.1]{beuzart2015local});

\item $\CS(\Fg) = \CS(\Fg(F))$ is the space of Schwartz-Bruhat functions on $\Fg(F)$ (\cite[1.4]{beuzart2015local}), and $\CS_\scusp(\Fg)$ is the space of strongly cuspidal Schwartz-Bruhat functions on $\Fg(F)$ (\cite[5.1]{beuzart2015local});

\item $\Gam(G)$ (resp. $\Gam(\Fg)$) is the set of semi-simple conjugacy classes in $G(F)$ (resp. $\Fg(F)$)
equipped with topologies (\cite[1.7]{beuzart2015local}), and subscript $\mathrm{ell}$ (resp. $\mathrm{rss}$, $\mathrm{ss}$) denotes the elliptic (resp. regular semi-simple, semi-simple) elements (\cite[1.7]{beuzart2015local});

\item For $x\in G_\ss(F)$ (resp. $X\in \Fg(F)$), $Z_G(x)$ (resp. $Z_G(X)$) is the centralizer of $x$ (resp. $X$) in $G(F)$, and $G_x$ is the identity connected component of $Z_G(x)$ (Note that $G_X= Z_G(X)$ is connected, c.f. \cite[p.16]{beuzart2015local});

\item For the notion of $G$-good (resp. $G$-excellent) open subsets, we refer to \cite[3.2]{beuzart2015local} (resp. \cite[3.3]{beuzart2015local});

\item For the notion and basic properties of log-norms $\sig=\sig_X$ on an algebraic variety $X$, we refer to \cite[1.2]{beuzart2015local};

\item Let $QC(G)$ (resp. $QC(\Fg)$) be the space of quasi-characters on $G(F)$ (resp. $\Fg(F)$). Let $SQC(\Fg)$ be the space of Schwartz quasi-characters on $\Fg(F)$. Let $QC_c(G)$ (resp. $QC_c(\Fg)$) be the space of quasi-characters that are compactly supported modulo conjugation. We refer definitions and basic properties of quasi-characters to \cite[4.1,4.2]{beuzart2015local};

\item For an irreducible admissible representation $\pi$ of $G$, $\pi^\infty$ is the space of smooth vectors, and $\theta_\pi$ is the distribution character of $\pi$;

\item For the definition of sets $R_{\mathrm{temp}}(G), \CX(G), \CX_\mathrm{temp}(G), \CX_{\mathrm{ell}}(G)$, see \cite[2.7]{beuzart2015local};

\item
For $f\in \CC_\scusp(G)$, set (\cite[5.2]{beuzart2015local})
$$
\theta_{f}(x) = 
(-1)^{a_{G}-a_{M(x)}}\nu(G_{x})^{-1}
D^{G}(x)^{-1/2}J_{M(x)}(x,f).
$$
\end{itemize}

\paragraph{\textbf{Acknowledgement}}
I would like to thank my advisor D. Jiang for suggesting the problem to me. I would also like to thank R. Beuzart-Plessis discussing the relation between the regular germ formula and the Kostant sections. I am particularly grateful to C. Wan for a lot of helpful discussions and various critical corrections when I was writing down the paper.
The work is supported in part through the NSF Grant: DMS–1901802 of D. Jiang.

\section{The Gan-Gross-Prasad triples}\label{sec:ggptriples}
Following \cite[Section~6]{beuzart2015local}, the concept of Gan-Gross-Prasad triples in the special orthogonal groups setting is introduced. Various algebraic, geometric and analytical properties for the triple are listed.

\subsection{Definitions}\label{sub:ggptriples:definitions}

\subsubsection{Quadratic spaces and orthogonal groups}
Let $(V,q)$ be a finite dimensional quadratic space. Let $\mathrm{O}(V)$ ($\SO(V)$) be the corresponding (special) orthogonal group and $\Fs\Fo(V)$ the Lie algebra defined over $F$. 

\begin{ex}\label{ex:ggptriples:def}
\begin{enumerate}
\item
Let $(D_{\nu},q)$ be an anisotropic line, i.e. $D_{\nu}\simeq F^{\times}$ with quadratic form defined by $q(x) = \nu x^{2}$, where $x\in F^{\times}$ and $\nu\in F^{\times}$.

\item Let $\BH$ be the $2$-dimensional split quadratic space, i.e. $\BH\simeq Fv\oplus Fv^{*}$ with quadratic form defined by $q(v,v^{*}) = 1$, $q(v,v) = q(v^{*},v^{*}) = 0$.

\item For any $b,c\in F^{\times}$, let $(E =F(\sqrt{b}),c\cdot \RN_{E/F})$ be the $2$-dimensional quadratic space sending $m\oplus n\sqrt{b}$ to $c(m^{2}-bn^{2})$, where $m,n\in F^{\times}$.
\end{enumerate}
\end{ex}

\begin{defin}\label{defin:ggptriples:def}
$(V,q)$ is \emph{quasi-split} if there exists some positive integer $n\in \BN$ such that
$$
(V,q)\simeq \BH^{n-1}\oplus 
\bigg\{
\begin{matrix}
(E=F(\sqrt{b}),c\cdot \RN_{E/F}), & \text{ if $\dim V$ is even}\\
D_{\nu}, & \text{ if $\dim V$ is odd}
\end{matrix}
$$
\end{defin}
Assume that $(V,q)$ is quasi-split. When $\dim V$ is odd, $(V,q)$ is always split; when $\dim V$ is even, $(V,q)$ is split if and only if $b\in F^{\times 2}$.

\subsubsection{Regular nilpotent orbits}

The set of $\SO(V)(F)$-regular nilpotent orbits in $\Fs\Fo(V)(F)$, which is denoted as $\nil_{\reg}(\Fs\Fo(V)) = \nil_{\reg}(\Fs\Fo(V)(F))$, admits the following description. Following \cite[Chaptire~I]{wald01nilpotent}, $\nil_{\reg}(\Fs\Fo(V))$ is nonempty only when $(V,q)$ is quasi-split. When $(V,q)$ is quasi-split, the regular nilpotent orbits form a single stable conjugacy class. Up to conjugation by $\SO(V)(F)$, they are contained in the $F$-points of a fixed $F$-rational Borel subgroup $B$ with Levi decomposition $B=TU$. Since conjugation by $U(F)$ stabilizes the regular nilpotent orbits, one only need to compute the conjugation action of $T(F)$. Through straightforward computation the following descriptions for $\nil_{\reg}(\Fs\Fo(V))$ can be deduced.
\begin{itemize}
\item When $\dim V$ is odd or $\leq 2$, $\nil_{\reg}(\Fs\Fo(V))$ contains only one element.

\item When $\dim V = 2m$ is even and $\geq 4$,
$$
(V,q) \simeq
\bigg\{
\begin{matrix}
\BH^{m},	& \text{ when $(V,q)$ is split},\\
\BH^{m-1}\oplus (E=F(\sqrt{b}),c\cdot \RN_{E/F}), & \text{ when $(V,q)$ is not split}.
\end{matrix}
$$
In particular, when $(V,q)$ is not split $b\notin F^{\times 2}$.

Let
$$
\CN^{V} = 
\bigg\{
\begin{matrix}
F^{\times}/F^{\times 2}, & \text{ when $(V,q)$ is split},\\
c\cdot \RN_{E/F}(E^{\times})/F^{\times 2}, & \text{ when $(V,q)$ is not split}.
\end{matrix}
$$
Then $\CN^{V}$ is in bijection with $\nil_{\reg}(\Fs\Fo(V))$.

More precisely, for any $\nu\in \CN^{V}$ with fixed representative in $F^{\times}$ which is still denoted by $\nu$, an explicit representative in $\nil_{\reg}(\Fs\Fo(V))$ can be constructed as follows. Write $V=D\oplus W$ such that $D$ is an anisotropic line, and the restriction of $q$ to $W$ has the quadratic form $x\to \nu x^{2}$ when restricted to the  anisotropic line in $W$. Let $H=\SO(W)$ where $(W,q|_{W})$ is split of odd dimension. It has a unique regular nilpotent orbit in $\Fs\Fo(W)(F)$. Let the associated nilpotent orbit in $\Fs\Fo(V)(F)$ be $\CO_{\nu}$. Then one can show that the orbit does not depend on the choice of the representative of $\nu$, and the set $\{ \CO_{\nu}|\quad \nu\in \CN^{V}\}$ is in bijection with $\nil_{\reg}(\Fs\Fo(V))$.
\end{itemize}

\subsubsection{Definition of GGP triples}
Let $(W,V)$ be a pair of quadratic spaces. $(W,V)$ is called \emph{admissible} if there exists an anisotropic line $D$ and a hyperbolic space of $Z$ of dimension $2r$ such that
$$
V\simeq W\oplus D\oplus Z.
$$
In particular, there exists a basis $\{ z_{i}\}_{i=\pm 1}^{\pm r}$ of $Z$ such that 
$$
q(z_{i},z_{j}) = \del_{i,-j}, \quad \forall i,j\in \{\pm 1,...,\pm r \}.
$$
There exist $\nu_{0}\in F^{\times}$ and $z_{0}\in D$ such that $q(z_{0},z_{0}) = \nu_{0}$. 

Let $P_{V}$ be the parabolic subgroup of $\SO(V)$ with Levi decomposition $P_{V}=M_{V}N$ stabilizing the following totally isotropic flag of $V$
$$
\langle z_{r} \rangle
\subset \langle z_{r},z_{r-1} \rangle 
\subset...\subset \langle z_{r},...,z_{1} \rangle.
$$

Set $G=\SO(W)\times \SO(V)$ and $P =\SO(W)\times P_{V}$. Then $P$ is a parabolic subgroup of $G$ with Levi decomposition $P=MN$ where $M=\SO(W)\times M_{V}$. Using the diagonal embedding $\SO(W)\hookrightarrow G$,  $\SO(W)$ can be identified as an algebraic subgroup of $G$ contained in $M$. In particular $\SO(W)$ acts via conjugation on $N$. Set
$$
H=\SO(W)\ltimes N.
$$

Define a morphism $\lam:N\to \BG_{a}$ by
$$
\lam(n) = \sum_{i=0}^{r-1}q(z_{-i-1}, nz_{i}), \quad n\in N.
$$
Then $\lam$ is $\SO(W)$ conjugation invariant and hence $\lam$ has a unique extension to $H$ that is trivial on $\SO(W)$, which is still denoted as $\lam$. Let $\lam_{F}:H(F)\to F$ be the induced morphism on $F$-rational points. Set
$$
\xi(h) = \psi(\lam_{F}(h)), \quad h\in H(F).
$$

\begin{defin}\label{defin:ggptriple}
The triple $(G,H,\xi)$ defined above is called the \emph{Gan-Gross-Prasad triple} associated to the admissible pair $(W,V)$.
\end{defin}

Besides the above notation and definitions, the following notions are introduced.

\begin{itemize}
\item $d=\dim (V)$, $m=\dim (W)$;

\item $Z_{+} = \langle z_{r},...,z_{1}\rangle$, $Z_{-} = \langle z_{-1},...,z_{-r}\rangle$;

\item $D=F\cdot z_{0}$, $V_{0} = W\oplus D$;

\item $H_{0}=\SO(W)$, $G_{0} = \SO(W)\times \SO(V_{0})$. The triple $(G_{0},H_{0},1)$ is associated to $(W,V_{0})$. It is called the \emph{codimension one} case;

\item $T$ is the subtorus of $\SO(V)$ stabilizing lines $\langle z_{i}\rangle$ for $i=\pm 1,...,\pm r$ and acting trivially on $V_{0}$. In particular $M=T\times G_{0}$;

\item Let $\Fh = \Lie(H)$. Using the same notation as before, $\xi$ is the character on $\Fh(F)$ that is trivial on $\Fs\Fo(W)(F)$ and equal to $\xi\circ \exp$ on $\Fn(F)$;

\item $B(\cdot,\cdot)$ is the following non-degenerate $G(F)$-invariant bilinear form on $\Fg(F)$,
$$
B((X_{W},X_{V}),(X_{W}^{\p},X^{\p}_{V})) = 
\frac{1}{2}
(\tr(X_{W}X_{W}^{\p})+\tr(X_{V}X^{\p}_{V}));
$$
\end{itemize}

The following lemma, whose proof is similar to \cite[Lemma~6.2.1]{beuzart2015local}, will be useful for later discussion.

\begin{lem}\label{lem:ggptriple:denseroot}
\begin{enumerate}
\item
The map $G\to H\bs G$ has the norm descent property.

\item 
The $M$-conjugacy orbit of $\lam$ in $(\Fn/[\Fn,\Fn])^{*}$ is Zariski open.
\end{enumerate}
\end{lem}

\subsubsection{The multiplicity $m(\pi)$}
For any $\pi\in \mathrm{Temp}(G)$, let $\Hom_{H}(\pi,\xi)$ be the space of continuous linear forms $\ell:\pi^{\infty}\to \BC$ satisfying
$$
\ell(\pi(h)e) = \xi(h)\ell(e),\quad \forall e\in \pi^{\infty},h\in H(F).
$$

Define
$$
m(\pi) = \dim \Hom_{H}(\pi,\xi).
$$

The following theorem is proved in \cite{jszmut1bessel} for archimedean case (for $r=0$ case it is also proved in \cite{szmut1}) and through a combination of \cite{agrsmut1} ($r=0$ case) and \cite{ggporiginal} (extending to general $r$ case) for $p$-adic case.

\begin{thm}\label{thm:ggptriples:multiplictyone}
The following inequality holds for any $\pi\in \mathrm{Temp}(G)$
$$
m(\pi)\leq 1.
$$
\end{thm}

Following \cite[6.3.1]{beuzart2015local} the identity
$$
m(\wb{\pi}) = m(\pi)
$$
holds where $\wb{\pi}$ is the complex conjugation of $\pi$.

\subsection{Spherical variety structure}\label{sub:ggptriples:sphvar}

A parabolic subgroup $\wb{Q}$ of $G$ is called \emph{good} if $H\wb{Q}$ is Zariski-open in $G$. Through taking the generic fiber the condition is equivalent to $H(F)\wb{Q}(F)$ being open in $G(F)$.

\begin{pro}\label{pro:ggptriples:sphvar}
\begin{enumerate}
\item There exists minimal parabolic subgroups of $G$ that are good, and they are conjugate under $H(F)$. Moreover, for a good minimal parabolic subgroup $\wb{P}_{\min} = M_{\min}\wb{U}_{\min}$, $H\cap \wb{U}_{\min} = \{ 1\}$, and the complement of $H(F)\wb{P}_{\min}(F)$ in $G(F)$ has zero measure;

\item A parabolic subgroup $\wb{Q}$ of $G$ is good if and only if it contains a good minimal parabolic subgroup.
\end{enumerate}
\end{pro}
\begin{proof}
The proof follows from the same argument as \cite[Propostion~6.4.1]{beuzart2015local}, except the fact that $H^{1}(F,\SO(W))$ classifies the isomorphism classes of quadratic spaces of the same dimension and discriminant as $W$ (\cite[29.29]{MR1632779}) (c.f. \cite[p.150]{beuzart2015local}). 
\end{proof}

\subsubsection{Relative weak Cartan decompositions}
In the codimension one case $(G_{0},H_{0},1)$, from Proposition \ref{pro:ggptriples:sphvar}, $G_{0}$ admits good minimal parabolic subgroups. Let $\wb{P}_{0} = M_{0}\wb{U}_{0}$ be a good minimal parabolic subgroup of $G_{0}$. Let $A_{0} = A_{M_{0}}$ be the maximal central split torus of $M_{0}$. Set
$$
A^{+}_{0} = 
\{ 
a\in A_{0}(F)|\quad
|\alp(a)|\geq 1, \forall \alp\in R(A_{0},\wb{P}_{0})
\}.
$$

\begin{pro}\label{pro:ggptriples:cartandecomp}
There exists a compact subset $\CK_{0}\subset G_{0}(F)$ such that
$$
G_{0}(F) = H_{0}(F)A^{+}_{0}\CK_{0}.
$$
\end{pro}
\begin{proof}
Together with \cite[7.2]{waldspurger10}, the discussion of \cite[Section~6.6]{beuzart2015local} works in parallel with the current situation.
\end{proof}

For a general GGP triple $(G,H,\xi)$, let $(\wb{P}_{0},M_{0},A_{0},A^{+}_{0})$ be as above. Let $\wb{P} = M\wb{N}$ be the parabolic subgroup opposite to $P$ w.r.t. $M$. Set
$$
A_{\min} = A_{0}T\subset M_{\min} =M_{0}T\subset \wb{P}_{\min} = \wb{P}_{0}T\wb{N}.
$$
Then $\wb{P}_{\min}$ is a parabolic subgroup, $M_{\min}$ is the Levi component of it and $A_{\min}$ is the maximal central split subtorus of $M_{\min}$. Moreover, from the proof of \cite[Proposition~6.4.1]{beuzart2015local}, $\wb{P}_{\min}$ is a good parabolic subgroup of $G$. 
Set
$$
A^{+}_{\min} = \{
a\in A_{\min}(F)|\quad 
|\alp(a)|\geq 1, \forall \alp\in R(A_{\min},\wb{P}_{\min})
 \}.
$$
Let $P_{\min}$ be the parabolic subgroup opposite to $\wb{P}_{\min}$ w.r.t. $M_{\min}$. Then $P_{\min}\subset P$. Let $\Del$ be the set of simple roots of $A_{\min}$ in $P_{\min}$ and set $\Del_{P} = \Del \cap R(A_{\min},N)$ be the set of simple roots appearing in $\Fn = \Lie(N)$. For $\alp\in \Del_{P}$ let $\Fn_{\alp}$ be the corresponding root subspace. 

\begin{lem}\label{lem:ggptriple:cartandecomp}
The following properties hold.

\begin{enumerate}
\item 
$$
A^{+}_{\min} = \{
a\in A^{+}_{0}T(F) |\quad |\alp(a)|\leq 1, \forall \alp\in \Del_{P}
 \};
$$

\item There exists a compact subset $\CK_{G}$ of $G(F)$ such that 
$$
G(F)= H(F)A^{+}_{0}T(F)\CK_{G};
$$ 

\item For any $\alp\in \Del_{P}$, the restriction of $\xi$ to $\Fn_{\alp}(F)$ is nontrivial.
\end{enumerate}
\end{lem}
\begin{proof}
The proof of \cite[Lemma~6.6.2]{beuzart2015local} works verbatim.
\end{proof}

\subsection{Analytical estimations}\label{subsec:analyticalestimate}
Various analytical estimations are stated associated to the GGP triples $(G,H,\xi)$. The results
will be used to establish the results for explicit tempered intertwining, the convergence of the distribution $J$ and $J^\Lie$, and the spectral expansion for the distribution $J$, which are introduced in the next sections. However, as we omit most of the details (those will appear in \cite{thesis_zhilin}) that are similar to those in \cite{beuzart2015local}, \cite{waldspurger10} and \cite{waldspurgertemperedggp}, they do not show up explicitly in the context. For the sake of completeness we prefer to keep them here.

\begin{lem}\label{lem:ggptriple:sigest}
Let $\wb{P}_{\min}= M_{\min}\wb{U}_{\min}$ be a good minimal parabolic subgroup and let $A_{\min}$ be the maximal split central torus of $M_{\min}$. Recall that
$$
A^{+}_{\min} = \{ a\in A_{\min}(F)|\quad |\alp(a)|\geq 1, \forall \alp\in R(A_{\min}, \wb{P}_{\min}) \}.
$$
Then the following inequalities hold.
\begin{enumerate}
\item $\sig(h)+\sig(a)\ll \sig(ha)$ for any $a\in A^{+}_{\min}$ and $h\in H(F)$;

\item $\sig(h)\ll \sig(a^{-1}ha)$ for any $a\in A^{+}_{\min}$ and $h\in H(F)$.
\end{enumerate}
\end{lem}
\begin{proof}
The proof follows from the same argument as \cite[Proposition~6.4.1~(iii)]{beuzart2015local}.
\end{proof}

\subsubsection{Estimations for Harish-Chandra $\Xi$ functions}
Various estimations are stated for the integrations of the Harish-Chandra $\Xi$ function of $G$ on $H$.

\begin{lem}\label{lem:ggptriple:hcxifunction}
\begin{enumerate}
\item There exists $\eps>0$ such that the integral 
$$
\int_{H_{0}(F)}
\Xi^{G_{0}}(h_{0})e^{\eps \sig(h_{0})}\ud h_{0}
$$
is absolutely convergent.

\item There exists $d>0$ such that the integral 
$$
\int_{H(F)}
\Xi^{G}(h)\sig(h)^{-d}\ud h
$$
is absolutely convergent.

\item For any $\del>0$ there exists $\eps>0$ such that the integral 
$$
\int_{H(F)}\Xi^{G}(h)e^{\eps \sig(h)}(1+|\lam(h)|)^{-\del}\ud h
$$
is absolutely convergent.

\item Let $\wb{P}_{\min} = M_{\min}U_{\min}$ be a good minimal parabolic subgroup of $G$. Then the following result holds.

For any $\del>0$ there exists $\eps >0$ such that the integral 
$$
I^{1}_{\eps, \del}(m_{\min}) = 
\int_{H(F)}
\Xi^{G}(hm_{\min})e^{\eps\sig(h)}
(1+|\lam(h)|^{-\del})\ud h
$$
is absolutely convergent for any $m_{\min}\in M_{\min}(F)$ and there exists $d>0$ such that 
$$
I^{1}_{\eps, \del}(m_{\min})
\ll \del_{P_{\min}}(m_{\min})^{-1/2}
\sig(m_{\min})^{d}
$$
for any $m_{\min}\in M_{\min}(F)$.

\item Assume moreover that $T$ is contained in $A_{M_{\min}}$. Then for any $\del>0$ there exists $\eps>0$ such that the integral
$$
I^{2}_{\eps, \del}
(m_{\min}) = 
\int_{H(F)\times H(F)}
\Xi^{G}(hm_{\min})
\Xi^{G}(h^{\p}hm_{\min})
e^{\eps \sig(h)}
e^{\eps \sig(h^{\p})}
(1+|\lam(h^{\p})|)^{-\del}\ud h^{\p}\ud h
$$
is absolutely convergent for any $m_{\min}\in M_{\min}(F)$, and there exists $d>0$ such that 
$$
I^{2}_{\eps, \del}(m_{\min})\ll
\del_{\wb{P}_{\min}}(m_{\min})^{-1}
\sig(m_{\min})^{d}
$$
for any $m_{\min}\in M_{\min}(F)$.
\end{enumerate}
\end{lem}

\begin{proof}
The proof of (1) follows from \cite[Lemme~4.9]{waldspurgertemperedggp}, whose proof works verbatim for archimedean case. 

The rest of the proof follows from \cite[Lemma~6.5.1]{beuzart2015local} verbatim.
\end{proof}

\subsubsection{The function $\Xi^{H\bs G}$}
Let $C\subset G(F)$ be a compact subset with nonempty interior. Define a function $\Xi^{H\bs G}_{C}$ on $H(F)\bs G(F)$ by
$$
\Xi^{H\bs G}_{C}(x) = \vol_{H\bs G}(xC)^{-1/2}
$$
for any $x\in H(F)\bs G(F)$. By finite cover theorem, for any other $C^{\p}\subset G(F)$ compact subset with nonempty interior,
$$
\Xi^{H\bs G}_{C}(x)\sim \Xi^{H\bs G}_{C^{\p}}(x)
$$
for any $x\in H(F)\bs G(F)$.
In the following we will implicitly fix a compact subset with nonempty interior $C\subset G(F)$ and set 
$$
\Xi^{H\bs G}(x) = \Xi^{H\bs G}_{C}(x)
$$
for any $x\in H(F)\bs G(F)$.

The following estimations for $\Xi^{H\bs G}$ hold.
\begin{pro}\label{pro:ggptriple:xiHbsGfun}
\begin{enumerate}
\item For any compact subset $\CK\subset G(F)$, the following equivalence hold
\begin{enumerate}
\item $\Xi^{H\bs G}(xk)\sim \Xi^{H\bs G}(x)$,

\item $\sig_{H\bs G}(xk)\sim \sig_{H\bs G}(x)$
\end{enumerate}
for any $x\in H(F)\bs G(F)$ and $k\in \CK$.

\item Let $\wb{P}_{0} = M_{0}\wb{U}_{0}\subset G_{0}$ be a good minimal parabolic subgroup of $G_{0}$ and $A_{0} = A_{M_{0}}$ be the split part of the center of $M_{0}$. Set 
$$
A^{+}_{0} = \{ a_{0}\in A_{0}(F)|\quad |\alp(a_{0})|\geq 1, \forall \alp\in R(A_{0},\wb{P}_{0}) \}.
$$
Then there exists a positive constant $d>0$ such that 
\begin{enumerate}
\item $\Xi^{G_{0}}(a_{0})\del_{P}^{1/2}(a)\sig(a_{9})^{-d}\ll \Xi^{H\bs G}(aa_{0}) \ll \Xi^{G_{0}}(a_{0})\del_{P}^{1/2}(a)$,

\item $\sig_{H\bs G}(aa_{0})\sim \sig_{G}(aa_{0})$ 
\end{enumerate}
for any $a_{0}\in A^{+}_{0}$ and $a\in T(F)$.

\item There exists $d>0$ such that 
$$
\int_{H(F)\bs G(F)}
\Xi^{H\bs G}(x)^{2}\sig_{H\bs G}(x)^{-d}\ud x
$$
is absolutely convergent.

\item For any $d>0$, there exists $d^{\p}>0$ such that 
$$
\int_{H(F)\bs G(F)}
1_{\sig_{H\bs G}\leq c}(x)\Xi^{H\bs G}(x)^{2}
\sig_{H\bs G}(x)^{d}\ud x\ll c^{d^{\p}}
$$
for any $c\geq 1$.

\item There exists $d>0$ and $d^{\p}>0$ such that 
$$
\int_{H(F)}
\Xi^{G}(x^{-1}hx)
\sig_{G}(x^{-1}hx)^{-d}\ud h
\ll
\Xi^{H\bs G}(x)^{2}\sig_{H\bs G}(x)^{d^{\p}}
$$
for any $x\in H(F)\bs G(F)$.

\item For any $d>0$, there exists $d^{\p}>0$ such that 
$$
\int_{H(F)}
\Xi^{G}(hx)\sig(hx)^{-d^{\p}}\ud h
\ll
\Xi^{H\bs G}(x)\sig_{H\bs G}(x)^{-d}
$$
for any $x\in H(F)\bs G(F)$.

\item Let $\del>0$ and $d>0$. Then the integral 
$$
I_{\del,d}(c,x) = 
\int_{H(F)}
\int_{H(F)}
1_{\sig \geq c}
(h^{\p})\Xi^{G}(hx)\Xi^{G}(h^{\p}hx)
\sig(hx)^{d}\sig(h^{\p}hx)^{d}
(1+|\lam(h^{\p})|)^{-\del}\ud h^{\p}\ud h
$$
is absolutely convergent for any $x\in H(F)\bs G(F)$ and $c\geq 1$. Moreover, there exist $\eps>0$ and $d^{\p}>0$ such that 
$$
I_{\del,d}(c,x)\ll \Xi^{H\bs G}(x)^{2}\sig_{H\bs G}(x)^{d^{\p}}e^{\eps c}
$$
for any $x\in H(F)\bs G(F)$ and $c\geq 1$.
\end{enumerate}
\end{pro}
\begin{proof}
The proof of \cite[Proposition~6.7.1]{beuzart2015local} works verbatim.
\end{proof}

\subsubsection{Parabolic degenerations}
Let $\wb{Q} = LU_{\wb{Q}}$ be a good parabolic subgroup of $G$. Let $\wb{P}_{\min} = M_{\min}\wb{U}_{\min}\subset \wb{Q}$ be a good minimal parabolic subgroup of $G$ with the Levi component chosen so that $M_{\min}\subset L$. Let $A_{\min} = A_{M_{\min}}$ be the maximal central split torus of $M_{\min}$ and set 
$$
A^{+}_{\min} = \{a\in A_{\min}(F) | \quad |\alp(a)|\geq 1, \forall \alp\in R(A_{\min}, \wb{P}_{\min}) \}.
$$
Let $H_{\wb{Q}} = H\cap \wb{Q}$ and $H_{L}$ be the image of $H_{\wb{Q}}$ by the natural projection $\wb{Q}\twoheadrightarrow L$. Define $H^{Q} = H_{L}\ltimes U_{Q}$.

\begin{pro}\label{pro:pardeg}
\begin{enumerate}
\item $H_{\wb{Q}}\cap U_{\wb{Q}} = \{ 1\}$ so that the natural projection $H_{\wb{Q}}\to H_{L}$ is an isomorphism.

\item $\del_{\wb{Q}}(h_{\wb{Q}}) = \del_{H_{\wb{Q}}}(h_{\wb{Q}})$ and $\del_{\wb{Q}}(h_{L}) = \del_{H_{L}}(h_{L})$ for any $h_{\wb{Q}}\in H_{\wb{Q}}(F)$ and $h_{L}\in H_{L}(F)$. In particular, the group $H^{Q}$ is unimodular.

Fix a left Haar measure $d_{L}h_{L}$ on $H_{L}(F)$ and a Haar measure $dh^{Q}$ on $H^{Q}(F)$.

\item There exists $d>0$ such that the integral 
$$
\int_{H_{L}(F)}
\Xi^{L}(h_{L})\sig(h_{L})^{-d}\del_{H_{L}}(h_{L})^{1/2}\ud_{L}h_{L}
$$
converges. Moreover, in the codimension one case (that is $G=G_{0}$ and $H=H_{0}$), the integral 
$$
\int_{H_{L}(F)}
\Xi^{L}(h_{L})\sig(h_{L})^{d}\del_{H_{L}}(h_{L})^{1/2}\ud_{L}h_{L}
$$
is convergent for any $d>0$.

\item There exists $d>0$ such that the integral 
$$
\int_{H^{Q}(F)}\Xi^{G}(h^{Q})\sig(h^{Q})^{-d}\ud h^{Q}
$$
converges.

\item $\sig(h^{Q})\ll \sig(a^{-1}h^{Q}a)$ for any $a\in A^{+}_{\min}$ and $h^{Q}\in H^{Q}(F)$.

\item There exist $d>0$ and $d^{\p}>0$ such that 
$$
\int_{H^{Q}(F)}
\Xi^{G}(a^{-1}h^{Q}a)
\sig(a^{-1}h^{Q}a)^{-d}\ud h^{Q}
\ll  \Xi_{H\bs G}(a)^{2}\sig_{H\bs G}(a)^{d^{\p}}
$$
for any $a\in A^{+}_{\min}$.
\end{enumerate}
\end{pro}
\begin{proof}
The proof of \cite[Proposition~6.8.1]{beuzart2015local} works verbatim.
\end{proof}

\section{Explicit tempered intertwining}
In this section an explicit relation between the non-vanishing of $m(\pi)$ and an associated linear form $\CL_{\pi}$ for any tempered representation $\pi\in \mathrm{Temp}(G)$ is stated. The relation between $m(\pi)$ and parabolic induction is also stated. As a corollary, we are able to establish Conjecture \ref{ggpconjec} for tempered $L$-parameters when $F=\BC$. Moreover, the results will be indispensable for the proof of the spectral expansion of the distribution $J$. Most of the proof appearing in \cite[Section~7]{beuzart2015local} works verbatim. The details are referred to \cite{thesis_zhilin}.

\subsubsection{The $\xi$-integral}\label{sec:temperedint:1}
For any $f\in \CC(G)$, the integral
$$
\int_{H(F)}f(h)\xi(h)\ud h
$$
is absolutely convergent by Lemma \ref{lem:ggptriple:sigest} (2). Moreover, by Lemma \ref{lem:ggptriple:sigest} (2) it defines a continuous linear form on $\CC(G)$. Recall that $\CC(G)$ is a dense subspace of the weak Harish-Chandra Schwartz space $\CC^{w}(G)$ (c.f. \cite[1.5.1]{beuzart2015local}).

\begin{pro}\label{pro:temperedint:continuous}
The linear form 
$$
f\to \CC(G) \to \int_{H(F)} f(h)\xi(h)\ud h
$$
extends continuously to $\CC^{w}(G)$.
\end{pro}
\begin{proof}
The proof of \cite[Proposition~7.1.1]{beuzart2015local} works verbatim.
\end{proof}

The continuous linear form on $\CC^{w}(G)$ proved above is called the \emph{$\xi$-integral} on $H(F)$ and will be denoted by
$$
f\in \CC^{w}(G) \to \int^{*}_{H(F)} f(h)\xi(h)\ud h
$$
or 
$$
f\in \CC^{w}(G)\to \CP_{H,\xi}(f).
$$

The following properties of the $\xi$-integral hold.

\begin{lem}\label{lem:eq:temperedint:continuous:2}
\begin{enumerate}
\item For any $f\in \CC^{w}(G)$ and $h_{0},h_{1}\in H(F),$ 
$$
\CP_{H,\xi}(L(h_{0})R(h_{1})f) = \xi(h_{0}h_{1}^{-1})\CP_{H,\xi}(f).
$$

\item Let $a:\BG_{m}\to T$ be a one-parameter subgroup such that $\lam(a(t)ha(t)^{-1}) = t\lam(h)$ for any $t\in \BG_{m}$ and $h\in H$. Denote by $\Ad_{a}$ the representation of $F^{\times}$ on $\CC^{w}(G)$ given by $\Ad_{a}(t) = L(a(t))R(a(t))$ for any $t\in F^{\times}$. Let $\vphi\in C^{\infty}_{c}(F^{\times})$. Set $\vphi^{\p}(t) = |t|^{-1}\del_{P}(a(t))\vphi(t)$ for any $t\in F^{\times}$ and denote by $\wh{\vphi^{\p}}$ its Fourier transform, that is
$$
\wh{\vphi^{\p}}(x) = \int_{F}\vphi^{\p}(t)\psi(tx)\ud t,\quad x\in F.
$$
Then
$$
\CP_{H,\xi}(\Ad_{a}(\vphi)f) = \int_{H(F)}f(h)\wh{\vphi^{\p}}(\lam(h))\ud h
$$
for any $f\in \CC^{w}(G)$, where the second integral is absolutely convergent.
\end{enumerate}
\end{lem}
\begin{proof}
Both statements are continuous in $f\in \CC^{w}(G)$ (for (2) it follows from Lemma \ref{lem:ggptriple:sigest} (3)). Hence we only need to verify it for $f\in \CC(G)$ where it follows from direct computation.
\end{proof}

\subsubsection{Definition of $\CL_{\pi}$}\label{sec:temperedint:2}
Let $\pi$ be a tempered representation of $G(F)$. For any $T\in \End(\pi)^{\infty}$, the function
$$
g\in G(F)\to \mathrm{Trace}(\pi(g^{-1})T)
$$
belongs to $\CC^{w}(G)$ by \cite[2.2.4]{beuzart2015local}.  Hence one may define a linear form $\CL_{\pi}:\End(\pi)^{\infty}\to \BC$ by 
$$
\CL_{\pi}(T) =
\int^{*}_{H(F)}\mathrm{Trace}(\pi(h^{-1})T)\xi(h)\ud h, \quad T\in \End(\pi)^{\infty}.
$$
By Lemma \ref{lem:eq:temperedint:continuous:2} (1),
$$
\CL_{\pi}(\pi(h)T \pi(h^{\p})) = \xi(h)\xi(h^{\p})\CL_{\pi}(T)\CL_{\pi}(T)
$$
for any $h,h^{\p}\in H(F)$ and $T\in \End(\pi)^{\infty}$. By \cite[2.2.5]{beuzart2015local}, the map which associates to $T\in \End(\pi)^{\infty}$ the function
$$
g\to \mathrm{Trace}(\pi(g^{-1})T)
$$
in $\CC^{w}(G)$ is continuous. Since the $\xi$-integral is a continuous linear form on $\CC^{w}(G)$, it follows that $\CL_{\pi}$ is continuous.

Recall that there exists a canonical continuous $G(F)\times G(F)$-equivariant embedding with dense image $\pi^{\infty}\otimes \wb{\pi^{\infty}}\hookrightarrow \End(\pi)^{\infty}$, $e\otimes e^{\p}\to T_{e,e^{\p}}$ (which is an isomorphism in the $p$-adic case). In any case, $\End^{\infty}(\pi)^{\infty}$ is naturally isomorphic to the completed projective tensor product $\pi^{\infty}\wh{\otimes}_{p}\wb{\pi^{\infty}}$. Thus $\CL_{\pi}$ can be identified with the continuous sesquilinear form on $\pi^{\infty}$ given by
$$
\CL_{\pi}(e,e^{\p}) : = \CL_{\pi}(T_{e,e^{\p}})
$$
for any $e,e^{\p}\in \pi^{\infty}$. Expanding the definitions,
$$
\CL_{\pi}(e,e^{\p}) = \int_{H(F)}^{*}
(e,\pi(h)e^{\p})\xi(h)\ud h
$$
for any $e,e^{\p}\in \pi^{\infty}.$ Fixing $e^{\p}\in \pi^{\infty}$, the map $e\in \pi^{\infty}\to \CL(e,e^{\p})$ then lies in $\Hom_{H}(\pi^{\infty},\xi)$. By the density, it follows that
$$
\CL_{\pi}\neq 0 \Rightarrow m(\pi)\neq 0.
$$

The goal is to show the converse.

\begin{thm}\label{thm:temperedint:main}

$$
\CL_{\pi}\neq 0 \Leftrightarrow m(\pi)\neq 0
$$
for any $\pi\in \mathrm{Temp}(G)$.
\end{thm}
When $F$ is non-archimedean, the result has been established in \cite[Proposition~5.7]{waldspurgertemperedggp}. To establish the theorem, a key result is Proposition \ref{pro:eq:temperedint:16}, which is going to be explained in the following sections.

Finally some basic properties of $\CL_{\pi}$ are listed. Since $\CL_{\pi}$ is a continuous sesqulinear form on $\pi^{\infty}$, it defines a continuous linear map 
$$
L_{\pi}:\pi^{\infty}\to \wb{\pi^{-\infty}}, e\to \CL_{\pi}(e,\cdot)
$$
where $\wb{\pi^{-\infty}}$ is the topological conjugate-dual of $\pi^{\infty}$ endowed with the strong topology. The operator $L_{\pi}$ has its image contained in $\wb{\pi^{-\infty}}^{H,\xi} = \Hom_{H}(\wb{\pi^{\infty}},\xi)$. By Theorem \ref{thm:ggptriples:multiplictyone}, this subspace has finite dimension, and is less than or equal to $1$ if $\pi$ is irreducible. Let $T\in \End(\pi)^{\infty}$. Recall that it extends uniquely to a continuous operator $T: \wb{\pi^{-\infty}}\to \pi^{\infty}$. Thus, the following two compositions can be formed
$$
TL_{\pi}:\pi^{\infty}\to \pi^{\infty}, \quad L_{\pi}T:\wb{\pi^{-\infty}}\to \wb{\pi^{-\infty}}
$$
which are both finite-rank operators. In particular, their traces are well-defined and
\begin{equation}\label{eq:temperedint:3}
\mathrm{Trace}(TL_{\pi}) = \mathrm{Trace}(L_{\pi}T) = \CL_{\pi}(T).
\end{equation}

\begin{lem}\label{lem:temperedint:traceform}
The following properties hold.
\begin{enumerate}
\item The maps
\begin{align*}
\pi\in \CX_{\mathrm{temp}}(G)	&\to L_{\pi}\in \Hom(\pi^{\infty}, \wb{\pi^{-\infty}})
\\
\pi\in \CX_{\mathrm{temp}}(G)	&\to \CL_{\pi}\in \End(\pi)^{-\infty}
\end{align*}
are smooth in the following sense: For every parabolic subgroup $Q=LU_{Q}$ of $G$, for any $\sig\in \Pi_{2}(L)$ and for every maximal compact subgroup $K$ of $G(F)$ that is special in the $p$-adic case, the maps
\begin{align*}
&\lam\in i\CA^{*}_{L} \to \CL_{\pi_{\lam}} \in \End(\pi_{\lam})^{-\infty} \simeq \End(\pi_{K})^{-\infty}
\\
&\lam\in i\CA^{*}_{L} \to L_{\pi_{\lam}}\in \Hom(\pi^{\infty}_{\lam}, \wb{\pi^{-\infty}_{\lam}})
\simeq \Hom(\pi^{\infty}_{K}, \wb{\pi^{-\infty}_{K}})
\end{align*}
are smooth, where $\pi_{\lam} = i^{G}_{Q}(\sig_{\lam})$ and $\pi_{K} = i^{K}_{Q\cap K}(\sig)$.

\item Let $\pi\in \mathrm{Temp}(G)$ or $\CX_{\mathrm{temp}}(G)$, Then for any $S,T\in \End(\pi)^{\infty}$, $SL_{\pi}\in \End(\pi)^{\infty}$ and 
$$
\CL_{\pi}(S)\CL_{\pi}(T) = \CL_{\pi}(SL_{\pi}T).
$$

\item Let $S,T\in \CC(\CX_{\mathrm{temp}}(G), \CE(G))$. Then, the section $\pi\in \mathrm{Temp}(G)\to S_{\pi}L_{\pi}T_{\pi}\in \End(\pi)^{\infty}$ belongs to $\CC^{\infty}(\CX_{\mathrm{temp}}(G),\CE(G))$.

\item Let $f\in \CC(G)$ and assume that its Plancherel transform $\pi\in \CX_{\mathrm{temp}}(G)\to \pi(f)$ is compactly supported (which is automatic when $F$ is $p$-adic). Then the following identity holds,
$$
\int_{H(F)}f(h)\xi(h)\ud h = \int_{\CX_{\mathrm{temp}}(G)}\CL_{\pi}(\pi(f))\mu(\pi)\ud\pi
$$
where both integrals are absolutely convergent.

\item Let $f,f^{\p}\in \CC(G)$ and assume that the Plancherel transform of $f$ is compactly supported, then the following identity holds
$$
\int_{\CX_{\mathrm{temp}}(G)}
\CL_{\pi}(\pi(f))\wb{\CL_{\pi}(\pi(\wb{f^{\p}}))}\mu(\pi)\ud\pi
=\int_{H(F)}\int_{H(F)}\int_{G(F)}
f(hgh^{\p})f^{\p}(g)\ud g
\xi(h^{\p})dh^{\p}\xi(h)\ud h
$$
where the first integral is absolutely convergent and the second integral is convergent in that order but not necessarily as a triple integral.
\end{enumerate}
\end{lem}
\begin{proof}
The proof of \cite[Lemma~7.2.2]{beuzart2015local} works verbatim.
\end{proof}

\subsubsection{Asymptotic of tempered intertwinings}\label{sec:temperedint:3}

\begin{lem}\label{lem:temperedint:asymptoticoftempered}
\begin{enumerate}
\item Let $\pi$ be a tempered representation of $G(F)$ and $\ell\in \Hom_{H}(\pi^{\infty},\xi)$ be a continuous $(H,\xi)$-equivariant linear form. Then, there exist $d>0$ and a continuous semi-norm $v_{d}$ on $\pi^{\infty}$ such that 
$$
|\ell(\pi(x)e)| \leq v_{d}(e)\Xi^{H\bs G}(x)\sig_{H\bs G}(x)^{d}
$$
for any $e\in \pi^{\infty}$ and $x\in H(F)\bs G(F)$.

\item For any $d>0$, there exists $d^{\p}>0$ and a continuous semi-norm $v_{d,d^{\p}}$ on $\CC^{w}_{d}(G(F))$ such that 
$$
|\CP_{H,\xi}(R(x)L(y)\vphi)|
\leq 
v_{d,d^{\p}}(\vphi)
\Xi^{H\bs G}(x)
\Xi^{H\bs G}(y)
\sig_{H\bs G}(x)^{d^{\p}}
\sig_{H\bs G}(y)^{d^{\p}}
$$
for any $\vphi\in \CC^{w}_{d}(G(F))$ and $x,y\in H(F)\bs G(F)$.
\end{enumerate}
\end{lem}
\begin{proof}
The proof follows from the same argument as \cite[Lemma~7.3.1]{beuzart2015local}.
\end{proof}

\subsubsection{Explicit intertwining and parabolic induction}\label{sec:temperedint:4}

Let $Q=LU_{Q}$ be a parabolic subgroup of $G = \SO(W)\times \SO(V)$. Then there are decompositions
\begin{equation} \label{eq:temperedint:16}
Q=Q_{W}\times Q_{V}, \quad L=L_{W}\times L_{V},
\end{equation}
where $Q_{W}$ and $Q_{V}$ are parabolic subgroups of $\SO(W)$ and $\SO(V)$ respectively and $L_{W}$ and $L_{V}$ are the associated Levi components. By the explicit descriptions of parabolic subgroups of special orthogonal groups,
\begin{equation}\label{eq:temperedint:17}
L_{W} = \GL(Z_{1,W})\times ... \times \GL(Z_{a,w})\times \SO(\wt{W}),
\end{equation}
\begin{equation}\label{eq:temperedint:18}
L_{V} = \GL(Z_{1,V})\times ... \times \GL(Z_{b,V})\times \SO(\wt{V}),
\end{equation}
where $Z_{i,W}$, $1\leq i\leq a$ (respectively $Z_{i,V}$, $1\leq i\leq b$) are totally isotropic subspaces of $W$ (respectively of $V$) and $\wt{W}$ (respectively $\wt{V}$) is a non-degenerate subspace of $W$ (respectively of $V$). Let $\wt{G}=  \SO(\wt{W})\times \SO(\wt{V})$. Up to permutation the pair $(\wt{V},\wt{W})$ is admissible, hence in particular it yields a GGP triple $(\wt{G},\wt{H},\wt{\xi})$ which is well-defined up to $\wt{G}(F)$-conjugation. For any tempered representations $\wt{\sig}$ of $\wt{G}(F)$, one can consider the continuous linear form $\CL_{\wt{\sig}}:\End(\wt{\sig})^{\infty}\to \BC$.

Let $\sig$ be a tempered representation of $L(F)$ which decomposes according to (\ref{eq:temperedint:16}), (\ref{eq:temperedint:17}), (\ref{eq:temperedint:18}) as a tensor product 
\begin{equation}\label{eq:temperedint:19}
\sig =\sig_{W}\boxtimes \sig_{V},
\end{equation}
\begin{equation}\label{eq:temperedint:20}
\sig_{W} = \sig_{1,W}\boxtimes...\boxtimes \sig_{a,W}\boxtimes \wt{\sig}_{W},
\end{equation}
\begin{equation}\label{eq:temperedint:21}
\sig_{V} = \sig_{1,V}\boxtimes...\boxtimes\sig_{b,V}\boxtimes \wt{\sig}_{V},
\end{equation}
where $\sig_{i,W}\in \mathrm{Temp}(\GL(Z_{i,W}))$ for $1\leq i\leq a$, $\sig_{i,V}\in \mathrm{Temp}(\GL(Z_{i,V}))$ for $1\leq i\leq b$, $\wt{\sig}_{W}$ is a tempered representation of $\SO(\wt{W})(F)$ and $\wt{\sig}_{V}$ is a tempered representation of $\SO(\wt{V})(F)$. Set $\wt{\sig} = \wt{\sig}_{W}\boxtimes \wt{\sig}_{V}$. It is a tempered representation of $\wt{G}(F)$. Finally set $\pi  =i^{G}_{Q}(\sig)$, $\pi_{W} = i^{\SO(W)}_{Q_{W}}(\sig_{W})$, and $\pi_{V} = i^{\SO(V)}_{Q_{V}}(\sig_{V})$ for the parabolic inductions of $\sig, \sig_{W}$ and $\sig_{V}$ respectively. Then $\pi = \pi_{W}\boxtimes \pi_{V}$.

\begin{pro}\label{pro:eq:temperedint:16}
With the notations as above,
$$
\CL_{\pi}\neq 0 \Leftrightarrow \CL_{\wt{\sig}}\neq 0.
$$
\end{pro}
\begin{proof}
The proof in \cite[Proposition~7.4.1]{beuzart2015local} works verbatim.
\end{proof}

Now the proof of Theorem \ref{thm:temperedint:main} follows from the same argument as in \cite[Section~7]{beuzart2015local}. The details are referred to \cite{thesis_zhilin}.

\subsubsection{A corollary}\label{sec:temperedint:6}
We adopt the notation and hypothesis of Section \ref{sec:temperedint:4}. In particular, $Q=LU_{Q}$ is a parabolic subgroup of $G$ with $L$ decomposes as in (\ref{eq:temperedint:16}), (\ref{eq:temperedint:17}) and (\ref{eq:temperedint:18}), $\sig$ is a tempered representation of $L(F)$ admitting decomposition as in (\ref{eq:temperedint:19}), (\ref{eq:temperedint:20}) and (\ref{eq:temperedint:21}) and set $\wt{\sig} =\wt{\sig}_{W}\boxtimes \wt{\sig}_{V}$. This is a tempered representation of $\wt{G}(F)$ where $\wt{G}(F) = \SO(\wt{W})\times \SO(\wt{V})$. Recall that the admissible pair $(\wt{W},\wt{V})$ (up to permutation) defines a GGP triple $(\wt{G}, \wt{H}, \wt{\xi})$. Hence, the multiplicity $m(\wt{\sig})$ of $\wt{\sig}$ relative to this GGP triple is defined. Set $\pi = i^{G}_{Q}(\sig)$.

\begin{cor}\label{cor:3.6.1}
\begin{enumerate}
\item Assume that $\sig$ is irreducible,
$$
m(\pi) = m(\wt{\sig}).
$$

\item Let $\CK\subset \CX_{\mathrm{temp}}(G)$ be a compact subset. There exists a section $T\in \CC(\CX_{\mathrm{temp}}(G),\CE(G))$ such that 
$$
\CL_{\pi}(T_{\pi}) = m(\pi)
$$
for any $\pi\in \CK$. Moreover, the same equality is satisfied for every sub-representation $\pi$ of some $\pi^{\p}\in \CK$.
\end{enumerate}
\end{cor}
\begin{proof}
The proof of \cite[Corollary~7.6.1]{beuzart2015local} works verbatim.
\end{proof}

When $F= \BC$, $\wt{W}$ and $\wt{V}$ can be taken to be zero dimension since tempered representations are always principal series. It follows that the tempered part of Conjecture \ref{ggpconjec} holds automatically when $F=\BC$.

\section{The distributions}\label{sec:distribution}

In this section, the trace distribution $J$ and its Lie algebra variant $J^{\Lie}$ is introduced.

\subsubsection{The distribution $J$}
For any $f\in \CC(G)$, define a function $K(f,\cdot)$ on $H(F)\bs G(F)$ by
$$
K(f,x) = \int_{H(F)}
f(x^{-1}hx)\xi(h)\ud h, \quad x\in H(F)\bs G(F).
$$
From Lemma \ref{lem:ggptriple:hcxifunction} (2) the integral is absolutely convergent. The theorem below and Proposition \ref{pro:ggptriple:xiHbsGfun} (3) shows that the following integral
$$
J(f) = \int_{H(F)\bs G(F)}K(f,x)\ud x
$$
is absolutely convergent for any $f\in \CC_{\scusp}(G)$, and $J(\cdot)$ defines a continuous linear form on $\CC_{\scusp}(G)$.

\begin{thm}\label{thm:covofJ}
\begin{enumerate}
\item There exists $d>0$ and a continuous semi-norm $\nu$ on $\CC(G)$ such that 
$$
|K(f,x)|
\leq \nu(f)
\Xi^{H\bs G}(x)^{2}\sig_{H\bs G}(x)^{d}
$$
for any $f\in \CC(G)$ and $x\in H(F)\bs G(F)$.

\item For any $d>0$ there exists a continuous semi-norm $\nu_{d}$ on $\CC(G)$ such that 
$$
|K(f,x)|\leq \nu_{d}(f)\Xi^{H\bs G}(x)^{2}\sig_{H\bs G}(x)^{-d}
$$
for any $x\in H(F)\bs G(F)$ and $f\in \CC_{\scusp}(G)$.
\end{enumerate}
\end{thm}
\begin{proof}
The proof in \cite[Theorem~8.1.1]{beuzart2015local} works verbatim.
\end{proof}

\subsubsection{The distribution $J^{\Lie}$}
Parallel to $J$ the following Lie algebra variant distribution $J^{\Lie}$ is introduced. 

For any $f\in \CS(\Fg)$, define a function $K^{\Lie}(f,\cdot)$ on $H(F)\bs G(F)$ by
$$
K^{\Lie}(f,x) = \int_{\Fh(F)}
f(x^{-1}Xx)\xi(X)\ud X, \quad x\in H(F)\bs G(F).
$$
The integral is absolutely convergent. The theorem below, whose proof is similar to Theorem \ref{thm:covofJ}, shows that the following integral 
$$
J^{\Lie}(f) = \int_{H(F)\bs G(F)}
K^{\Lie}(f,x)\ud x
$$
is absolutely convergent for any $f\in \CS_{\scusp}(\Fg)$ and defines a continuous linear form on $\CS_{\scusp}(\Fg)$.

\begin{thm}\label{thm:abcovJLie}
\begin{enumerate}
\item There exists $c>0$ and a continuous semi-norm $\nu$ on $\CS(\Fg)$ such that 
$$
|K^{\Lie}(f,x)|
\leq \nu(f)e^{c\sig_{H\bs G}(x)}
$$
for any $x\in H(F)\bs G(F)$ and $f\in \CS(\Fg)$.

\item For any $c>0$, there exists a continuous semi-norm $\nu_{c}$ on $\CS(\Fg)$ such that 
$$
|K^{\Lie}(f,x)|
\leq \nu_{c}(f)e^{-c\sig_{H\bs G}(x)}
$$
for any $x\in H(F)\bs G(F)$ and $f\in \CS_{\scusp}(\Fg)$.
\end{enumerate}
\end{thm}

\section{Spectral expansion for the distribution $J$}\label{sec:spectralexp}

In this section the spectral expansion of the distribution $J$ is established. The proof of the spectral side of the distribution $J$ for unitary groups case (\cite{beuzart2015local}) carries without hard effort to the special orthogonal groups setting. To save the length of the paper, we outline the main analytical results and refer the details to \cite{thesis_zhilin}.

We first give the statement of the theorem.
\begin{thm}\label{thm:spectralexpansionJ}
For any $f\in \CC_{\scusp}(G)$, set
$$
J_{\mathrm{spec}}(f) = 
\int_{\CX(G)}
D(\pi)\wh{\theta}_{f}(\pi)m(\wb{\pi})\ud\pi.
$$
Then the integral is absolutely convergent and
$$
J(f) = J_{\mathrm{spec}}(f)
$$
for any $f\in \CC_{\scusp}(G)$.
\end{thm}

From \cite[Lemma~5.4.2]{beuzart2015local} and Theorem \ref{thm:covofJ}, both sides of the equality are continuous on $\CC_{\scusp}(G)$. Hence by \cite[Lemma~5.3.1 (ii)]{beuzart2015local} it is sufficient to establish the equality for functions $f\in \CC_{\scusp}(G)$ which have compactly supported Fourier transforms (where the Fourier transform is understood as the spectral transform appearing in matrix Paley-Wiener theorem, see \cite[p.121]{beuzart2015local}). Throughout the section a function $f\in \CC_{\scusp}(G)$ with compactly supported Fourier transform is fixed.

\subsubsection{Study of an auxiliary distribution}\label{sec:5.2}
For any $f^{\p}\in \CC(G),$ the following integrals are introduced,
\begin{align*}
&K^{A}_{f,f{\p}}(g_1,g_2) = 
\int_{G(F)}f(g^{-1}_{1}gg_{2})f^{\p}(g)\ud g,\quad g_{1},g_{2}\in G(F),
\\
&K^{1}_{f,f^{\p}}(g,x) = 
\int_{H(F)}
K^{A}_{f,f^{\p}}(g,hx)\xi(h)\ud h, \quad g,x\in G(F),
\\
&K^{2}_{f,f^{\p}}(x,y) = \int_{H(F)}K^{1}_{f,f^{\p}}(h^{-1}x,y)\xi(h)\ud h, \quad x,y\in G(F),
\\
&J_{\mathrm{aux}}(f,f^{\p}) = \int_{H(F)\bs G(F)}
K^{2}_{f,f^{\p}}(x,x)\ud x.
\end{align*}

The following proposition gives estimations for the auxiliary distributions.
\begin{pro}\label{pro:5.2.1}
\begin{enumerate}
\item The integral defining $K^{A}_{f,f^{\p}}(g_{1},g_{2})$ is absolutely convergent. For any $g_{1}\in G(F)$ the map 
$$
g_{2}\in G(F) \to K^{A}_{f,f^{\p}}(g_{1},g_{2})
$$
belongs to $\CC(G)$. Moreover, for any $d>0$ there exists $d^{\p}>0$ such that for every continuous semi-norm $\nu$ on $\CC^{w}_{d^{\p}}(G(F))$, there exists a continuous semi-norm $\mu$ on $\CC(G)$ satisfying
$$
\nu(K^{A}_{f,f^{\p}}(g,\cdot)) \leq \nu(f^{\p})\Xi^{G}(g)\sig(g)^{-d}
$$
for any $f^{\p}\in \CC(G)$ and $g\in G(F)$.

\item The integral defining $K^{1}_{f,f^{\p}}(g,x)$ is absolutely convergent. Moreover, for any $d>0$, there exists $d^{\p}>0$ and a continuous semi-norm $\nu_{d,d^{\p}}$ on $\CC(G)$ such that 
$$
|K^{1}_{f,f^{\p}}(g,x)|
\leq \nu_{d,d^{\p}}(f^{\p})
\Xi^{G}(g)
\sig(g)^{-d}
\Xi^{H\bs G}(x)
\sig_{H\bs G}(x)^{d^{\p}}
$$
for any $f^{\p}\in \CC(G)$ and $g,x\in G(F)$.

\item The integral defining $K^{2}_{f,f^{\p}}(x,y)$ is absolutely convergent convergent. Moreover
$$
K^{2}_{f,f^{\p}}
(x,y) = 
\int_{\CX_{\mathrm{temp}}(G)}
\CL_{\pi}(\pi(x)\pi(f)\pi(y^{-1}))
\wb{\CL_{\pi}(\pi(\wb{f^{\p}}))}
\mu(\pi)\ud\pi
$$
for any $f^{\p}\in \CC(G)$ and $x,y\in G(F)$. The integral is absolutely convergent.

\item The integral defining $J_{\mathrm{aux}}(f,f^{\p})$ is absolutely convergent. More precisely, for every $d>0$ there exists a continuous semi-norm $\nu_{d}$ on $\CC(G)$ such that
$$
|K^{2}_{f,f^{\p}}(x,x)|
\leq 
\nu_{d}(f^{\p})
\Xi^{H\bs G}(x)^{2}
\sig_{H\bs G}(x)^{-d}
$$
for any $f^{\p}\in \CC(G)$ and $x\in H(F)\bs G(F)$. In particular, the linear form
$$
f^{\p}\in \CC(G)\to J_{\mathrm{aux}}(f,f^{\p})
$$
is continuous.
\end{enumerate}
\end{pro}
\begin{proof}
The proof of \cite[Proposition~9.2.1]{beuzart2015local} works verbatim.
\end{proof}

From Proposition \ref{pro:5.2.1}, the following proposition can be established. Note that the upshot of the two propositions is various analytical estimations, which builds upon the analytical estimations proved in Subsection \ref{subsec:analyticalestimate}.
\begin{pro}\label{pro:5.2.2}
The following equality holds
$$
J_{\mathrm{aux}}(f,f^{\p}) = 
\int_{\CX(G)}
D(\pi)
\wh{\theta}_{f}(\pi)
\wb{\CL_{\pi}(\pi(\wb{f^{\p}}))}\ud\pi
$$
for any $f^{\p}\in \CC(G)$.
\end{pro}
\begin{proof}
The proof of \cite[Proposition~9.2.2]{beuzart2015local} works verbatim.
\end{proof}

Now we end of proof of the theorem \ref{thm:spectralexpansionJ}.

Recall that a function $f\in \CC_{\scusp}(G)$ with compactly supported Fourier transform has been fixed. By Lemma \ref{lem:temperedint:traceform} (4),
\begin{equation}\label{5.3.1}
K(f,x) = 
\int_{\CX_{\mathrm{temp}}(G)}
\CL_{\pi}(\pi(x)\pi(f)\pi(x^{-1}))\mu(\pi)\ud\pi
\end{equation}
for any $x\in H(F)\bs G(F)$. By Corollary \ref{cor:3.6.1} (2) there exists a function $f^{\p}\in \CC(G)$ such that 
\begin{equation}\label{5.3.2}
\CL_{\pi}(\pi(\wb{f^{\p}})) = m(\pi)
\end{equation}
for any $\pi\in \CX_{\mathrm{temp}}(G)$ such that $\pi(f)\neq 0$. Also, by Theorem \ref{thm:temperedint:main} and Corollary \ref{cor:3.6.1} (1), for any $\pi\in \CX_{\mathrm{temp}}(G)$,
$$
\CL_{\pi}\neq 0 \Leftrightarrow m(\pi) = 1.
$$
Hence by (\ref{5.3.1}) the equality holds
$$
K(f,x) = 
\int_{\CX_{\mathrm{temp}}(G)}
\CL_{\pi}(\pi(x)\pi(f)\pi(x^{-1}))
\wb{\CL_{\pi}(\pi(\wb{f^{\p}}))}
\mu(\pi)\ud\pi
$$
and by Proposition \ref{pro:5.2.1} (3), it follows that
$$
K(f,x) = K^{2}_{f,f^{\p}}(x,x)
$$
for any $x\in H(F)\bs G(F)$. Consequently, 
$$
J(f)= J_{\mathrm{aux}}(f,f^{\p}).
$$
After applying Proposition \ref{pro:5.2.2}, 
$$
J(f) = \int_{\CX(G)}
D(\pi)\wh{\theta}_{f}(\pi)
\wb{\CL_{\pi}(\pi(\wb{f^{\p}}))}\ud\pi.
$$
Let $\pi\in \CX(G)$ be such that $\wh{\theta}_{f}(\pi)\neq 0$ and let $\pi^{\p}$ be the unique representation in $\CX_{\mathrm{temp}}(G)$ such that $\pi$ is a linear combination of sub-representations of $\pi^{\p}$. Then $\pi^{\p}(f) \neq 0$. Hence by \ref{5.3.2} and Corollary \ref{cor:3.6.1} (2) $\wb{\CL_{\pi}(\pi(\wb{f^{\p}})) = \wb{m(\pi)}} = m(\wb{\pi})$. It follows that 
$$
J(f) = 
\int_{\CX(G)}
D(\pi)\wh{\theta}_{f}(\pi)m(\wb{\pi})\ud\pi
$$
and this ends the proof of Theorem \ref{thm:spectralexpansionJ}.

\section{Spectral expansion for the distribution $J^{\Lie}$}
In this section a spectral expansion of $J^{\Lie}$ is established, which is a key step towards the geometric expansion of $J$.

\subsection{The affine subspace $\Sig$}\label{sec:6.1}

In Section \ref{sub:ggptriples:definitions} a parabolic subgroup $P=MN$ of $G=\SO(W)\times \SO(V)$ has been defined, which can be written as $\SO(W)\times P_{V}$, with $P_{V}$ a parabolic subgroup of $\SO(V).$ Let $\wb{P} = M\wb{N}$ be the parabolic subgroup opposite to $P$ w.r.t. $M$. The unipotent radicals $N$ and $\wb{N}$ can be identified as subgroups of $\SO(V)$.

In Section \ref{sub:ggptriples:definitions} a character $\xi$ on $N(F)$ has been defined which extends to a character of $H(F) = \SO(W)(F)\ltimes N(F)$ trivial on $\SO(W)(F)$. Using the $G(F)$-invariant bilinear pairing $B$ on $\Fg(F)$ defined in Section \ref{sub:ggptriples:definitions}, there exists a unique element $\Xi\in \wb{\Fn}(F)$ such that
$$
\xi(X) = \psi(B(\Xi,X))
$$
for any $X\in \Fn(F)$.

There is the following explicit description of $\Xi\in \Fs\Fo(V)$,
\begin{equation}\label{6.1.1}
\Xi z_{i} = z_{i-1}, 1\leq i\leq r, 
\Xi z_{-i} = -z_{-i-1}, 1\leq i\leq r-1,
\Xi z_{0} = -\nu^{-1}_{0} z_{-1}, \Xi(W) = 0.
\end{equation}
Set $\Sig = \Xi+\Fh^{\perp}$ where $\Fh^{\perp}$ is the orthogonal complement of $\Fh$ in $\Fg$ w.r.t. $B(\cdot,\cdot)$.

Fix a Haar measure $\ud\mu_{\Fh}(X)$ on $\Fh(F)$. From \cite[Section~1.6]{beuzart2015local}, there is a natural way to associate to $\ud\mu_{\Fh}(X)$, using $B(\cdot,\cdot)$, a Haar measure $\ud\mu^{\perp}_{\Fh}$ on $\Fh^{\perp}(F)$. Denote by $\ud\mu_{\Sig}$ the translation of the measure to $\Sig(F)$. By Fourier inversion, the following equality holds
\begin{equation}\label{6.1.2}
\int_{\Fh(F)}
f(X)\xi(X)\ud\mu_{\Fh}(X) = 
\int_{\Sig(F)}
\wh{f}(Y)\ud\mu_{\Sig}(Y)
\end{equation}
for any $f\in \CS(\Fg)$.

\subsubsection{Conjugation by $N$}\label{sec:6.2}
By explicit calculation, there is the following description of $\Fh^{\perp}$: an element $X = (X_{W},X_{V})\in \Fg=  \Fs\Fo(W)\oplus \Fs\Fo(V)$ is in $\Fh^{\perp}$ if and only if
$$
X_{V} = -X_{W}+c(z_{0},w)+T+N
$$
for some $w\in W_{\wb{F}}$, $T\in \Ft$ and $N\in \Fn$. Thus for every element  $X = (X_{V},X_{W})$ of $\Sig$,
\begin{equation}\label{6.2.1}
X_{V} = \Xi-X_{W}+c(z_{0},w)+T+N
\end{equation}
where $w,T, N$ are as above. Define the following affine subspaces of $\Fg$:
\begin{itemize}
\item $\Fs\Fo(W)^{-} = \{ (X_{W},-X_{W}) |\quad X_{W}\in \Fs\Fo(W) \}$;

\item $\Lam_{0}$ is the subspace of $\Fs\Fo(V)\subset \Fg$ generated by $c(z_{i},z_{i+1})$ for $i=0,...,r-1$ and $c(z_{r},w)$ for $w\in W$;

\item $\Lam = \Xi+(\Fs\Fo(W)^{-}\oplus \Lam_{0})$.
\end{itemize}

\begin{pro}\label{pro:6.2.1}
Conjugation by $N$ preserves $\Sig$, and induces an isomorphism of algebraic varieties:
\begin{align*}
N\times \Lam &\to \Sig
\\
(n,X)&\to nXn^{-1}.
\end{align*}
\end{pro}
\begin{proof}
First we show that the map 
\begin{align}\label{6.2.2}
N\times \Lam &\to \Sig
\\
(n,x) &\to nXn^{-1}	\nonumber
\end{align}
is injective. This amounts to proving that for any $n\in N$ and $X\in \Lam$, if $nXn^{-1}\in \Lam$, then $n=1$. We let $n\in N$ and $X = (X_{W},X_{V})\in \Lam$ be such that $nXn^{-1}\in \Lam$. By the definition of $\Lam$, we may write $X_{V}$ and $nX_{V}n^{-1}$ as
\begin{equation}\label{6.2.3}
X_{V} = \Xi-X_{W} +c(z_{r},w)+\sum_{i=0}^{r-1}\mu_{i}c(z_{i},z_{i+1})
\end{equation}
\begin{equation}\label{6.2.4}
nX_{V}n^{-1} = \Xi - X_{W} +c(z_{r},w^{\p})+\sum_{i=0}^{r-1}
\mu^{\p}_{i}c(z_{i},z_{i+1})
\end{equation}
where $w,w^{\p}\in W_{\wb{F}}$ and $\mu_{i}, \mu^{\p}_{i}\in \wb{F}$, $0\leq i\leq r-1$. We first prove the following statement,
\begin{equation}\label{6.2.5}
nz_{i} = z_{i}, 0\leq i\leq r.
\end{equation}
The proof is by descending induction. The result is trivial for $i=r$ since $n\in N$. Assume now that the equality (\ref{6.2.5}) is true for some $1\leq i\leq r$. Then from (\ref{6.2.3}) we have 
$$
(nX_{V}n^{-1})z_{i} = nX_{V}z_{i} = n\Xi z_{i} = nz_{i-1}
$$
and from (\ref{6.2.4}), we also have 
$$
(nX_{V}n^{-1})z_{i} = \Xi z_{i} = z_{i-1},
$$
By induction hypothesis we get (\ref{6.2.5}).

We now prove the following 
\begin{equation}\label{6.2.6}
nz_{-i} = z_{-i}, \quad 1\leq i\leq r.
\end{equation}
We prove by strong induction on $i$. First we treat the case $i=1$. By (\ref{6.2.3}) and (\ref{6.2.5}) we have 
$$
(nX_{V}n^{-1})z_{0} = 
nX_{V}z_{0} = n(-\nu^{-1}_{0}z_{-1}+\mu_{0}\nu_{0}z_{1}) 
 = -\nu^{-1}_{0}nz_{-1}+\mu_{0}\nu_{0}z_{1}.
$$
On the other hand, from (\ref{6.2.4}) we get
$$
(nX_{V}n^{-1})z_{0} = -\nu^{-1}_{0}z_{-1}+\mu^{\p}_{0}\nu_{0}z_{1}.
$$
It follows that 
$$
nz_{-1}-z_{-1} = (\mu_{0}-\mu^{\p}_{0})\nu^{2}_{0}z_{1}.
$$
After pairing both sides with $z_{-1}$, we have $q(nz_{-1},z_{-1}) = q(z_{-1},z_{-1}) = 0$. The first identity $q(nz_{-1}, z_{-1}) =0$ follows from the fact that $n\in N$. Hence we can deduce that $\mu_{0} = \mu^{\p}_{0}$ so that we indeed have $nz_{-1} = z_{-1}$. Now let $1\leq j\leq r-1$ and assume that (\ref{6.2.6}) is true for any $1\leq i\leq j$. By (\ref{6.2.3}) and (\ref{6.2.5}) we have 
\begin{align*}
(nX_{V}n^{-1})z_{-j} = nX_{V}z_{-j} 
&=
n(-z_{-j-1}-\mu_{j-1}z_{j-1}+\mu_{j}z_{j+1})
\\
&=
-nz_{-j-1}-\mu_{j-1}z_{j-1}+\mu_{j}z_{j+1}.
\end{align*}
On the other hand, we have 
$$
(nX_{V}n^{-1})z_{-j} = 
-z_{-j-1}-\mu^{\p}_{j-1}z_{j-1}+\mu^{\p}_{j}z_{j+1}.
$$
It follows that 
$$
nz_{-j-1}-z_{-j-1}
=(\mu^{\p}_{j-1}-\mu_{j-1})
z_{j-1}+(\mu_{j}-\mu^{\p}_{j})z_{j+1}.
$$
By induction hypothesis, we have $nz_{-j} = z_{-j}$, $nz_{-j+1} = z_{-j+1}$. Moreover since $n\in \SO(V)$, we have $q(nz_{-j-1}, z_{-j}) = q(z_{-j-1}, z_{-j}) = 0$, $q(nz_{-j-1}, z_{-j+1}) = q(-z_{-j-1}, z_{-j+1}) = 0$. Hence we deduce that $\mu^{\p}_{j-1} = \mu_{j-1}$ and $\mu^{\p}_{j} = \mu_{j}$, and $nz_{-j-1} = z_{-j-1}.$ This ends the proof of (\ref{6.2.6}).

From (\ref{6.2.5}) and (\ref{6.2.6}) and since $n\in N$, we immediately deduce that $n=1$. This ends the proof that the map (\ref{6.2.2}) is injective. By explicit computation, we have 
\begin{align*}
\dim (N\times \Lam)  &=\dim(N)+\dim (\Lam)
\\
&=(r^{2}+rm)+m(m-1)/2+m+r
\\
&=m(m+1)/2+r^{2}+mr+r
\\
\dim(\Sig) = \dim(\Fh^{\perp}) 
&= \dim(G) - \dim(H) 
\\
&=
(m+2r+1)(m+2r)/2 - (r^{2}+rm) 
\\
&= m(m+1)/2+mr+r^{2}+r,
\end{align*}
where $m = \dim(W)$. Hence $\dim(\Sig) = \dim(N\times \Lam)$. Since we are in characteristic $0$ situation, we get that the regular map (\ref{6.2.2}) induces an isomorphism between $N\times \Lam$ and a Zariski open subset of $\Sig$. But $N\times \Lam$ and $\Sig$ are both affine spaces hence the only Zariski open subset of $\Sig$ that can be isomorphic to $N\times \Lam$ is $\Sig$ itself. It follows that the regular map (\ref{6.2.2}) is an isomorphism.
\end{proof}

\subsection{Characteristic polynomial}\label{sec:6.3}
Let $X = (X_{W},X_{V})\in \Lam$. By the definition of $\Lam$, one may write 
\begin{equation}\label{6.3.1}
X_{V} = \Xi - X_{W}+ c(z_{r},w)+
\sum_{i=0}^{r-1}
\mu_{i}c(z_{i},z_{i+1})
\end{equation}
where $w\in W_{\wb{F}}$ and $\mu_{i}\in \wb{F}$. Denote by $P_{X_{V}}$ and $P_{-X_{W}}$ the characteristic polynomials of $X_{V}$ and $-X_{W}$ acting on $V_{\wb{F}}$ and $W_{\wb{F}}$ respectively, both of which are elements of $\wb{F}[T]$. Let $D$ be the $\wb{F}$-linear endomorphism of $\wb{F}[T]$ given by $D(T^{i+1}) = T^{i}$ for $i\geq 0$ and $D(1) = 0$.

Over the algebraic closure $\wb{F}$, we fix a hyperbolic basis for $W_{\wb{F}}$ which are eigenvectors for $-X_{W}$, i.e. when $\dim W = m$ is even, a basis $\{w_{i}\}_{i= \pm 1,..., \pm \frac{m}{2} }$ satisfying $q(w_{i},w_{j})= \del_{i,-j}$; when $\dim W = m$ is odd, a basis $\{ w_{i}\}_{i=0, \pm 1, ..., \pm \frac{m-1}{2}}$ satisfying $q(w_{i},w_{j}) = \del_{i,-j}$. Moreover we have $-X_{W}(w_{i}) = s_{i}w_{i}$, $s_{i}\in \wb{F}$ with $i>0$. Then in particular 
$$
P_{-X_{W}}(T) = 
\bigg\{
\begin{matrix}
\prod_{i=1}^{\frac{m}{2}}(T^{2}-s_{i}^{2}),& \text{ $\dim W =m$ is even;}\\
T\prod_{i=1}^{\frac{m-1}{2}}(T^{2}-s_{i}^{2}),& \text{ $\dim W =m$ is odd.}
\end{matrix}
$$

Write $w\in W$ as $w = \sum_{i=\pm 1}^{\pm \frac{m}{2}}z_{i}w_{i}$ with $z_{i}\in \wb{F}$ when $\dim W= m$ is even, and $w=\sum_{i=0, \pm 1}^{\pm \frac{m-1}{2}}z_{i}w_{i}$ with $z_{i}\in \wb{F}$ when $\dim W =m$ is odd. In the following proposition we will write $w$ to indicate the associated column vector of $w$ under the above basis, and $q(w,W)$ to indicate the row vector $\sum_{i=\pm 1}^{\pm \frac{m}{2}}q(w,w_{i})w_{i} = \sum_{i=\pm 1}^{\pm \frac{m}{2}}z_{-i}w_{i}$ when $\dim W=m$ is even, and $\sum_{i=0, \pm 1}^{\pm \frac{m-1}{2}}q(w,w_{i})w_{i} = \sum_{i=0,\pm 1}^{\pm \frac{m-1}{2}}z_{-i}w_{i}$ when $\dim W=m$ is odd.

\begin{pro}\label{pro:6.3.1}
The following equality relating $P_{X_{V}}(T)$ and $P_{-X_{W}}(T)$ holds.

If $r=0$, then 
\begin{align*}
P_{X_{V}}(T) &= 
\det\left(
\begin{matrix}
T-(-X_{W}) &-\nu_{0}w	\\
q(w,W)& T
\end{matrix}
\right)
\\
&=TP_{-X_{W}}(T)+
\sum_{i=0}^{m-1}
(-1)^{j}
\nu_{0}q(w,X^{j}_{W}w)D^{j+1}(P_{-X_{W}}(T)).
\end{align*}
The formula can also be written as 
\begin{align*}
P_{X_{V}}(T) &= 
TP_{-X_{W}}(T)+
\\
&\bigg\{
\begin{matrix}
2\nu_{0}\sum_{i=1}^{\frac{m}{2}}
z_{i}z_{-i}\frac{TP_{-X_{W}}(T)}{T^{2}-s_{i}^{2}},
& \text{ when $\dim W= m$ is even,}\\
\nu_{0}z_{0}^{2}\frac{P_{-X_{W}}(T)}{T}
+2\nu_{0}\sum_{i=1}^{\frac{m-1}{2}}
z_{i}z_{-i}\frac{TP_{-X_{W}}(T)}{T^{2}-s_{i}^{2}},
& \text{ when $\dim W= m$ is odd.}
\end{matrix}
\end{align*}

If $r>0$, then
\begin{align*}
P_{X_{V}}(T) 
&=
(-1)^{r}
\det
\left(
\begin{matrix}
T-(-X_{W})	&	w 	\\
q(w,W)	& 0
\end{matrix}
\right)+
\\ 
&P_{-X_{W}}(T)(T^{2r+1}-
(\nu_{0}+\nu_{0}^{-1})T^{2r-1}\mu_{0}
+
 \sum_{j=1}^{r-1}(-1)^{j+1}2\nu_{0}^{-1}\mu_{j}T^{2r-1-2j} ).
\end{align*}
Here 
\begin{align*}
&\det
\left(
\begin{matrix}
T-(-X_{W})	&	w 	\\
q(w,W)	& 0
\end{matrix}
\right)
=\sum_{i=0}^{m-1}(-1)^{j+1}q(w,X^{j}_{W}w)D^{j+1}(P_{-X_{W}}(T))
\\
&=
\bigg\{
\begin{matrix}
-2\sum_{i=1}^{\frac{m}{2}}
z_{i}z_{-i}\frac{TP_{-X_{W}}(T)}{T^{2}-s_{i}^{2}},
& \text{ when $\dim W= m$ is even,}\\
-z_{0}^{2}\frac{P_{-X_{W}}(T)}{T}
-2\sum_{i=1}^{\frac{m-1}{2}}
z_{i}z_{-i}\frac{TP_{-X_{W}}(T)}{T^{2}-s_{i}^{2}},
& \text{ when $\dim W= m$ is odd.}
\end{matrix}
\end{align*}

\end{pro}
\begin{proof}
The statement can be proved via induction on $r$. We left the proof to the reader.
\end{proof}

\begin{cor}\label{cor:6.3.2}
The following $\SO(W)$-invariant polynomial functions on $\Lam$
$$
X=  (X_{W},X_{V}) \to q(w,X^{j}_{W}w)\in \wb{F}, \quad 
j=0,...,m-1
$$
(where we have written $X_{V}$ as in (\ref{6.3.1})) extend to $G$-invariant polynomial functions on $\Fg$ defined over $F$.
\end{cor}

In particular, the polynomial function 
$$
X = (X_{W},X_{V})\to \det(q(X^{i}_{W}w,X^{j}_{W}w))_{0\leq i,j \leq m-1}\in \wb{F}
$$
extends to a $G$-invariant polynomial function on $\Fg$ defined over $F$. Denote by $Q_{0}$ such an extension. Set $d^{G}(X): = \det(1-\Ad(X))_{|\Fg/\Fg_{X}}$ for any $X\in \Fg_{reg}$. The $d^{G}$ extends uniquely to a polynomial $d^{G}\in F[\Fg]^{G}.$ Set $Q = Q_{0}d^{G}\in F[\Fg]^{G}$ and let $\Lam^{\p}$ and $\Sig^{\p}$ be the non-vanishing loci of $Q$ in $\Lam$ and $\Sig$ respectively. Notice that we have $\Lam^{\p}\subset \Lam_{reg}$ and $\Sig^{\p}\subset \Sig_{reg}$ since $d^{G}$ divides $Q$.

There is the following characterization of $\Sig^{\p}$.
\begin{pro}\label{pro:6.4.1}
$\Sig^{\p}$ is precisely the set of $X=  (X_{W},X_{V})\in \Sig_{reg}$ such that the family
$$
z_{r}, X_{V}z_{r},..., X^{d-1}_{V}z_{r}
$$
generates $V_{\wb{F}}$ as a $\wb{F}$-module (Recall that $d=\dim (V)$).
\end{pro}
\begin{proof}
The proof in \cite[Proposition~10.4.1]{beuzart2015local} works verbatim.
\end{proof}

\subsection{Conjugation classes in $\Sig^{\p}$}\label{sec:6.5}
\begin{pro}\label{pro:6.5.1}
The conjugation action of $H = \SO(W)\ltimes N$ on $\Sig^{\p}$ is free. Two elements of $\Sig^{\p}$ are $G$-conjugate if and only if they are $H$-conjugate. Moreover, after taking the $F$-rational points, two elements of $\Sig^{\p}(F)$ are $G$-conjugate (in the sense of stable conjugacy) if and only if they are $H(F)$-conjugate.
\end{pro}
\begin{proof}
For the case when $d=\dim V\leq 2$, the proposition holds automatically. Therefore in the following we only consider the case when $\dim V\geq 3$. In this case we notice that the adjoint orbit of $\Fs\Fo(V)$ is determined by its characteristic polynomial, in other words two elements of $\Fs\Fo(V)$ are conjugate by $\SO(V)$ action if and only if they have the same characteristic polynomial. We also notice that when $\dim V = 2$, two elements fo $\Fs\Fo(V)$ have the same characteristic polynomial if and only if they are conjugate by $\mathrm{O}(V)$.

(1) We look at the case when $\dim W\neq 2$. We follow closely with \cite[Lemme~9.5]{waldspurger10}.

Recall that by definition, $H$ acts freely on $\Sig^{\p}$ if the map
\begin{align*}
H\times \Sig^{\p}\to \Sig^{\p}\times \Sig^{\p}
\\
(h,X)\to (X,hXh^{-1})
\end{align*}
is a closed immersion. By Proposition \ref{pro:6.2.1}, this is equivalent to proving that
\begin{align}\label{6.5.1}
\SO(W)\times \Lam^{\p}\to \Lam^{\p}\times \Lam^{\p}
\\
(h,X)\to (X,hXh^{-1})	\nonumber
\end{align}
is a closed immersion. For $X= (X_{W},X_{V})\in \Fg$, we define the characteristic polynomial of $X$ to be the pair $P_{X} = (P_{X_{W}},P_{X_{V}})$. Let $\CY\subset \Lam^{\p}\times \Lam^{\p}$ be the closed subset of pairs $(X,X^{\p})$ such that $P_{X}= P_{X^{\p}}$. We claim the following
\begin{num}
\item\label{6.5.2}
The map is a closed immersion whose image is $\CY$.
\end{num}
This will prove the two points of the proposition, since if two elements of $\Fg$ are $G$-conjugate, they share the same characteristic polynomial, therefore they are $H$-conjugate by (\ref{6.5.2}). First by definition, the image of (\ref{6.5.1}) is contained in $\CY$. Let $(X,X^{\p})\in \CY$. We may write 
\begin{align*}
X_{V}  =\Xi- X_{W} +c(z_{r},w)
+\sum_{i=0}^{r-1}\mu_{i}c(z_{i},z_{i+1})
\\
X^{\p}_{V} = \Xi -X^{\p}_{W}+c(z_{r},w^{\p})
+\sum_{i=0}^{r-1}\mu^{\p}_{i}c(z_{i},z_{i+1})
\end{align*}
where $X = (X_{V},X_{W}),X^{\p} = (X^{\p}_{V},X^{\p}_{W}),w,w^{\p}\in W_{\wb{F}}$, and $\mu_{i},\mu^{\p}_{i}\in \wb{F}$. Since $\dim W\neq 2$, and $P_{X_{W}} = P_{X^{\p}_{W}},$ we know that $X_{W}$ and $X_{W}^{\p}$ are conjugate by $\SO(W)$ action. Hence up to $\SO(W)$-conjugation we may assume that $X_{W}=  X^{\p}_{W}$. By Proposition \ref{pro:6.3.1}, we have $\mu_{i} = \mu^{\p}_{i}$ for $i=0,...,r-1$ and 
\begin{equation}\label{6.5.3}
q(w,X^{i}_{W}w)  = q(w^{\p},X^{\p i}_{W}w^{\p}) \quad i=0,...,m-1.
\end{equation}
Moreover, by definition of $\Lam^{\p}$, $(w,X_{W}w,...,X^{m-1}_{W}w)$ and $(w^{\p},X^{\p}_{W}w^{\p},...,X^{\p m-1}_{W}w^{\p})$ are bases of $W_{\wb{F}}$.

We first treat the case when $m=\dim W$ is even.
We fix a basis $\{w_{i}| \quad i=\pm 1,...,\pm \frac{m}{2} \}$ of $W_{\wb{F}}$ satisfying $q(w_{i},w_{j}) = \del_{i,-j}$ for any $i,j = \pm 1,...,\pm \frac{m}{2}$. Since $X_{W}$ lies in $\Sig_{reg}$, up to $\SO(W)$-conjugation we may assume that $X_{W}$ lies in the diagonal maximal split Cartan subalgebra of $\Fs\Fo(W)$. In particular, the centralizer of $X_{W}$ in $\SO(W)$ is the diagonal split torus of $\SO(W)$, and the action of any element in the centralizer of the form $\diag(z_{\frac{m}{2}},...,z_{1},z_{-1},...,z_{-\frac{m}{2}})$ on $w_{i}$ is via scaling $z_{i}$, where $i=\pm 1,...,\pm \frac{m}{2}$. If we write $w = \sum_{i=\pm 1}^{\pm \frac{m}{2}}a_{i}w_{i}$, and $w^{\p} = \sum_{i=\pm 1}^{\pm \frac{m}{2}}a_{i}^{\p}w_{i}$, then from (\ref{6.5.3}) we have $a_{i}a_{-i}= a_{i}^{\p}a^{\p}_{-i}$ for any $i=\pm 1,..., \pm \frac{m}{2}$. Moreover, since $(w,X_{W}w,...,X^{m-1}_{W}w)$ and $(w^{\p},X^{\p}_{W}w^{\p},...,X^{\p m-1}_{W}w^{\p})$ are bases of $W_{\wb{F}}$, and $X_{W} = X^{\p}_{W}$ is diagonal, we notice that $a_{i}$ and $a^{\p}_{i}$ are nonzero for any $i$. Now the element $\diag(z_{\frac{m}{2}},...,z_{1},z_{-1},...,z_{-\frac{m}{2}})$ with
$z_{i} = \frac{a_{i}^{\p}}{a_{i}}$ conjugates $X_{V}$ to $X^{\p}_{V}$. 

Then we consider the case when $m=\dim W$ is odd. Similarly we fix a basis $\{w_{0},w_{i}| \quad i=\pm 1,...,\pm \frac{m-1}{2} \}$ of $W_{\wb{F}}$ satisfying $q(w_{i},w_{j}) = \del_{i,-j}$ for any $i,j= 0,\pm 1,...,\pm \frac{m-1}{2}$. Using the same notation as the even case, we find that $a_{i}a_{-i} = a_{i}^{\p}a^{\p}_{-i}$ with $a_{i}$ and $a^{\p}_{i}$ being nonzero for any $i = 0,\pm 1,...,\pm \frac{m-1}{2}$. Then the element $\diag(z_{\frac{m-1}{2}}, ...,z_{1},1,z_{-1},...,z_{-\frac{m-1}{2}}$ with $z_{i} = \frac{a^{\p}_{i}}{a_{i}}$ conjugates $X_{V}$ to $X^{\p\p}_{V}$, with $X^{\p\p}_{V} = X^{\p}_{V}$, or $X^{\p\p}_{V}$ and $X^{\p}_{V}$ differ by the conjugation action of the element $\del\in\mathrm{O}(V)$ sending $w_{0}$ to $-w_{0}$ and stabilizing the complement of $w_{0}$ in $V$. We only need to show that the second case does not occur. Otherwise, since $X_{V}$ is $\SO(V)$-conjugate to both $X^{\p}_{V}$ and $X^{\p\p}_{V}$, and $X_{V}^{\p}$ is conjugate to $X^{\p\p}_{V}$ by $\del\in \mathrm{O}(V)\bs \SO(V)$ action, we notice that the centralizer of $X_{V}$ in $\mathrm{O}(V)$ is not connected. However, from Proposition \ref{pro:6.4.1}, we know that $X_{V}$ is regular and does not contain $0$ as its eigenvalues, therefore the centralizer of $X_{V}$ should be connected, which is a contradiction. Hence in conclusion we have that the element $\diag(z_{\frac{m-1}{2}}, ...,z_{1},1,z_{-1},...,z_{-\frac{m-1}{2}})\in \SO(W)$ conjugates $X_{V}$ to $X^{\p}_{V}$.

The uniqueness of $g\in \SO(W)$ is easy to derive, which, again using the conjugation of $c(z_{r},w)$ and $c(z_{r},w^{\p})$, we know that $g$ must send $w$ to $w^{\p}$, and therefore the element $g$ should exactly send the basis $(w,X_{W}w,...,X^{m-1}_{W}w)$ to $(w^{\p},X^{\p}_{W}w^{\p},...,X^{\p m-1}_{W}w^{\p})$, hence unique.

Hence, we have proved that the map induces a bijection from $\SO(W)\times \Lam^{\p}$ to $\CY$ and we have constructed the inverse, which is a morphism of algebraic varieties.

Finally, if we take the $F$-rational points, then the uniqueness of the element $g\in \SO(W)$ immediately implies that $g\in \SO(W)(F)$ as it sends the $F$-rational basis $(w,X_{W}w, ..., X^{m-1}_{W}w)$ to $(w^{\p}, X_{W}^{\p}w^{\p}, ..., X^{\p m-1}_{W}w^{\p})$, from which we directly deduce that the two elements in $\Lam^{\p}(F)$ are $\SO(W)(F)$-conjugate, and so the parallel statement for $\Sig^{\p}(F)$ holds.

(2) We look at the case when $\dim W =2$. We use same notation as the case when $\dim W\neq 2$.

Similar to the case when $\dim W\geq 3$, since we assume that $X= (X_{W},X_{V})$ and $(X_{W}^{\p},X^{\p}_{V})$ are $G$-conjugate, we know that $X_{W} =X^{\p}_{W}$. The remaining discussion is the same as the $\dim W\geq 3$ case.
\end{proof}

\begin{cor}\label{cor:6.5.2}
The following inequality holds
$$
\sig_{G}(t) \ll \sig_{H\bs G}(t) \sig_{\Sig^{\p}}(X)
$$
for any $X\in \Sig^{\p}$ and $t\in G_{X}$.
\end{cor}
\begin{proof}
The proof of \cite[Corollary~10.5.2]{beuzart2015local} works verbatim.
\end{proof}

\subsection{Borel sub-algebras and $\Sig^{\p}$}\label{sec:6.6}
\begin{pro}\label{pro:6.6.1}
Let $X\in \Sig^{\p}$ and $\Fb$ be a Borel subalgebra of $\Fg$ defined over $\wb{F}$ containing $X$, then 
$$
\Fb \oplus \Fh =\Fg.
$$
\end{pro}
\begin{proof}
Let $X\in \Sig^{\p}$ and $\Fb\subset \Fg$ be a Borel subalgebra containing $X$. By Proposition \ref{pro:6.2.1}, up to $N(F)$-conjugation we may assume that $X\in \Lam^{\p}$ and we will assume this so henceforth. Write $X = (X_{W},X_{V})$ with $X_{W}\in \Fs\Fo(W)$ and $X_{V}\in \Fs\Fo(V)$. By definition of $\Lam$, we have a decomposition 
\begin{equation}\label{6.6.1}
X_{V}  = \Xi - X_{W}+c(z_{r},w) + \sum_{i=0}^{r-1}
\mu_{i}c(z_{i},z_{i+1})
\end{equation}
where $w\in W_{\wb{F}}$, $0\leq i\leq r$ and $\mu_{i}\in \wb{F}$, $0\leq i\leq r-1$.

By direct computation we have $\dim(\Fb)+\dim(\Fh) = \dim(\Fg)$, hence it suffices to prove 
\begin{equation}\label{6.6.2}
\Fb\cap \Fh = 0.
\end{equation}

There exist Borel subalgebras $\Fb_{W}$ and $\Fb_{V}$ of $\Fs\Fo(W)$ and $\Fs\Fo(V)$ respectively such that $X_{W}\in \Fb_{W}$, $X_{V}\in \Fb_{V}$ and $\Fb = \Fb_{W}\times \Fb_{V}$. Then $\Fb\cap \Fh = 0$ is equivalent to 
\begin{equation}\label{6.6.3}
(\Fb_{W}+\Fn)\cap \Fb_{V} =0.
\end{equation}
Fix an $F$-embedding $F\hookrightarrow \wb{F}$ and set $\wb{V} = V\otimes_{F}\wb{F}$, $\wb{W} = W\otimes_{F}\wb{F}$. Then we have isomorphism of $\wb{F}$-vector spaces
\begin{align*}
V_{\wb{F}}\simeq \wb{V},\quad 
W_{\wb{F}} \simeq \wb{W}.
\end{align*}
Also, if $U$ is a subspace of $V$ we will set $\wb{U } = U\otimes_{F}\wb{F}$ and view it as a subspace of $\wb{V}$. We will adopt similar notation w.r.t. $\wb{W}$. We have isomorphisms
$$
\Fs\Fo(V)_{\wb{F}}\simeq \Fs\Fo(\wb{V}), \quad 
\Fs\Fo(W)_{\wb{F}} \simeq \Fs\Fo(\wb{W})
$$
and we will identify two sides via the isomorphisms. Then, $\Fb_{W}$ is the stabilizer in $\Fs\Fo(\wb{W})$ of a complete flag
$$
0 = \wb{W}_{0}\subset \wb{W}_{1} \subset... \subset \wb{W}_{m}  = \wb{W}
$$
and similarly $\Fb_{V}$ is the stabilizer in $\Fs\Fo(\wb{V})$ of a complete flag
$$
\CF: 0 = \wb{V}_{0}\subset \wb{V}_{1} \subset... \subset \wb{V}_{d} = \wb{V}.
$$
We define another complete flag
$$
\CF^{\p}: 0 = \wb{V}^{\p}_{0}\subset \wb{V}^{\p}_{1} \subset... \subset
\wb{V}^{\p}_{d} = \wb{V}
$$
of $\wb{V}$ by setting
\begin{itemize}
\item $\wb{V}^{\p}_{i} = \langle \wb{z}_{r},..., \wb{z}_{r-i+1} \rangle $ for $i=1,...,r+1$;

\item $\wb{V}^{\p}_{r+1+i} = \wb{Z}_{+}\oplus \wb{D}\oplus \wb{W}_{i}$ for $i=1,...,m$;

\item $\wb{V}^{\p}_{r+m+1+i} = \wb{Z}_{+}\oplus \wb{V}_{0}\oplus \langle \wb{z}_{-1},...,\wb{z}_{-i}\rangle$ for $i=1,...,r$.
\end{itemize}
for any $v\in \wb{V}$, let us denote by $V(X_{V},v)$ the subspace of $\wb{V}$ generated by $v, X_{V}v,X^{2}_{V}v,...$ We have the following lemma.
\begin{lem}\label{lem:6.6.2}
Let $1\leq i\leq d$, then we have 
\begin{enumerate}
\item for any nonzero $v\in \wb{V}^{\p}_{i}$, $\wb{V}^{\p}_{i-1}+V(X_{V},v) = \wb{V}$;

\item $\wb{V}^{\p}_{i}\cap \wb{V}_{d-i} = 0$.
\end{enumerate}
\end{lem}
\begin{proof}
First we prove that (1) implies (2). Indeed, if $v\in \wb{V}^{\p}_{i}\cap \wb{V}_{d-i}$ is nonzero, then by (1), we would have $\dim(V(X_{V},v))\geq d+1-i$. But $V(X_{V},v)\subset \wb{V}_{d-i}$ (since $v\in \wb{V}_{d-i}$ and $X_{V}\in \Fb_{V}$ preserves $\wb{V}_{d-i}$), and therefore $\dim(V(X_{V},v))\leq \dim(\wb{V}_{d-i}) = d-i$. This is a contradiction.

We now prove the first statement. Let $v\in \wb{V}^{\p}_{i}$ be nonzero. Without loss of generality, we may assume that $v\in \wb{V}^{\p}_{i}\bs \wb{V}^{\p}_{i-1}$ since otherwise the result with $i-1$ instead of $i$ is stronger. We assume this is so henceforth and it follows that 
\begin{equation}\label{6.6.4}
\wb{V}^{\p}_{i-1}+V(X_{V},v) = \wb{V}^{\p}_{i}+V(X_{V},v).
\end{equation}
By definition of $\wb{V}_{i}^{\p}$, we have $\wb{z}_{r}\in \wb{V}^{\p}_{i}+V(X_{V},v)$ and so by Proposition \ref{pro:6.4.1}, it suffices to show that $\wb{V}^{\p}_{i-1}+V(X_{V},v)$ is $X_{V}$-stable. The subspace 
$V(X_{V},v)$ is $X_{V}$-stable almost by definition. Hence, we are left with proving that
\begin{equation}\label{6.6.5}
X_{V}\wb{V}^{\p}_{i-1}\subset \wb{V}^{\p}_{i}+V(X_{V},v).
\end{equation}
This is clear if $1\leq i\leq r+1$ or $r+m+2\leq i\leq d = 2r+m+1$ since in this case using the decomposition of $X_{V}$ we easily check that $X_{V}\wb{V}^{\p}_{i-1}\subset \wb{V}^{\p}_{i}$. It remains to show that it also holds for $r+2\leq i \leq r+m+1$. In this case, again using the decomposition of $X_{V}$, we can find that
\begin{equation}\label{6.6.6}
X_{V}v^{\p}\in \wb{V}^{\p}_{i}+\langle \wb{z}^{*}_{0},v^{\p}\rangle X_{V}\wb{z}_{0}
\end{equation}
for any $v^{\p}\in \wb{V}^{\p}_{i}$ (where $X_{V}\wb{z}_{0} = -\nu_{0}^{-1}\wb{z}_{-1}+\mu_{0}\nu_{0} \wb{z}_{1}$ if $r\geq 1$ and $X_{V}\wb{z}_{0} = 0$ if $r=0$). Here, we have used the fact that $X_{W}\in \Fb_{W}$ so that $\wb{W}_{i-r-2}$ and $\wb{W}_{i-r-1}$ are $X_{W}$-stable. As $v\in \wb{V}_{i}^{\p}$, it suffices to show that the existence of $k\geq 0$ such that $\langle \wb{z}^{*}_{0}, X^{k}_{V}v\rangle \neq 0$. By Proposition \ref{pro:6.4.1}, the family
$$
\wb{z}^{*}_{r}, {}^tX_{V}\wb{z}^{*}_{r}, {}^tX^{2}_{V}\wb{z}^{*}_{r},...
$$
generates $\wb{V}^{*}$. Hence, since $v\neq 0$, there exists $k_{0}\geq 0$ such that 
$\langle {}^tX^{k_{0}}_{V}\wb{z}^{*}_{r},v\rangle = \langle \wb{z}^{*}_{r}, X^{k_{0}}_{V}v \rangle \neq 0$. This already settles the case where $r=0$. In the case $r\geq 1$, since $\wb{V}^{\p}_{i}$ is included in the kernel of $\wb{z}^{*}_{r}$ this shows that the sequence $v,X_{V}v,X^{2}_{V}v,...$ eventually does not lie in $\wb{V}^{\p}_{i}$ and by (\ref{6.6.6}) this implies also the existence of $k\geq 0$ such that $\langle \wb{z}^{*}_{0}, X^{k}_{V}v \rangle \neq 0$. This ends the proof of (\ref{6.6.5}) and of the lemma.
\end{proof}

Let us now set $D_{i} = \wb{V}^{\p}_{i}\cap \wb{V}_{d+1-i}$ for $i=1,..,d$. By the previous lemma and dimension consideration, these are one dimensional subspaces of $\wb{V}$ and we have 
$$
\wb{V} = \bigoplus_{i=1}^{d}D_{i}.
$$
Let $Y\in (\Fb_{W}+\Fn)\cap \Fb_{V}.$ We want to prove that $Y=0$ so as to obtain (\ref{6.6.2}). By definition $Y$ must stabilize the flags $\CF$ and $\CF^{\p}$ so that $Y$ stabilizes the lines $D_{1},...,D_{d}$. We claim that 
\begin{equation}\label{6.6.7}
Y(D_{i}) = 0, i=1,..,r+1 , i= r+m+2,...,d.
\end{equation}
Indeed, since $Y\in \Fb_{W}+\Fn$, we have $Y\wb{V}^{\p}_{i}\subset \wb{V}^{\p}_{i-1}$ for any $i=1,...,r+1$ and $i=r+m+2,...,d$ and so $YD_{i}\subset \wb{V}^{\p}_{i-1}\cap \wb{V}_{d+1-i} =0$.

To deduce that $Y = 0$, it only remains to show that 
\begin{equation}\label{6.6.8}
Y(D_{i}) = 0 , i=r+2,...,r+m+1.
\end{equation}
Assume, by way of contradiction, that there exists $1\leq j\leq m$ such that $YD_{r+1+j}\neq 0$. Since $Y\in \Fb_{W}+\Fn$, we have $Y\wb{V}^{\p}_{r+1+j}\subset \wb{W}_{j}\oplus \wb{Z}_{+}$ so that $D_{r+1+j} = YD_{r+1+j}\subset \wb{W}_{j}\oplus \wb{Z}_{+}$. Let $v\in D_{r+1+j}$ be nonzero. We claim that
\begin{equation}\label{6.6.9}
(\wb{Z}_{+}\oplus \wb{W}_{j})+V(X_{V},v) = \wb{V}
\end{equation}
Indeed by the previous lemma, it suffices to prove that $\wb{z}_{0}\in (\wb{Z}_{+}\oplus \wb{W}_{j})+V(X_{V},v)$. By the decomposition (\ref{6.6.1}), we can find that
$$
X_{V}(\wb{Z}_{+}\oplus \wb{W}_{j})\subset \wb{Z}_{+}\oplus \wb{W}_{j}
\oplus \wb{F}\wb{z}_{0}
$$
so that we only need to check that the sequence $v,X_{V}v,...$ eventually does not lie in $\wb{Z}_{+}\oplus \wb{W}_{j}$. From Proposition \ref{pro:6.4.1}, we know that there exists $k\geq 0$ such that $\langle \wb{z}^{*}_{r}, X^{k}_{V}v \rangle \neq 0$. Since  $\wb{Z}_{+}\oplus \wb{W}_{j}$ is included in the kernel of $\wb{z}^{*}_{r}$, this proves (\ref{6.6.9}).

From (\ref{6.6.9}), we deduce that $\dim V(X_{V},v)\geq 1+d-j-r$. On the other hand, we have $v\in \wb{V}_{d-r-j}$ since $v\in \wb{D}_{r+1+j}$, and $X_{V}$ leaves $\wb{V}_{d-r-j}$ stable (since $X_{V}\in \Fb_{V}$). As $\dim \wb{V}_{d-r-j} = d-r-j$ it is a contradiction. This ends the proof of (\ref{6.6.8}) and of the Proposition.
\end{proof}

\begin{cor}\label{cor:6.6.3}
There exists a constant $c>0$ such that for any $\eps>0$ sufficiently small, any $X\in \Sig^{\p}(F)$ and parabolic subalgebras $\Fp$ of $\Fg$ defined over $F$ and containing $X$, 
$$
\exp[B(\cdot, \eps e^{-c\sig_{\Sig^{\p}}(X)})] \subset H(F)\exp(B(0,\eps)\cap \Fp(F)).
$$
\end{cor}
\begin{proof}
The proof of \cite[Corollary~10.6.3]{beuzart2015local} works verbatim.
\end{proof}

\subsection{The quotient $\Sig^{\p}(F)/H(F)$}\label{sec:6.7}
By Proposition \ref{pro:6.2.1}, $\Sig^{\p}$ has a geometric quotient by $N$ and $\Sig^{\p}/N \simeq \Lam^{\p}$. Because $H = N\rtimes \SO(W)$ and $\SO(W)$ is reductive, the geometric quotient of $\Sig^{\p}$ by $H$ exists and $\Sig^{\p}/H \simeq \Lam^{\p}/\SO(W)$. Denote by $\Fg^{\p}$ the non-vanishing locus of $Q$ in $\Fg$ and by $\Fg^{\p}/G$ the geometric quotient of $\Fg^{\p}$ by $G$ for the adjoint action. The natural map $\Sig^{\p}\to \Fg^{\p}/G$ factors through the quotient $\Sig^{\p}/H$ and there is the induced morphism
$$
\pi: \Sig^{\p}/H \to \Fg^{\p}/G.
$$
We will also consider the $F$-analytic counterpart of this map:
$$
\pi_{F}:\Sig^{\p}(F)/H(F)\to \Fg^{\p}(F)/G(F).
$$
Put on $H(F)$ the Haar measure $\mu_{H}$ which lift the Haar measure $\mu_{\Fh}$ on $\Fh(F)$. Because $H(F)$ acts freely on $\Sig^{\p}(F)$, we can define a measure $\mu_{\Sig^{\p}/H}$ on $\Sig^{\p}(F)/H(F)$ to be the quotient of (the restriction to $\Sig^{\p}(F)$ of) $\mu_{\Sig}$ by $\mu_{H}$. It is the unique measure on $\Sig^{\p}(F)/H(F)$ such that 
$$
\int_{\Sig(F)}\vphi(X)\ud\mu_{\Sig}(X)= \int_{\Sig^{\p}(F)/H(F)}
\int_{H(F)}\vphi(h^{-1}Xh)\ud h \ud\mu_{\Sig^{\p}/H}(X)
$$
for any $\vphi\in C_{c}(\Sig(F))$. One can define a measure $\ud X$ on $\Fg^{\rss}(F)/G(F) = \Gam^{\rss}(\Fg)$ following \cite[Section~1.7]{beuzart2015local}. Moreover, $\Fg^{\p}(F)/G(F)$ is an open subset of $\Fg^{\rss}(F)/G(F)$ and we will still denote by $\ud X$ the restriction of this measure to $\Fg^{\p}(F)/G(F)$.

\begin{pro}\label{pro:6.7.1}
\begin{enumerate}
\item
$\pi$ is an isomorphism of algebraic varieties and $\pi_{F}$ is an open embedding of $F$-analytic spaces;

\item $\pi_{F}$ sends the measure $\ud\mu_{\Sig^{\p}/H}(X)$ to $D^{G}(X)^{1/2}\ud X$;

\item The natural projection $p: \Sig^{\p}\to \Sig^{\p}/H$ has the norm descent property.
\end{enumerate}
\end{pro}
\begin{proof}
The proof of (1) and (2) follows from the same argument as in \cite[Proposition~10.7.1]{beuzart2015local}. We only establish (3).

For (3), by Proposition \ref{pro:6.2.1}, it is sufficient to show that 
$$
\Lam^{\p}\to \Lam^{\p}/\SO(W)
$$
has the norm descent property. Denote by $\Lam_{Q_{0}}$ the non-vanishing locus of $Q_{0}$ in $\Lam$ (where $Q_{0}\in F[\Fg]^{G}$ is defined in Section \ref{sec:6.3}. Then we have the following Cartesian diagram where horizontal maps are open immersions
$$
\xymatrix{
\Lam^{\p} \ar[r] \ar[d] & \Lam_{Q_{0}} \ar[d] \\
\Lam^{\p}/\SO(W) \ar[r] & \Lam_{Q_{0}}/\SO(W)
}
$$
Thus, if we prove that $\Lam_{Q_{0}}\to \Lam_{Q_{0}}/\SO(W)$ has the norm descent property, we will be done. By definition of $\Lam$ (cf. Section \ref{sec:6.2}), we have an $\SO(W)$-equivariant isomorphism
$$
\Lam\simeq \Fs\Fo(W)\times W\times \BA^{r},
$$
where the action of $\SO(W)$ on the RHS is the product of the adjoint action on $\Fs\Fo(W)$, the natural action on $W$ and the trivial action on $\BA^{r}.$ Denote by $(\Fs\Fo(W)\times W)_{0}$ the open-Zariski subset of $\Fs\Fo(W)\times W$ consisting of all pairs $(X,w)\in \Fs\Fo(W)\times W$ such that $(w,Xw,X^{2},...)$ generates $W$. Then $\Lam_{Q_{0}}$ corresponds via the previous isomorphism to $(\Fs\Fo(W)\times W)_{0}\times \BA^{r}$. Since $\SO(W)$ acts trivially on $\BA^{r}$, we are reduced to show that 
$$
(\Fs\Fo(W)\times W)_{0}\to (\Fs\Fo(W)\times W)_{0}/\SO(W)
$$
has the norm descent property. Let $\CB$ be the flag variety of basis of $W$ and let $Pol_{m}$ be the variety of monic polynomial $P\in \wb{F}[T]$ of degree $m$. Consider the action of $\SO(W)$ on $Pol_{m}\times \CB$ which is trivial on $Pol_{m}$ and given by $g(e_{1},...,e_{m}) = (ge_{1},...,ge_{m})$ on $\CB$. By Proposition \ref{pro:6.5.1} the map
\begin{align*}
(\Fs\Fo(W)\times W)_{0}
&\to Pol_{m}\times \CB
\\
(X,w)
&\to (P_{X},w,Xw,...,X^{m-1}W)
\end{align*}
is a $\SO(W)$-equivariant closed immersion. Passing to the quotient, we get a commutative diagram
$$
\xymatrix{
(\Fs\Fo(W)\times W)_{0} \ar[r] \ar[d] & Pol_{m}\times \CB \ar[d] \\
(\Fs\Fo(W)\times W)_{0}/\SO(W)\ar[r] & Pol_{m}\times \CB/\SO(W)
}
$$
where horizontal maps are closed immersion (since $\SO(W)$ is reductive). Moreover the diagram is Cartesian (because all $\SO(W)$-orbits in $\CB$ are closed and so the quotient separates all orbits). Thus, we are finally reduced to showing that $\CB\to \CB/\SO(W)$ has the norm descent property. Choosing a particular basis of $W$, this amounts to proving that 
$$
\GL(W)\to \GL(W)/\SO(W)
$$
has the norm descent property. Since the map is $\GL(W)$-equivariant for the action by left translation, by \cite[Lemma~1.2.2(i)]{beuzart2015local}, it suffices to show the existence of a nonempty Zariski open subset of $\GL(W)/\SO(W)$ over which the previous map has the norm descent property. Choose an orthogonal basis $(e_{1},...,e_{m})$ of $W$ and denote by $B$ the standard Borel subgroup of $\GL(W)$ relative to this basis. Then $B\cap \SO(W) =Z$ is the subtorus acting by $\pm 1$ on each $e_{i}$ (so that $Z\simeq (\BZ/2\BZ)^{m}$). Let $U$ be the unipotent radical of $B$, then $U\cap \SO(W) = \{1 \}$, moreover $\dim U+\dim \SO(W) = \dim \GL(W)$. It follows that we have the canonical open immersion $U\hookrightarrow \GL(W)/\SO(W)$. Therefore
$$
\GL(W)\to \GL(W)/\SO(W)
$$
admits a section on an open dense subset, from which we ends the proof of (3).
\end{proof}

\subsection{Spectral expansion of $J^{\Lie}$}\label{sec:6.8}
Define $\Gam(\Sig)$ to be the subset of $\Gam(\Fg)$ consisting of the conjugacy classes of the semi-simple parts of elements in $\Sig(F)$. Equip the subset with the restriction of the measure defined on $\Gam(\Fg)$. Thus, if $\CT(G)$ is a set of representatives for the $G(F)$-conjugacy classes of maximal tori of $G$ and if for any $T\in \CT(G)$ we denote by $\Ft(F)_{\Sig}$ the subset of elements $X\in \Ft(F)$ whose conjugacy class belongs to $\Gam(\Sig)$, then
$$
\int_{\Gam(\Sig)}
\vphi(X)\ud X = 
\sum_{T\in \CT(G)}
|W(G,T)|^{-1}
\int_{\Ft(F)_{\Sig}}
\vphi(X)\ud X
$$
for any $\vphi\in \CC_{c}(\Gam(\Sig))$. Recall that in Section \ref{sec:distribution}, a continuous linear form $J^{\Lie}$ on $\CS_{\scusp}(\Fg)$ has been defined.

The theorem below is one of the most technical result of the paper. Fortunately the proof of \cite[\S 10.8-\S 10.11]{beuzart2015local} works verbatim. The details are referred to \cite{thesis_zhilin}.

\begin{thm}\label{thm:6.8.1}
The following identity holds
$$
J^{\Lie}(f) = 
\int_{\Gam(\Sig)}
D^{G}(X)^{1/2}
\wh{\theta}_{f}(X)\ud X
$$
for any $f\in \CS_{\scusp}(\Fg)$.
\end{thm}

\section{Geometric expansion and multiplicity formula}
In this section, we are going to establish the geometric expansion of the trace distribution $J$, from which we are able to deduce a geometric multiplicity formula $m(\pi)$ for any tempered representation $\pi$ of $G(F)$.

\subsection{The linear forms $m_\geom$ and $m^\Lie_\geom$}\label{sec:linearformgeom}
\subsubsection{Geometric support}

For any $x\in H_\ss(F)$, we are going to introduce spaces of semi-simple conjugacy classes in $G(F)$, $G_x(F)$ and $\Fg(F)$, which will be denoted by $\Gam(G,H)$, $\Gam(G_x,H_x)$ and $\Gam^\Lie(G,H)$, respectively. The definition is in parallel with \cite[11.1]{beuzart2015local}.

Fix $x\in H_\ss(F)$, up to conjugation we may assume that $x\in \SO(W)_\ss(F)$. Let $W^\p_x$ (resp. $V^\p_x$) be the kernel of $1-x$ in $W$ (resp. $V$) and $W^{\p\p}_x$ (resp. $V^{\p\p}_x$) be the image of $1-x$. There are orthogonal decompositions $W= W^\p_x\oplus^\perp W^{\p\p}_x$ and $V= V^\p_x\oplus^\perp W^{\p\p}_x$. Set $H^\p_x = \SO(W^\p_x)\ltimes N_x$ (where $N_x$ is the centralizer of $x$ in $N$), $G^\p_x = \SO(W^\p_x)\times \SO(V^\p_x)$, 
$H^{\p\p}_x = \SO(W^{\p\p}_x)_x$ and $G^{\p\p}_x = \SO(W^{\p\p}_x)_x\times \SO(W^{\p\p}_x)_x$. Set $\xi_x = \xi|_{H_x(F)}$. Then there are following decompositions 
\begin{align*}
\SO(V)_{x} = \SO(V^{\p}_{x})\times \SO(W^{\p\p}_{x})_{x}, \quad
\SO(W)_{x} &= \SO(W^{\p}_{x})\times \SO(W^{\p\p}_{x})_{x}, 
\\
H_{x} &= \SO(W)_{x}\ltimes N_{x},
\end{align*}
and 
\begin{align}\label{7.1.0}
Z_{\SO(V)}(x) &= \SO(V^{\p}_{x})\times Z_{\SO(W^{\p\p}_{x})}(x),
\\
Z_{\SO(W)}(x) &= \SO(W^{\p}_{x})\times Z_{\SO(W^{\p\p}_{x})}(x)\nonumber.
\end{align}
Moreover, it is not hard to verify that $\SO(W^{\p\p}_{x})_{x}$ commutes with $N_{x}$. Hence there are also decompositions
\begin{equation}\label{7.1.1}
G_{x} = G^{\p}_{x}\times G^{\p\p}_{x}, \quad 
H_{x} = H^{\p}_{x}\times H^{\p\p}_{x},
\end{equation}
where the inclusion $H_{x}\subset G_{x}$ is the product of the two inclusions $H^{\p}_{x}\subset G^{\p}_{x}$ and $H^{\p\p}_{x}\subset G^{\p\p}_{x}$. Since $\xi_{x}$ is trivial on $H^{\p\p}_{x}$, there is a further decomposition
$$
(G_{x}, H_{x},\xi_{x}) = 
(G^{\p}_{x}, H^{\p}_{x}, \xi^{\p}_{x})
\times (G^{\p\p}_{x}, H^{\p\p}_{x},1),
$$
where $\xi^{\p}_{x} = \xi_{|H^{\p}_{x}}$. Note that the triple $(G^{\p}_{x}, H^{\p}_{x}, \xi^{\p}_{x})$ coincides with the GGP triple associated to the admissible pair $(V^{\p}_{x}, W^{\p}_{x})$. The second triple $(G^{\p\p}_{x}, H^{\p\p}_{x},1)$ is also of a particular shape: the group $G^{\p\p}_{x}$ is the product of two copies of $H^{\p\p}_{x}$ and the inclusion $H^{\p\p}_{x}\subset G^{\p\p}_{x}$ is the diagonal one. Following \cite[11.1]{beuzart2015local} such a triple is called an \emph{Arthur triple}. Finally note that although $x\in \SO(W)_{\ss}(F)$, there is a decomposition similar to (\ref{7.1.1}) for any $x\in H_{\ss}(F)$ (simply conjugated $x$ inside $H(F)$ to an element in $\SO(W)_{\ss}(F)$) and that if $x,y\in H_{\ss}(F)$ are $H(F)$-conjugate, then there are natural isomorphisms of triples
$$
(G^{\p}_{x}, H^{\p}_{x}, \xi^{\p}_{x})
\simeq (G^{\p}_{y}, H^{\p}_{y}, \xi^{\p}_{y}),
\quad 
(G^{\p\p}_{x}, H^{\p\p}_{x},1) \simeq (G^{\p\p}_{y}, H^{\p\p}_{y},1)
$$
well-defined up to inner automorphisms by $H^{\p}_{x}(F)$ and $H^{\p\p}_{x}(F)$ respectively.

Let $x\in H_{\ss}(F)$. Denote by $\Gam(H)$, $\Gam(H_{x})$, $\Gam(G)$ and $\Gam(G_{x})$ the sets of semi-simple conjugacy classes in $H(F), H_{x}(F), G(F)$ and $G_{x}(F)$ respectively and they are equipped with topologies.

By the definition of GGP triples, the two canonical maps $\Gam(H_{x})\to \Gam(G_{x})$ and $\Gam(H)\to \Gam(G)$ are injective. Since these maps are continuous and proper (\cite[Section~1.7]{beuzart2015local}) and $\Gam(G_{x})$, $\Gam(H)$, $\Gam(G)$ are all Hausdorff and locally compact, it turns out that $\Gam(H_{x})\to \Gam(G_{x})$ and $\Gam(H) \to \Gam(G)$ are closed embeddings.

We are going to define a subset $\Gam(G,H)$ of $\Gam(H)$ as follows: $x\in \Gam(G,H)$ if and only if $H^{\p\p}_{x}$ is an anisotropic torus (and hence $G^{\p\p}_{x}$ also). Since $\Gam(H)\to \Gam(G)$ is a closed embedding, $\Gam(G,H)$ can also be viewed as a subset of $\Gam(G)$. Notice that $\Gam(G,H)$ is a subset of $\Gam_{\mathrm{ell}}(G)$ containing $1$. We now equip $\Gam(G,H)$ with a topology, which is finer than the one induced from $\Gam(G)$, and a measure. For this, we need to give a more concrete description of $\Gam(G,H)$. Consider the following set $\udl{\CT}$ of subtori of $\SO(W)$: $T\in \udl{\CT}$ if and only if there exists a non-degenerate subspace $W^{\p\p} \subset W$ (possibly $W^{\p\p} = 0$) such that $T$ is a maximal elliptic subtorus of $\SO(W^{\p\p})$. For such a torus $T$, denote by $T_{\natural}$ the open Zariski subset of elements $t\in T$ which are regular in $\SO(W^{\p\p})$ acting with distinct eigenvalues on $W^{\p\p}$. Then $\Gam(G,H)$ is the set of conjugacy classes that meet
$
\bigcup_{T\in \udl{\CT}}
T_{\natural}(F)
$
(c.f. \cite[p.259]{beuzart2015local}). Let $W(H,T) = \mathrm{Norm}_{H}(T)/Z_{H}(T)$. Then, by definition there is a natural bijection
\begin{equation}\label{7.1.3}
\Gam(G,H) \simeq \bigsqcup_{T\in \CT}
T_{\natural}(F)/W(H,T).
\end{equation}
Now the RHS of (\ref{7.1.3}) has a natural topology and we transfer it to $\Gam(G,H)$. Moreover, we equip $\Gam(G,H)$ with the unique regular Borel measure such that
$$
\int_{\Gam(G,H)}
\vphi(x)dx
=\sum_{T\in \CT}
|W(H,T)|^{-1}
\nu(T)
\int_{T(F)}\vphi(t)dt
$$
for any $\vphi\in C_{c}(\Gam(G,H))$. Recall that $\nu(T)$ is the only positive factor such that the total mass of $T(F)$ for the measure $\nu(T)\ud t$ is one. Note that $1$ is an atom for the measure whose mass is equal to $1$ (this corresponds to the contribution of the trivial torus in the formula above).

More generally, for any $x\in H_{\ss}(F)$ we may construct a subset $\Gam(G_{x},H_{x})$ of $\Gam(G_{x})$ which is equipped with its own topology and measure as follows. By (\ref{7.1.1}) we have a decomposition $\Gam(G_{x}) = \Gam(G^{\p}_{x})\times \Gam(G^{\p\p}_{x})$. Since the triple $(G^{\p}_{x}, H^{\p}_{x}, \xi^{\p}_{x})$ is a GGP triple, the previous construction provides us with a space $\Gam(G^{\p}_{x}, H^{\p}_{x})$ of semi-simple conjugacy classes in $G^{\p}_{x}(F)$. On the other hand, we define $\Gam(G^{\p\p}_{x}, H^{\p\p}_{x})$ to be the image of $\Gam_{\ani}(H^{\p\p}_{x})$ (the set of anisotropic conjugacy classes in $H^{\p\p}_{x}(F)$, c.f. \cite[Section~1.7]{beuzart2015local}) by the diagonal embedding $\Gam(H^{\p\p}_{x})\subset \Gam(G^{\p\p}_{x})$. Following \cite[Section~1.7]{beuzart2015local}, the set $\Gam(G^{\p\p}_{x}, H^{\p\p}_{x}) = \Gam_{\ani}(H^{\p\p}_{x})$ can be equipped with a topology and a measure. Set
$$
\Gam(G_{x}, H_{x}) = \Gam(G^{\p}_{x}, H^{\p}_{x})
\times \Gam(G^{\p\p}_{x}, H^{\p\p}_{x})
$$
and we equip this set with the product of the topologies and the measures defined on $\Gam(G^{\p}_{x}, H^{\p}_{x})$ and $\Gam(G^{\p\p}_{x}, H^{\p\p}_{x})$. Note that $\Gam(G_{x},H_{x}) = \emptyset$ unless $x\in G(F)_{\mathrm{ell}}$ (because otherwise $\Gam_{\ani}(H^{\p\p}_{x}) = \emptyset$). 

There is a parallel Lie algebra variant, $\Gam^{\Lie}(G,H)$ is a subset of $\Gam(\Fg)$, again equipped with a topology and a measure, as follows. for any $X\in \Fs\Fo(W)_{\ss}(F)$, we have decompositions
\begin{equation}\label{7.1.5}
G_{X} = G^{\p}_{X}\times G^{\p\p}_{X}, \quad 
H_{X} = H^{\p}_{X}\times H^{\p\p}_{X}
\end{equation}
where
\begin{align*}
&G^{\p}_{X} = \SO(W^{\p}_{X})\times \SO(V^{\p}_{X}), \quad G^{\p\p}_{X} = \SO(W^{\p\p}_{X})_{X}\times \SO(W^{\p\p}_{X})_{X},
\\
&H^{\p}_{X} = \SO(W^{\p}_{X})\ltimes N_{X}, \qquad 
H^{\p\p}_{X} = \SO(W^{\p\p}_{X})_{X}
\end{align*}
for $W^{\p}_{X}, V^{\p}_{X}$ the kernels of $X$ acting on $W$ and $V$ respectively, $W^{\p\p}_{X}$ the image of $X$ in $W$ and $N_{X}$ the centralizer of $X$ in $N$. Again the decomposition (\ref{7.1.5}) still hold for every $X\in \Fh_{\ss}(F)$ and they depend on the choice of representatives in the conjugacy class of $X$ only up to an inner automorphism. We now define $\Gam^{\Lie}(G,H)$ to be the set of semi-simple conjugacy classes $X\in \Gam(\Fh)$ such that $H^{\p\p}_{X}$ is an anisotropic torus. Similarly we can identify $\Gam^{\Lie}(G,H)$ with a subset of $\Gam(\Fg)$. Notice that $\Gam^{\Lie}(G,H)$ is a subset of $\Gam_{\mathrm{ell}}(\Fg)$ containing $0$. Moreover, fixing a set of tori $\CT$ as before we have a identification
\begin{equation}\label{7.1.6}
\Gam^{\Lie}(G,H)  =
\bigsqcup_{T\in \CT}
\Ft_{\natural}(F)/W(H,T)
\end{equation}
where for $T\in \CT$, $\Ft_{\natural}$ denotes the Zariski open subset consisting of elements $X\in \Ft$ that are regular in $\Fs\Fo(W^{\p\p}_{T})$ and acting with distinct eigenvalues in $W^{\p\p}_{T}$. By the identification (\ref{7.1.6}), $\Gam^{\Lie}(G,H)$ inherits a natural topology. Moreover, we equip $\Gam^{\Lie}(G,H)$ with the unique regular Borel measure such that
$$
\int_{\Gam^{\Lie}(G,H)}
\vphi(X)\ud X = 
\sum_{T\in \CT}
|W(H,T)|^{-1}
\nu(T)
\int_{\Ft(F)}\vphi(X)\ud X
$$
for any $\vphi\in C_{c}(\Gam^{\Lie}(G,H))$. Note that $0$ is an atom for this measure whose associated mass is $1$. The following lemma establishes a link between $\Gam^{\Lie}(G,H)$ and $\Gam(G,H)$.

\begin{lem}\label{lem:7.1.2}
Let $\ome\subset \Fg(F)$ be a $G$-excellent open neighborhood of $0$ $($\cite[Section~3.3]{beuzart2015local}$)$ and set $\Ome = \exp(\ome)$. Then, the exponential map induces a topological isomorphism
$$
\ome\cap \Gam^{\Lie}(G,H) \simeq \Ome \cap \Gam(G,H)
$$
preserving measures.
\end{lem}
\begin{proof}
The proof of \cite[Lemma~11.1.2]{beuzart2015local} works verbatim.
\end{proof}

We are ready to introduce the geometric multiplicity formula and its Lie algebra variant.

\subsubsection{The germs $c_\theta$}

We first define the germ $c_{\theta}$ associated to a quasi-character $\theta$ of $G(F)$. Here we would like to mention that the germ $c_{\theta}$ we are considering are of different nature from the unitary group case considered in \cite{beuzart2015local}.

Let $\theta$ be a quasi-character of $G(F) =  \SO(W)\times \SO(V).$ For convenience we only consider the case when $G$ is quasi-split. Recall from Section \ref{sub:ggptriples:definitions} that the regular nilpotent orbits of $G(F)$ has the following description.
\begin{itemize}
\item If $m=\dim W\geq 4$ is even, then the regular nilpotent orbits of $G(F)$ are parametrized by $\CN^{W};$

\item If $d = \dim V\geq 4$ is even, then the regular nilpotent orbits of $G(F)$ are parametrized by $\CN^{V}$;

\item For other cases, $G(F)$ has a unique regular nilpotent orbit $\CO_{\reg}$.
\end{itemize}

In each situation, for any $\nu\in \CN^{W}$ (resp. $\nu\in \CN^{V}$), let $\CO_{\nu}$ be the associated regular nilpotent orbit of $G$, which admits explicit description as in Section \ref{sub:ggptriples:definitions}.

We first consider the case when $G$ has only one regular nilpotent orbit $\CO_{\reg}$. In this case, we simply set $c_{\theta} = c_{\theta, \CO_{\reg}}$. 

Then we consider the case when $d = \dim V\geq 4$ is even. In this case the regular nilpotent orbits of $G(F)$ are parametrized by $\CN^{V}$. Since $\dim W$ is odd and $\SO(W)$ is quasi-split, we have that $q_{W,\an}$ is of dimension one. Let $q_{D}$ be the restriction of the quadratic form $q_{V}$ to the associated line $D$. Then $(V,q_{V})$ has the same anisotropic kernel as $q_{W,\an}\oplus q_{D}$. In particular, $\nu_{0} = q(v_{0})\in \CN^{V}$. Then we set $c_{\theta} = c_{\theta, \CO_{\nu_{0}}}$. 

Finally when $m=\dim W\geq 4$ is even, the regular nilpotent orbits of $G(F)$ are parametrized by the set $\CN^{W}$. Since $G(F)$ is quasi-split, the quadratic form $q_{W,\an}\oplus q_{D}$ is actually split. In particular, $-\nu_{0} = -q(v_{0})\in \CN^{W}$. Then set $c_{\theta} = c_{\theta, \CO_{-\nu_{0}}}$.

\subsubsection{Discriminants}

Set
$$
\Del(x) = D^{G}(x)D^{H}(x)^{-2}
$$
for any $x\in H_{\ss}(F)$ where $D^{G}(x) = |\det(1-\Ad(x))_{|\Fg/\Fg_{x}}|$ and $D^{H}(x) = |\det (1-\Ad(x))_{|\Fh/\Fh_{x}}|$. Then (\cite[Lemma~3.1.1]{beuzart2015local}) 
\begin{equation}\label{7.2.1}
\Del(x) = |\det(1-x)_{|W^{\p\p}_{x}}|
\end{equation}
for any $x\in H_{\ss}(F)$. Similarly define
$$
\Del(X) =D^{G}(X)D^{H}(X)^{-2}
$$
for any $X\in \Fh_{\ss}(F)$ and
\begin{equation}\label{7.2.2}
\Del(X) = |\det(X_{|W^{\p\p}_{X}})|
\end{equation}
for any $X\in \Fh_{\ss}(F)$. Let $\ome\subset \Fg(F)$ be a $G$-excellent open neighborhood of $0$. Then $\ome\cap \Fh(F)$ is an $H$-excellent open neighborhood of $0$ (cf. the remark at the end of \cite[Section~3.3]{beuzart2015local}). We may thus set
$$
j^{H}_{G}(X)= j^{H}(X)^{2}j^{G}(X)^{-1}
$$
for any $X\in \ome\cap \Fh(F)$. By \cite[3.3.1]{beuzart2015local},
\begin{equation}\label{7.2.3}
j^{H}_{G}(X) = \Del(X)\Del(e^{X})^{-1}
\end{equation}
for any $X\in \ome \cap \Fh_{\ss}(F)$. Note that $j^{H}_{G}$ is a smooth positive and $H(F)$-invariant function on $\ome \cap \Fh(F)$. It actually extends (not uniquely) to a smooth, positive and $G(F)$-invariant function on $\ome$. This can be seen as follows. We can embed the groups $H_{1} = \SO(W)$ and $H_{2} = \SO(V)$ into a GGP triple $(G_{1},H_{1},\xi_{1})$ and $(G_{2},H_{2},\xi_{2})$. Then the function $X=(X_{W},X_{V})\in \ome \to j^{H_{1}}_{G_{1}}(X_{W})^{1/2}j^{H_{2}}_{G_{2}}(X_{V})^{1/2}$ can be seen as such an extension through using the equality (\ref{7.2.3}). We will always assume that such an extension has been chosen and we will still denote it by $j^{H}_{G}$.

Let $x\in H_{\ss}(F)$. Define
$$
\Del_{x}(y) = D^{G_{x}}(y)D^{H_{x}}(y)^{-2}
$$
for any $y\in H_{x,\ss}(F)$. On the other hand, since the triple $(G^{\p}_{x}, H^{\p}_{x}, \xi^{\p}_{x})$ is a GGP triple, the previous construction yields a function $\Del^{G^{\p}_{x}}$ on $H^{\p}_{x,\ss}(F)$. Then
\begin{equation}\label{7.2.4}
\Del_{x}(y) = \Del^{G^{\p}_{x}}(y^{\p})
\end{equation}
for any $y = (y^{\p}, y^{\p\p})\in H_{x,\ss}(F) = H^{\p}_{x,\ss}(F)\times H^{\p\p}_{x,\ss}(F)$.

Let $\Ome_{x}\subset G_{x}(F)$ be a $G$-good open neighborhood of $x$. Then, $\Ome_{x}\cap H(F)\subset H_{x}(F)$ is a $H$-good open neighborhood of $x$. This allows us to set
$$
\eta^{H}_{G,x}(y)  = \eta^{H}_{x}(y)^{2}\eta^{G}_{x}(y)^{-1}
$$
for any $y\in \Ome_{x}\cap H(F)$. By \cite[3.2.4]{beuzart2015local}
\begin{equation}\label{7.2.5}
\eta^{H}_{G,x}(y) = \Del_{x}(y) \Del(y)^{-1}
\end{equation}
for any $y\in \Ome_{x}\cap H_{\ss}(F)$. Note that $\eta^{H}_{G,x}$ is a smooth, positive and $H_{x}(F)$-invariant function on $\Ome_{x}\cap H(F)$. It also extends (not uniquely) to a smooth, positive and $G_{x}(F)$-invariant function on $\Ome_{x}$. We will always still denote by $\eta^{H}_{G,x}$ such an extension.

The definitions of the distributions $m_{\geom}$ and $m^{\Lie}_{\geom}$ are contained in the following proposition.

\begin{pro}\label{pro:7.2.1}
\begin{enumerate}
\item Let $\theta\in QC(G)$. Then the following integral is absolutely convergent for any $\Re(s)>0$
$$
\int_{\Gam(G,H)}
D^{G}(x)^{1/2}c_{\theta}(x)\Del(x)^{s-1/2}\ud x
$$
and the following limit
$$
m_{\geom}(\theta) :=
\lim_{s\to 0^+}
\int_{\Gam(G,H)}
D^{G}(x)^{1/2}c_{\theta}(x)\Del(x)^{s-1/2}\ud x.
$$
exists. Similarly, for any $x\in H_{\ss}(F)$ and $\theta_{x}\in QC(G_{x})$, the following integral is absolutely convergent for any $\Re(s)>0$
$$
\int_{\Gam(G_{x},H_{x})}
D^{G_{x}}(y)^{1/2}
c_{\theta_{x}}(y)
\Del_{x}(y)^{s-1/2}\ud y
$$
and the following limit
$$
m_{x,\geom}(\theta_{x}) :=
\lim_{s\to 0^+}
\int_{\Gam(G_{x},H_{x})}
D^{G_{x}}(y)^{s-1/2}
c_{\theta_{x}}(y)
\Del_{x}(y)^{s-1/2}\ud y
$$
exists. In particular, $m_{\geom}$ is a continuous linear form on $QC(G)$ and for any $x\in H_{\ss}(F)$, $m_{x,\geom}$ is a continuous linear form on $QC(G_{x}).$

\item Let $x\in H_{\ss}(F)$ and let $\Ome_{x}\subset G_{x}(F)$ be a $G$-good open neighborhood of $x$ and set $\Ome = \Ome^{G}_{x}$. Then, if $\Ome_{x}$ is sufficiently small,
$$
m_{\geom}(\theta) = 
\frac{2}{[Z_{H^{+}}(x)(F):H_{x}(F)]}
m_{x,\geom}((\eta^{H}_{G,x})^{1/2}\theta_{x,\Ome_{x}})
$$
for any $\theta\in QC_{c}(\Ome)$.

\item Let $\theta\in QC_{c}(\Fg)$. Then the following integral is absolutely convergent for any $\Re(s)>0$
$$
\int_{\Gam^{\Lie}(G,H)}
D^{G}(X)^{1/2}c_{\theta}(X)\Del(X)^{s-1/2}\ud X
$$
and the following limit
$$
m^{\Lie}_{\geom}(\theta):=
\lim_{s\to 0^+}
\int_{\Gam^{\Lie}(G,H)}
D^{G}(X)^{1/2}c_{\theta}(X)\Del(X)^{s-1/2}\ud X
$$
exists.
In particular, $m^{\Lie}_{\geom}$ is a continuous linear form on $QC_{c}(\Fg)$ that extends continuously to $SQC(\Fg)$ and
$$
m^{\Lie}_{\geom}(\theta_{\lam}) = 
|\lam|^{\del(G)/2}
m^{\Lie}_{\geom}(\theta)
$$
for any $\theta\in SQC(\Fg)$ and $\lam\in F^{\times}$. Here we recall that $\theta_{\lam}(X) = \theta(\lam^{-1}X)$ for any $X\in \Fg_{\reg}(F)$.

\item Let $\ome\subset \Fg(F)$ be a $G$-excellent open neighborhood of $0$ and set 
$\Ome = \exp(\ome)$. Then
$$
m_{\geom}(\theta) = m^{\Lie}_{\geom}((j^{H}_{G})^{1/2}\theta_{\ome})
$$
for any $\theta\in QC_{c}(\Ome)$.
\end{enumerate}
\end{pro}
\begin{proof}
(1) 
When $F$ is $p$-adic, the proof follows from \cite[\S7.3-\S7.7]{waldspurger10}.

When $F=\BR$, the only difference between our situation and that of \cite[Proposition~11.2.1~(i)]{beuzart2015local} is the last part (\cite[(11.2.22)]{beuzart2015local}) of the proof in \cite{beuzart2015local}, where in the unitary group situation the regular nilpotent orbits can be permuted by scaling (\cite[Section~6.1]{beuzart2015local}), which is not the case for special orthogonal groups. 
However, following the proof of \cite[Proposition~11.2.1~(i)]{beuzart2015local}, indeed, the problem can be reduced to show the limit 
\begin{equation}\label{eq:cov}
\lim_{s\to 0^+}m^\Lie_{\geom,s}(\theta)
\end{equation}
exist for any $\theta = \vphi\wh{j}(\CO,\cdot)$, where $\vphi\in \CC^\infty (\Fg)^G$ is an invariant smooth function compactly supported modulo conjugation and equals $1$ in some neighborhood of $0$, and $\CO\in \nil_\reg(\Fg)$. 

The proof of \eqref{eq:cov} follows from the same argument as in \cite[\S7.3-\S7.7]{waldspurger10}, except there are some small remarks needed:
\begin{itemize}
\item In the proof of \cite[Lemme~7.4]{waldspurger10}, change $\Ome_s$ to be the subset of $\ome$ with coordinates $(\lam_j)_{j=1,...,m}$ in the basis $(e_j)_{j=1,...,m}$ satisfying $\frac{1}{2}\leq |\lam_j|\leq 1$, and change $\vpi_{F}^{2k}$ to $(\frac{1}{2})^{2k}$;

\item The proof of \cite[Lemme~7.6]{waldspurger10} does not work in general in the archimedean case, as the germ expansions for quasi-characters is not an exact identity parametrized by nilpotent orbits. However, the quasi-character $\theta$ appearing in \eqref{eq:cov} is equal to $\vphi\wh{j}(\CO,\cdot)$. In particular the proof of \cite[Lemme~7.6]{waldspurger10} indeed applies to our situation.
\end{itemize}

(2) We notice that $\theta_{x,\Ome_{x}}\in QC_{c}(\Ome_{x})^{Z_{G}(x)}$ (cf. \cite[Proposition~4.4.1 (iii)]{beuzart2015local}). Hence the proof follows from \cite[Lemme~8.3]{waldspurger10} directly.

(3) The absolute convergence can be proved in parallel with (1) using \cite[Lemma~B.1.2(ii)]{beuzart2015local}. The existence of the limit follows from the same argument as in \cite[Proposition~11.2.1~(iii)]{beuzart2015local}. The homogeneity of $m^{\Lie}_{\geom}(\theta_{\lam})$  follows from the direct computation as in \cite[Proposition~11.2.1]{beuzart2015local}. Finally by \cite[Proposition~4.6.1(ii)]{beuzart2015local}, $m^{\Lie}_{\geom}$ can be extended to $SQC(\Fg)$.

(4) The proof is parallel to (2), where we apply Lemma \ref{lem:7.1.2} to our situation.
\end{proof}

\subsection{Geometric multiplicity and parabolic induction}\label{sec:7.3}
Let $L$ be a Levi subgroup of $G$. Then as in Section \ref{sec:temperedint:4}, $L$ can be decomposed as a product 
$$
L = L^{\GL}\times \wt{G}
$$
where $L^{\GL}$ is a product of general linear groups over $F$ and $\wt{G}$ belongs to a GGP triple $(\wt{G}, \wt{H}, \wt{\xi})$ which is well-defined up to $\wt{G}(F)$ conjugation. In particular, there is continuous linear form $m^{\wt{G}}_{\geom}$ on $QC(\wt{G})$. Define a continuous linear form $m^{L}_{\geom}$ on $QC(L) = QC(L^{\GL})\wh{\otimes}_{p}QC(\wt{G})$ by setting 
$$
m^{L}_{\geom}
(\theta^{\GL}\otimes \wt{\theta}) = 
m^{\wt{G}}_{\geom}(\wt{\theta})c_{\theta^{\GL}}(1)
$$
for any $\theta^{\GL}\in QC(L^{\GL})$ and $\wt{\theta}\in QC(\wt{G})$. 

Before stating the relation between geometric multiplicity and parabolic induction, we note the following lemma, which is the analogue result for \cite[Lemme~2.3]{waldspurgertemperedggp} associated with regular nilpotent germs in the archimedean case. For the notation appearing in the lemma below we refer to \cite[Lemme~2.3]{waldspurgertemperedggp}.

\begin{lem}
Suppose $\theta^L\in QC(L)$ and $\theta = i^G_L\theta^L$ (c.f. \cite[(3.4.2)]{beuzart2015local}). Then
\begin{enumerate}
\item $\theta\in QC(G)$;

\item Suppose $x\in G_\ss(F)$ and $\CO\in \nil_\reg(\Fg_x)$, then 
$$
D^G(x)^{1/2}c_{\theta,\CO}(x) = 
\sum_{y\in \CX^L(x)}
\sum_{g\in \Gam_y/G_x(F)}
\sum_{\CO^\p\in \nil_\reg(\Fl_y)}
\frac{[g\CO:\CO^\p]}{[Z_L(y):L_y(F)]}
D^L(y)^{1/2}c_{\theta^L, \CO^\p}(y).
$$
\end{enumerate}
\end{lem}
\begin{proof}
(1) follows from \cite[Proposition~4.7.1~(i)]{beuzart2015local}.

For (2), through the limit formula in \cite[p.98, Section~4.5]{beuzart2015local} together with the explicit induction formula (\cite[(3.4.2)]{beuzart2015local}), we can reduce $\theta^L\in QC(L)$ to $\theta^L\in QC(\Fl)$ which is finite linear combinations of distributions $\{\wh{j}(\CO^\p,\cdot)|\quad \CO^\p\in \nil_\reg(\Fl)\}$. Since the parabolic induction of finite linear combinations in $\{\wh{j}(\CO^\p,\cdot)|\quad \CO^\p\in \nil_\reg(\Fl)\}$ is finite linear combinations of $\{\wh{j}(\CO,\cdot)|\quad \CO\in \nil_\reg(\Fg) \}$, the proof of \cite[Lemme~2.3]{waldspurgertemperedggp} applies verbatim.
\end{proof}

Now combing with the lemma above, the following lemma follows from the same argument as \cite[Lemme~7.2]{waldspurgertemperedggp}.
\begin{lem}\label{lem:7.3.1}
Let $\theta^{L}\in QC(L)$ and set $\theta = i^{G}_{L}(\theta^{L})$. Then
$$
m_{\geom}(\theta) = m^{L}_{\geom}(\theta^{L}).
$$
\end{lem}

\subsection{The theorems}\label{sec:7.4}
Set
\begin{align*}
&J_{\geom}(f) = m_{\geom}(\theta_{f}), \quad f\in \CC_{\scusp}(G),
\\
&m_{\geom}(\pi) = m_{\geom}(\theta_{\pi}) , \quad \pi\in \Temp(G),
\\
&J^{\Lie}_{\geom}(f) = m^{\Lie}_{\geom}(\theta_{f}), \quad f\in \CS_{\scusp}(\Fg).
\end{align*}

\begin{thm}\label{thm:7.4.1}
$$
J(f)= J_{\geom}(f)
$$
for any $f\in \CC_{\scusp}(G)$.
\end{thm}

\begin{thm}\label{thm:7.4.2}
$$
m(\pi) = m_{\geom}(\pi)
$$
for any $\pi\in \mathrm{Temp}(G)$.
\end{thm}

\begin{thm}\label{thm:7.4.3}
$$
J^{\Lie}(f) = J^{\Lie}_{\geom}(f)
$$
for any $f\in \CS_{\scusp}(\Fg)$.
\end{thm}

The proof is by induction on $\dim (G)$ (the case $\dim(G) =1$ is trivial). Hence, we make the following induction hypothesis.

(HYP) Theorem \ref{thm:7.4.1}, Theorem \ref{thm:7.4.2} and Theorem \ref{thm:7.4.3} hold for any GGP triples $(G^{\p}, H^{\p}, \xi^{\p})$ such that $\dim (G^{\p})<\dim (G)$.

\subsubsection{Equivalence of theorem \ref{thm:7.4.1} and theorem \ref{thm:7.4.2}}\label{sec:7.5}
\begin{pro}
Assume the induction hypothesis $\mathrm{(HYP)}$.
\begin{enumerate}
\item Let $\pi\in R_{ind}(G)$. Then
$$
m(\pi) = m_{\geom}(\pi).
$$

\item For any $f\in \CC_{\scusp}(G)$, there is the following equality
$$
J(f) = J_{\geom}(f) + 
\sum_{\pi\in \CX_{\mathrm{ell}}(G)}
D(\pi)\wh{\theta}_{f} (\pi)
(m(\wb{\pi}) - m_{\geom}(\wb{\pi})),
$$
and the sum on the RHS is absolutely convergent.

\item Theorem \ref{thm:7.4.1} and Theorem \ref{thm:7.4.2} are equivalent.

\item There exists a unique continuous linear form $J_{\qc}$ on $QC(G)$ such that
\begin{itemize}
\item $J(f) = J_{\qc}(\theta_{f})$ for any $f\in \CC_{\scusp}(G)$;

\item $\mathrm{Supp}(J_{\qc}) =  G(F)_{\mathrm{ell}}$.

\end{itemize}
\end{enumerate}
\end{pro}
\begin{proof}
The proof of \cite[Proposition~11.5.1]{beuzart2015local} works verbatim.
\end{proof}

\subsubsection{Semisimple descent and the support of $J_{\qc}-m_{\geom}$}\label{sec:7.6}
\begin{pro}\label{pro:7.6.1}
Assume the induction hypothesis $\mathrm{(HYP)}$. Let $\theta\in QC(G)$ and assume that $1\notin \mathrm{Supp}(\theta)$. Then
$$
J_{\qc}(\theta) = m_{\geom}(\theta).
$$
\end{pro}
\begin{proof}
Since both $J_{\qc}$ and $m_{\geom}$ are both supported on $\Gam_{\mathrm{ell}}(G)$, by partition of unity, we only need to prove the equality for $\theta\in QC_{c}(\Ome)$ where $\Ome$ is a completely $G(F)$-invariant open subset of $G(F)$ of the form $\Ome^{G}_{x}$ for some $x\in G(F)_{\mathrm{ell}}, x\neq 1$, and some $G$-good open neighborhood $\Ome_{x}\subset G_{x}(F)$ of $x$. Moreover, we may take $\Ome_{x}$ as small as we want. In particular, we will assume that $\Ome_{x}$ is relatively compact modulo conjugation.

Assume first that $x$ is not conjugate to any element of $H_{\ss}(F)$. Then since $\Gam(H)$ is closed in $\Gam(G)$, if $\Ome_{x}$ is chosen sufficiently small, we would have $\Ome \cap \Gam(H) = \emptyset$. In this case, both sides of the equality are zero for any $\theta\in QC_{c}(\Ome)$.

Assume now that $x$ is conjugate to some element of $H_{\ss}(F)$. We may simply assume $x\in H_{\ss}(F)$. Then,
\begin{align*}
G_{x}  &= G^{\p}_{x} \times G^{\p\p}_{x}, \quad
G^{\p\p}_{x} = H^{\p\p}_{x}\times H^{\p\p}_{x}.
\end{align*}
By \eqref{7.1.0},
\begin{align*}
Z_{G}(x) = G^{\p}_{x}\times Z_{\SO(W^{\p\p}_{x})}(x) \times  Z_{\SO(W^{\p\p}_{x})}(x).
\end{align*}
Shrinking $\Ome_{x}$ if necessary, we may assume that $\Ome_{x}$ decomposes as a product
$$
\Ome_{x}  = \Ome^{\p}_{x}\times (\Ome^{\p\p}_{x}\times \Ome^{\p\p}_{x})
$$
where $\Ome^{\p}_{x}\subset G^{\p}_{x}(F)$ (resp. $\Ome^{\p\p}_{x}\subset H^{\p\p}_{x}(F)$) is open and completely $G^{\p}_{x}$-invariant (resp. completely $Z_{\SO(W^{\p\p}_{x})}(x)$-invariant). Note that $x$ is elliptic in both $G^{\p}_{x}$ and $H^{\p\p}_{x}$. Hence, by \cite[Corollary~5.7.2(i)]{beuzart2015local}, shrinking $\Ome_{x}$ if necessary, we may assume that the linear maps
\begin{align*}
&f^{\p}_{x}\in \CS_{\scusp}(\Ome^{\p}_{x}) \to \theta_{f^{\p}_{x}}\in QC_{c}(\Ome^{\p}_{x})
\\
&f^{\p\p}_{x}\in \CS_{\scusp}(\Ome^{\p\p}_{x})\to \theta_{f^{\p\p}_{x}}\in QC_{c}(\Ome^{\p\p}_{x})
\end{align*}
have dense image. Since (\cite[Proposition~4.4.1(v)]{beuzart2015local})
\begin{align*}
&QC_{c}(\Ome_{x}) = 
QC_{c}(\Ome^{\p}_{x})
\wh{\otimes}_{p}
QC_{c}(\Ome^{\p\p}_{x})
\wh{\otimes}_{p}
QC_{c}(\Ome^{\p\p}_{x}),
\end{align*}
\begin{align*}
&QC_{c}(\Ome_{x})^{Z_{G}(x)} = 
QC_{c}(\Ome^{\p}_{x})
\wh{\otimes}_{p}
QC_{c}(\Ome^{\p\p}_{x})^{Z_{\SO(W^{\p\p}_{x})}(x)}
\wh{\otimes}_{p}
QC_{c}(\Ome^{\p\p}_{x})^{Z_{\SO(W^{\p\p}_{x})}(x)},
\end{align*}
and $J_{\qc}$ and $m_{\geom}$ are continuous linear forms on $QC_{c}(\Ome)$, we only need to prove the equality of the proposition for quasi-characters $\theta\in QC_{c}(\Ome)$ such that $\theta_{x,\Ome_{x}} = \theta_{f_{x}}\in QC_{c}(\Ome_{x})^{Z_{G}(x)}$ for some $f_{x}\in \CS_{\scusp}(\Ome_{x})$ which further decomposes as a tensor product $f_{x} = f^{\p}_{x}\otimes (f^{\p\p}_{x,1}\otimes f^{\p\p}_{x,2})$ where $f^{\p}_{x}\in \CS_{\scusp}(\Ome^{\p}_{x})$ and $f^{\p\p}_{x,1}, f^{\p\p}_{x,2}\in \CS_{\scusp}(\Ome^{\p\p}_{x})$. Up to translation we may assume that the functions $f^{\p\p}_{x,1}$ and $f^{\p\p}_{x,2}$ are invariant under conjugation by the Weyl elements showing up in $Z_{\SO(W^{\p\p}_{x})}(x)/H^{\p\p}_{x}(F)$.

Consider a map as in \cite[Proposition~5.7.1]{beuzart2015local}, and set $f = \wt{f}_{x}\in \CS_{\scusp}(\Ome)$. Then
$$
(\theta_{f})_{x,\Ome_{x}} = \sum_{z\in Z_{G}(x)/G_{x}(F)}{}^z\theta_{f_{x}}
=[Z_{G}(x):G_{x}]\theta_{f_{x}}.
$$
In particular up to translation we can assume that
$$
(f)_{x,\Ome_{x}} = \frac{1}{[Z_{G}(x):G_{x}]}f_{x}.
$$
We also have 
\begin{equation}\label{7.6.1}
J(f) = J_{\qc}(\theta_{f}) = J_{\qc}(\theta).
\end{equation}
We denote by $J^{G^{\p}_{x}}$ the continuous linear form on $\CC_{\scusp}(G^{\p}_{x})$ associated to the GGP triple $(G^{\p}_{x}, H^{\p}_{x}, \xi^{\p}_{x})$. Also we denote by $J^{A,H^{\p\p}_{x}}$ the continuous bilinear form on $\CC_{\scusp}(H^{\p\p}_{x})$ introduced in \cite[Section~5.5]{beuzart2015local} (where we replace $G$ by $H^{\p\p}_{x}$). We show the following
\begin{num}
\item\label{7.6.2}
If $\Ome_{x}$ is sufficiently small, we have 
$$
J(f) = \frac{2}{[Z_{H^{+}}(x)(F):H_{x}(F)]}J^{G^{\p}_{x}}(f^{\p}_{x})J^{A,H^{\p\p}_{x}}
((\eta^{H}_{x,G})^{1/2}f^{\p\p}_{x,1}, f^{\p\p}_{x,2}).
$$
\end{num}
The intersection $\Ome_{x}\cap H_{x}(F)\subset H_{x}(F)$ is a $H$-good open neighborhood of $x$ (cf. the remark at the end of \cite[Section~3.2]{beuzart2015local}). By \cite[3.2.5]{beuzart2015local} and \cite[Lemme~8.2]{waldspurger10},
\begin{align}\label{7.6.3}
&J(f)= 
\int_{H(F)\bs G(F)}
\int_{H(F)}{}^gf(h)\xi(h)\ud h\ud g
\\
&=
\frac{2}{[Z_{H^{+}}(x)(F):H_{x}(F)]}\int_{H_{x}(F)\bs G(F)}
\int_{H_{x}(F)}
\eta^{H}_{x}(h_{x}){}^{g}f(h_{x})\xi_{x}(h_{x})\ud h_{x}\ud g \nonumber
\\
&=
\frac{2}{[Z_{H^{+}}(x)(F):H_{x}(F)]}
\int_{H(F)\bs G(F)}
\int_{H_{x}(F)\bs H(F)}\nonumber
\\
&
\int_{H_{x}(F)}
\eta^{H}_{x,G}(h_{x})^{1/2} ({}^{hg}f)_{x,\Ome_{x}}(h_{x})\xi_{x}(h_{x})\ud h_{x}\ud h\ud g. \nonumber
\end{align}
Assume one moment that the exterior double integral above is absolutely convergent. Then
\begin{align*}
&J(f)
= 
\frac{2}{[Z_{H^{+}}(x)(F):H_{x}(F)]}
\int_{H_{x}(F)\bs G(F)}
\int_{H_{x}(F)}
\eta^{H}_{x,G}(h_{x})^{1/2}
({}^g f)_{x,\Ome_{x}}(h_{x})\xi_{x}(h_{x})\ud h_{x}\ud g
\\
&=
\frac{2}{[Z_{H^{+}}(x)(F):H_{x}(F)]}
\int_{Z_{G}(x)(F)\bs G(F)}
\int_{H_{x}(F)\bs Z_{G}(x)(F)}
\\
&\int_{H_{x}(F)}
\eta^{H}_{x,G}(h_{x})^{1/2}
({}^gf)_{x,\Ome_{x}}
(g^{-1}_{x}h_{x}g_{x})
\xi_{x}(h_{x})\ud h_{x}\ud g_{x}\ud g.
\end{align*}
Introduce a function on $Z_{G}(x)(F)\bs G(F)$ as in \cite[Proposition~5.7.1]{beuzart2015local}. Let $g\in G(F)$. Up to translating $g$ by an element of $Z_{G}(x)(F)$, we may assume that $({}^gf)_{x,\Ome_{x}} = \frac{1}{[Z_{G}(x):G_{x}]}\alp(g)f_{x}$. Then the interior integral above decomposes as
\begin{align*}
&\frac{\alp(g)}{[Z_{G}(x):G_{x}]}
\int_{H_{x}(F)\bs Z_{G}(x)(F)}
\int_{H_{x}(F)}
\eta^{H}_{x,G}(h_{x})^{1/2}
f_{x}(g^{-1}_{x}h_{x}g_{x})
\xi_{x}(h_{x})\ud h_{x}\ud g_{x}
\\
&=
\frac{\alp(g)}{[Z_{G}(x):G_{x}]}
\int_{H^{\p}_{x}(F)\bs G^{\p}_{x}(F)}
\int_{H^{\p}_{x}(F)}
f^{\p}_{x}(g^{\p -1}_{x}h^{\p}_{x}g^{\p}_{x})
\xi^{\p}_{x}(h^{\p}_{x})\ud h^{\p}_{x}\ud g^{\p}_{x}
\\
&\times
\int_{Z_{\SO(W^{\p\p}_{x})(x)(F)}}
\int_{Z_{\SO(W^{\p\p}_{x})(x)(F)}}
\eta^{H}_{x,G}(h^{\p\p}_{x})^{1/2}
f^{\p\p}_{x,1}
(g^{\p\p -1}_{x}h^{\p\p}_{x}g^{\p\p}_{x})
f^{\p\p}_{x,2}(h^{\p\p}_{x})\ud h^{\p\p}_{x}\ud g^{\p\p}_{x}
\\
&=
\frac{\alp(g)[Z_{\SO(W^{\p\p}_{x})(x)}(F): H_{x}^{\p\p}(F)]^{2}}{[Z_{G}(x):G_{x}]}
\int_{H^{\p}_{x}(F)\bs G^{\p}_{x}(F)}
\int_{H^{\p}_{x}(F)}
f^{\p}_{x}(g^{\p -1}_{x}h^{\p}_{x}g^{\p}_{x})
\xi^{\p}_{x}(h^{\p}_{x})\ud h^{\p}_{x}\ud g^{\p}_{x}
\\
&\times
\int_{H^{\p\p}_{x}(F)}
\int_{H^{\p\p}_{x}(F)}
\eta^{H}_{x,G}(h^{\p\p}_{x})^{1/2}
f^{\p\p}_{x,1}
(g^{\p\p -1}_{x}h^{\p\p}_{x}g^{\p\p}_{x})
f^{\p\p}_{x,2}(h^{\p\p}_{x})\ud h^{\p\p}_{x}\ud g^{\p\p}_{x},
\end{align*}
since by \eqref{7.1.0},
$$
[Z_{\SO(W^{\p\p}_{x})(x)}(F): H_{x}^{\p\p}(F)]^{2} = [Z_{G}(x):G_{x}],
$$ 
we finally can write the interior integral as 
\begin{align*}
&\alp(g)
\int_{H^{\p}_{x}(F)\bs G^{\p}_{x}(F)}
\int_{H^{\p}_{x}(F)}
f^{\p}_{x}(g^{\p -1}_{x}h^{\p}_{x}g^{\p}_{x})
\xi^{\p}_{x}(h^{\p}_{x})\ud h^{\p}_{x}\ud g^{\p}_{x}
\\
&\times
\int_{H^{\p\p}_{x}(F)}
\int_{H^{\p\p}_{x}(F)}
\eta^{H}_{x,G}(h^{\p\p}_{x})^{1/2}
f^{\p\p}_{x,1}
(g^{\p\p -1}_{x}h^{\p\p}_{x}g^{\p\p}_{x})
f^{\p\p}_{x,2}(h^{\p\p}_{x})\ud h^{\p\p}_{x}\ud g^{\p\p}_{x}.
\end{align*}
We recognize the two integrals above: the first one is $J^{G^{\p}_{x}}(f^{\p}_{x})$ and the second one is $J^{A,H^{\p\p}_{x}}((\del^{H}_{x,G})^{1/2}f^{\p\p}_{x,1}, f^{\p\p}_{x,2})$ (Notice that the center of $H^{\p\p}_{x}(F)$ is compact since $x$ is elliptic). By Theorem \ref{thm:covofJ} (2) and \cite[Theorem~5.5.1(ii)]{beuzart2015local} and since the function $\alp$ is compactly supported, this shows that the exterior double integral of the last line (\ref{7.6.3}) is absolutely convergent. Moreover, since we have 
$$
\int_{Z_{G}(x)(F)\bs G(F)}
\alp(g)dg = 1,
$$
this also proves (\ref{7.6.2}).

We assume from now on that $\Ome_{x}$ is sufficiently small so that (\ref{7.6.2}) holds. By the induction hypothesis (HYP), we have
$$
J^{G^{\p}_{x}}(f^{\p}_{x}) = m^{G^{\p}_{x}}_{\geom}(\theta_{f^{\p}_{x}}).
$$
On the other hand, by \cite[Theorem~5.5.1(iv)]{beuzart2015local}, we have 
$$
J^{H^{\p\p}_{x}}_{A}((\eta^{H}_{x,G})^{1/2}, f^{\p\p}_{x,1}, f^{\p\p}_{x,2}) = 
\int_{\Gam(H^{\p\p}_{x})}
\eta^{H}_{G,x}(y)^{1/2}
D^{G^{\p\p}_{x}}(y)^{1/2}
\theta_{f^{\p\p}_{x,1}}(y)
\theta_{f^{\p\p}_{x,2}}(y)\ud y.
$$
Notice that here both $f^{\p\p}_{x,1}$ and $f^{\p\p}_{x,2}$ are strongly cuspidal, hence the terms corresponding to $\Gam(H^{\p\p}_{x}) - \Gam_{\ani}(H^{\p\p}_{x})$ vanishes. Moreover, we can verify that
$$
m_{\geom,x}((\eta^{H}_{G,x})^{1/2}\theta_{f_{x}}) = 
m^{G^{\p}_{x}}_{\geom}(\theta_{f^{\p}_{x}})
\times \int_{\Gam_{\ani}(H^{\p\p}_{x})}
\eta^{H}_{x,G}(y)^{1/2}
D^{G^{\p\p}_{x}}(y)^{1/2}
\theta_{f^{\p\p}_{x,1}}(y)
\theta_{f^{\p\p}_{x,2}}(y)\ud y.
$$
Hence, by (\ref{7.6.1}), (\ref{7.6.2}) and Proposition \ref{pro:7.2.1} (2), we have 
$$
J_{\qc}(\theta) = J(f) = m_{\geom, x}((\eta^{H}_{G,x})^{1/2}\theta_{f_{x}}) = m_{\geom}(\theta).
$$
This ends the proof of the proposition.
\end{proof}

\subsubsection{Descent to the Lie algebra and equivalence of theorem \ref{thm:7.4.1} and theorem \ref{thm:7.4.3}}\label{sec:7.7}
Let $\ome\subset \Fg(F)$ be a $G(F)$-excellent open neighborhood of $0$ and set $\Ome= \exp(\ome)$. Recall that for any quasi-character $\theta\in QC(\Fg)$ and all $\lam\in F^{\times}$, $\theta_{\lam}$ denotes the quasi-character given by $\theta_{\lam}(X) = \theta(\lam^{-1}X)$ for any $X\in \Fg^{\rss}(F)$.

\begin{pro}\label{pro:7.7.1}
Assume the induction hypothesis \textbf{$(\mathrm{HYP})$}.
\begin{enumerate}
\item For any $f\in \CS_{\scusp}(\Ome)$,
$$
J(f) = J^{\Lie}((j^{H}_{G})^{1/2}f_{\ome}).
$$

\item There exists a unique continuous linear form $J^{\Lie}_{\qc}$ on $SQC(\Fg)$ such that
$$
J^{\Lie}(f) = J^{\Lie}_{\qc}(\theta_{f})
$$
for any $f\in \CS_{\scusp}(\Fg)$. Moreover,
$$
J^{\Lie}_{\qc}(\theta_{\lam}) = |\lam|^{\del(G)/2}J^{\Lie}_{\qc}(\theta)
$$
for any $\theta\in SQC(\Fg)$. Moreover,
$$
J^{\Lie}_{\qc}(\theta_{\lam}) = |\lam|^{\del(G)/2}J^{\Lie}_{\qc}(\theta)
$$
for any $\theta\in SQC(\Fg)$ and $\lam\in F^{\times}$.

\item Theorem \ref{thm:7.4.1} and Theorem \ref{thm:7.4.3} are equivalent.

\item Let $\theta\in SQC(\Fg(F))$ and assume that $0\notin \mathrm{Supp}(\theta)$. Then,
$$
J^{\Lie}_{\qc}(\theta) = m^{\Lie}_{\geom}(\theta).
$$
\end{enumerate}
\end{pro}
\begin{proof}
The proof follows from the same argument as \cite[Proposition~11.7.1]{beuzart2015local}.
\end{proof}

\subsubsection{End of the proof}\label{sec:7.8}
\begin{pro}\label{pro:7.8.1}
Assume the induction hypothesis $\mathrm{(HYP)}$, then
$$
J^{\Lie}_{\qc}(\theta) = m^{\Lie}_{\geom}(\theta)
$$
for any $\theta\in SQC(\Fg)$.
\end{pro}
\begin{proof}
We first establish the following fact.
\begin{num}
\item\label{7.8.1}
There exists constants $c_{\CO}, \CO\in \mathrm{Nil_{reg}}(\Fg)$, such that
$$
J^{\Lie}_{\qc}(\theta) - m^{\Lie}_{\geom}(\theta) = \sum_{\CO\in \mathrm{Nil_{reg}}(\Fg)}
c_{\CO}c_{\theta, \CO}(0)
$$
for any $\theta\in SQC(\Fg)$.
\end{num}
Let $\theta\in SQC(\Fg)$ be such that $c_{\theta, \CO}(0)= 0$ for any $\CO\in \mathrm{Nil_{reg}}(\Fg)$. We want to show that $J^{\Lie}_{\qc}(\theta) = m^{\Lie}_{\geom}(\theta)$. Let $\lam\in F^{\times}$ be such that $|\lam| \neq 1$. Denote by $M_{\lam}$ the operator on $SQC(\Fg)$ given by $M_{\lam}\theta = |\lam|^{-\del(G)/2}\theta_{\lam}$. Then, from \cite[Proposition~4.6.1 (i)]{beuzart2015local}, we may find $\theta_{1}, \theta_{2}\in SQC(\Fg)$ such that $\theta = (M_{\lam} - 1)\theta_{1}+\theta_{2}$ and $\theta_{2}$ is compactly supported away from $0$. By Proposition \ref{pro:7.7.1} (4), we have $J^{\Lie}_{\qc}(\theta_{2}) = m^{\Lie}_{\geom}(\theta_{2})$. On the other hand, by the homogeneity property of $J^{\Lie}_{\qc}$ and $m^{\Lie}_{\geom}$ (cf. Proposition \ref{pro:7.7.1} (2) and Proposition \ref{pro:7.2.1} (2)), we also have $J^{\Lie}_{\qc}((M_{\lam} - 1)\theta_{1}) = m^{\Lie}_{\geom}((M_{\lam} - 1)\theta_{1}) = 0$. This proves (\ref{7.8.1}).

To end the proof of the proposition, it remains to show that the coefficients $c_{\CO}$ for $\CO\in \mathrm{Nil_{reg}}(\Fg)$, are all zero. We separate the discussion to the following cases.
\begin{itemize}
\item
If $G$ is not quasi-split then there is nothing to prove. 

\item 
When $G$ has a unique regular nilpotent orbit, we follow the strategy of \cite[Section~11.9]{beuzart2015local} as follows. Fix a Borel subgroup $B\subset G$ and a maximal torus $T_{\qd}\subset B$, both of which are defined over $F$. Let $\Gam_{\qd}(\Fg)$ be the subset of $\Gam(\Fg)$ consisting of conjugacy classes that meet $\Ft_{\qd}(F)$. Recall that in Section \ref{sec:6.8}, we have defined a subset $\Gam(\Sig)\subset \Gam(\Fg)$. It consists of conjugacy classes of the semi-simple parts of elements in the affine subspace $\Sig(F)\subset \Fg(F)$ defined in Section \ref{sec:6.1}. We claim that
\begin{equation}\label{7.8.2}
\Gam_{\qd}(\Fg) \subset \Gam(\Sig).
\end{equation}
Up to $G(F)$ conjugation, we may assume that $B$ is a good Borel subgroup (c.f. Section \ref{sub:ggptriples:sphvar}). Then we have $\Fh\oplus \Fb = \Fg$ and therefore
\begin{equation}\label{7.8.3}
\Fh^{\perp}\oplus \Fu = \Fg
\end{equation}
where $\Fu$ denotes the unipotent radical of $\Fb$. Recall that $\Sig = \Xi+\Fh^{\perp}$. From (\ref{7.8.3}) we find that the restriction of the natural projection $\Fb \to \Ft_{\qd}$ to $\Sig\cap \Fb$ induces an affine isomorphism $\Sig\cap \Fb = \Ft_{\qd}$ and it implies (\ref{7.8.2}).

Let $\theta_{0}\in C^{\infty}_{c}(\Ft_{\qd,\reg})$ be $W(G,T_{\qd})$-invariant and such that
\begin{equation}\label{7.8.4}
\int_{\Ft_{\qd}(F)}
D^{G}(X)^{1/2}\theta_{0}(X)dX \neq 0.
\end{equation}
We may extend $\theta_{0}$ to a smooth invariant function on $\Fg^\rss(F)$, still denoted by $\theta_{0}$, which is zero outside $\Ft_{\qd,\reg}(F)^{G}$. Then $\theta_{0}$ is a compactly supported quasi-character. We consider its Fourier transform $\theta = \wh{\theta_{0}}$. By \cite[3.4.5, Lemma~4.2.3 (iii), Proposition~4.1.1 (iii)]{beuzart2015local}, $\theta$ is supported on $\Gam_{\qd}(\Fg)$. Since $\Gam_{\qd}(\Fg)\cap \Gam(G,H) = \{ 1 \}$, by definition of $m^{\Lie}_{\geom}$, we have 
$$
m^{\Lie}_{\geom}(\theta) = c_{\theta}(0).
$$
On the other hand, by \cite[Proposition~4.1.1 (iii), Lemma~4.2.3(iii), Proposition~4.5.1.2 (v)]{beuzart2015local}, we have 
\begin{equation}\label{7.8.5}
c_{\theta}(0) = 
\int_{\Gam(\Fg)}
D^{G}(X)^{1/2}
\theta_{0}(X)
c_{\wh{j}(X,\cdot)}(0)dX =  
\int_{\Gam_{\qd}(\Fg)}
D^{G}(X)^{1/2}\theta_{0}(X)dX.
\end{equation}
By definition of $J^{\Lie}_{\qc}$ and (\ref{7.8.2}), the last term is also equal to $J^{\Lie}_{\qc}(\theta)$. Hence we have $J^{\Lie}_{\qc}(\theta) = m^{\Lie}_{\geom}(\theta)$. Combining (\ref{7.8.4}) and (\ref{7.8.5}), we find that $c_{\theta}(0)\neq 0$ and it follows that we have proved the proposition.

\item
Assume now that $G$ is quasi-split with more than one regular nilpotent orbit parametrized by the set $\CN^{W}$ (i.e. $d=\dim W$ is even. The case for $\dim V$ even can be treated similarly). We use the notation from Section \ref{sub:germcal:explicitexample} then recall the following results from Lemma \ref{lem:germcal:Xpm}.
\begin{num}
\item\label{7.8.6}
We have 
$$
\Gam_{\CO_{\nu}}(X^{+}_{F_{1}}) - \Gam_{\CO_{\nu}}(X^{-}_{F_{1}})  =\sgn_{F_{1}/F}(\eta \nu), \quad 
\nu\in \CN^{V}.
$$
The elements $X^{+}_{F_{1}}$ and $X^{-}_{F_{1}}$ are stably conjugate but not $F$-conjugate, and from Proposition \ref{pro:6.5.1} only one of them lies in $\Gam(\Sig)(F)$.
\end{num}
In the following we will determine exactly which one of $X^{\pm}_{F_{1}}$ lies in $\Gam(\Sig)(F)$.
We will follow the strategy of \cite[11.5]{waldspurger10}.

From Proposition \ref{pro:6.3.1}, we know that when $\dim W$ is even, an element $X=(X_{W},X_{V})\in \Fg(F)$ lies in $\Sig^{\p}(F)$ if and only if there  exists elements $(z_{\pm i})_{i=1}^{m/2}$ of $\wb{F}$, where $m=\dim W$, such that
\begin{equation}\label{eq:intersec:1}
\bigg\{
\begin{matrix}
z_{i}z_{-i} = \frac{P_{X_{V}}(s_{i})}{2\nu_{0}s_{i}P_{-X_{W},i}(s_{i})}, & \text{ for any $i=1,...,\frac{m}{2}$},\\
\sum_{i=\pm 1}^{\pm \frac{m}{2}}z_{i}w_{i}\in W
\end{matrix}
\end{equation}
where $P_{-X_{W}, i}(T) = \frac{P_{-X_{W}}(T)}{T^{2}-s^{2}_{i}}$. Moreover we notice that $P_{X_{V}}(s_{i})\neq 0$ for any $i$.

Now following the notation from Section \ref{sub:germcal:explicitexample}, there exists 
\begin{itemize}
\item a decomposition of $W$ as a direct sum
$$
F_{1}\oplus F_{2}\oplus \wt{Z},
$$
where $F_{i}$ is a quadratic extension of $F$ for $i=1,2$;

\item element $c\in F^{\times}$ such that $q_{W}$ is the direct sum of the quadratic forms $c\RN_{F_{1}/F}$ on $F_{1}$, $-c\RN_{F_{2}/F}$ and
$\wt{Z}$ is a hyperbolic space;

\item elements $a_{j}\in F^{\times}_{j}$ such that $\tau_{F_{j}} (a_{j}) = -a_{j}$ for $j=1,2$ and an element $\wt{S}$ belonging to the Lie algebra of a maximal split subtorus of the special orthogonal group of $\wt{Z}$, such that $X_{F_{1}}$ acts by multiplication by $a_{j}$ on $F_{j}$ and by $\wt{S}$ on $\wt{Z}$.
\end{itemize}

Following Lemma \ref{lem:gemcal:basisform}, for any $j=1,2$, there exists a basis $\{e_{j},e_{-j} \}$ of $F_{j}\otimes_{F}\wb{F}$ such that any element $x\in F_{j}$ can be written as $xe_{j}\otimes \tau_{F_{j}}(x)e_{-j}$. Moreover we have $\tau_{F_{j}}(e_{j}) = e_{-j}$. We have the equality $q(e_{j},e_{-j}) = c_{j}$ with $c_{1} = c$ and $c_{2} = -c$. We can set $w_{j} =e_{j}$ and $w_{-j}= c^{-1}_{j}e_{-j}$ for $j=1,2$. We also fix a hyperbolic basis $(w_{j})_{j=\pm 3,...,\pm \frac{m}{2}}$ for $\wt{Z}$. It turns out that we have $s_{j} = a_{j}$ for $j=1,2$. Now the property (\ref{eq:intersec:1}) can be decomposed into $\frac{m}{2}$ separate relations $(\ref{eq:intersec:1})_{j}$ for the pairs $(z_{j},z_{-j})$ by taking $z_{-j} = 1$ and $z_{j}$ is equal to the RHS of the first relation appearing in (\ref{eq:intersec:1}). For $j\leq 2$, the second relation is equivalent to $z_{j}\in F^{\times}_{j}$ and $\tau_{F_{j}}(z_{j}) = c^{-1}_{j}z_{-j}$. Therefore the condition $(\ref{eq:intersec:1})_{j}$ is equivalent to
$$
\sgn_{F_{j}/F}
(\frac{P_{X_{V}(a_{j})}}{2c_{j}\nu_{0}a_{j}P_{-X_{W}}(a_{j})}) = 1.
$$
Following the same explicit computation as done in the proof of Lemma \ref{lem:germcal:Xpm}, for $X^{\zet}_{F_{1}}$ with $\zet = \pm$, the above formula is equal to $\sgn_{F_{1}/F}(-\eta\nu_{0}) = \zet$ if $j=1$ and $\sgn_{F_{2}/F}(-\eta\nu_{0}) = \zet$ if $j=2$. But since $\sgn_{E/F}(-\eta\nu_{0}) = 1$, therefore the two equalities for $j=1,2$ are actually equivalent.

It follows that we have shown that $X^{\zet}_{F_{1}}$ lies in $\Sig^{\p}(F)$ if and only if $\sgn_{F_{1}}(-\eta\nu_{0}) = \zet$.

Then, parallel to the previous case, for any $T\in \CT(G)$ with Lie algebra $\Ft$, we can find $\theta_{0,T}\in C^{\infty}_{c}(\Ft_{\reg})$ that is $W(G,T)$-invariant, such that 
$$
\int_{\Ft(F)}
D^{G}(X)^{1/2}\theta_{0,T}(X)dX = 1,
$$
and we can extend $\theta_{0,T}$ to a smooth invariant function on $\Fg^\rss(F)$ still denoted by $\theta_{0}$ that is zero outside $\Ft_{\reg}(F)^{G}$. It is a compactly supported quasi-character. We let the associated Fourier transform be $\theta_{T} = \wh{\theta_{0,T}}$. In particular, the torus associated to $X^{+}_{F_{1}}$ (resp. $X^{-}_{F_{1}}$) is denoted by $T_{1}^{+}$(resp. $T^{-}_{1}$). Moreover, from (\ref{7.8.5}) up to scaling we can further assume that $c_{\theta_{T_{1}^{+}},\CO}(X^{+}_{F_{1}}) = \Gam_{\CO}(X^{+}_{F_{1}})$ (resp. $c_{\theta_{T^{-}_{1}},\CO}(X^{-}_{F_{1}}) = \Gam_{\CO}(X^{-}_{F_{1}})$ ) for any $\CO\in \CN^{V}$.
Then we set
$$
\wt{\theta} = |\CN^{V}|^{-1}
(\theta+\sum_{F_{1}\in \CF^{V}} (\sgn_{F_{1}/F}(\nu \eta))(\theta_{T^{+}_{1}} - \theta_{T^{-}_{1}})),
$$
where the quasi-character $\theta$ is the one constructed from the case when $G(F)$ has a unique regular nilpotent orbit. Here $\CF^{V}$ is the set of collections of all degree two field extensions of $F$ if $\SO(W)$ is split. When $\SO(W)$ is quasi-split but not split, following the notation in Section \ref{sub:germcal:explicitexample}, we consider the pairs of quadratic extensions $F_{1},F_{2}$ of $F$ such that $\sgn_{F_{1}/F}\sgn_{F_{2}/F} = \sgn_{E/F}$ and non of $F_{1}$ nor $F_{2}$ is equal to $E$. We pick up either element in the pair and let the associated set be $\CF^{V}$. We note that $|\CN^{V}| = |\CF^{V}|+1$.
Then by definition we immediately find that
$$
m^{\Lie}_{\geom}(\wt{\theta}) = \del_{\nu,-\nu_{0}}.
$$

On the other hand, using the same argument as previous case, we have that 
$J^{\Lie}(\theta) = 1$, and $J^{\Lie}(\theta_{T_{1}^{+}}) = 1$ (resp. $J^{\Lie}(\theta_{T_{1}}^{-})$=1) if and only if $\sgn_{F_{1}/F}(-\eta\nu_{0}) = 1$ (resp. $\sgn_{F_{1}/F}(-\eta\nu_{0}) = -1$) and $0$ otherwise. Therefore we have 
$$
J^{\Lie}(\theta_{T^{+}_{1}}) - J^{\Lie}(\theta_{T^{-}_{1}}) = \sgn_{F_{1}/F}(-\eta\nu_{0})
$$
and hence
\begin{align*}
J^{\Lie}(\wt{\theta}) 
&= 
|\CN^{V}|^{-1}
(J^{\Lie}(\theta) + \sum_{F_{1}\in \CF^{V}}
(\sgn_{F_{1}/F}(\nu\eta) (J^{\Lie}(\theta_{T^{+}_{1}}) - J^{\Lie}(\theta_{T^{-}_{1}})  )
\\
&=
|\CN^{V}|^{-1}
(1+\sum_{F_{1}\in \CF^{V}}
\sgn_{F_{1}/F}(\nu\eta) \sgn_{F_{1}/F}(-\eta \nu_{0})),
\end{align*}
which again is equal to $\del_{\nu,-\nu_{0}}$. It follows that $c_{\CO} = 0$ for any $\CO\in \nil_{\reg}(\Fg(F))$ in (\ref{7.8.1}) and we have completed the proof.

The situation for $\dim V$ being even can be treated similarly following \cite[11.6]{waldspurger10}, and we omit the details.
\end{itemize}
\end{proof}

\section{An application to the Gan-Gross-Prasad conjecture}
We are going to prove that there is exactly one distinguished representation $\pi$ (i.e. one such that $m(\pi) = 1$) in every (extended) tempered $L$-packet of $G(F)$. In the $p$-adic case, this result was already proved by Waldspurger (\cite{waldspurger10}, \cite{waldspurgertemperedggp}).

\subsection{Strongly stable conjugacy classes, transfer between pure-inner forms and the Kottwitz sign}\label{sec:8.1}

Let $G$ be any connected connected group defined over $F$. In the discussion below we assume that $F$ is not algebraically closed, i.e. $F$ is either $p$-adic or real. 

Recall that a \emph{pure inner form} for $G$ is a triple $(G^\p,\psi,c)$ where 
\begin{itemize}
\item $G^{\p}$ is a connected reductive group defined over $F$;

\item $\psi: G_{\wb{F}}\simeq G^{\p}_{\wb{F}}$ is an isomorphism defined over $\wb{F}$;

\item $c:\sig\in \Gam_{F}\to c_{\sig}\in G(\wb{F})$ is a $1$-cocycle such that $\psi^{-1\sig}\psi = \Ad(c_{\sig})$ for any $\sig\in \Gam_{F}$.
\end{itemize}
The set of pure inner forms of $G$ can be parametrized by $H^1(F,G)$.
Moreover, inside an equivalence class of pure inner forms $(G^{\p},\psi,c)$, the group $G^{\p}$ is well-defined up to $G^{\p}(F)$-conjugacy. We will always assume for any $\alp\in H^{1}(F,G)$ a pure inner form in the class of $\alp$ that we will denote by $(G_{\alp},\psi_{\alp},c_{\alp})$ is fixed.

Let $(G^{\p},\psi,c)$ be a pure inner form of $G$. Following \cite[12.1]{beuzart2015local}, say that two semi-simple elements $x\in G_{\ss}(F)$ and $y\in G^{\p}_{\ss}(F)$ are strongly stably conjugate and write
$$
x\sim_{\stab}y
$$
if there exists $g\in G(F)$ such that $y = \psi(gxg^{-1})$ and the isomorphism $\psi\circ \Ad(g):G_{x}\simeq G_{y}$ is defined over $F$. The last condition has an interpretation in terms of cohomological classes: it means that the $1$-cocycle $\sig\in \Gam_{F}\to g^{-1}c_{\sig}\sig(g)$ takes its values in $Z(G_{x})$. For $x\in G_{\ss}(F)$ the set of semi-simple conjugacy classes in $G^{\p}(F)$ that are strongly stably conjugate to $x$ is naturally in bijection with
$$
\Im(H^{1}(F, Z(G_{x}) )\to H^{1}(F,Z_{G}(x))) \cap p^{-1}_{x}(\alp)
$$
where $\alp\in H^{1}(F,G)$ parametrizes the equivalence class of pure inner form $(G^{\p},\psi,c)$ and $p_{x}$ denotes the natural map $H^{1}(F,Z_{G}(x))\to H^{1}(F,G)$. The following fact will be needed (c.f. \cite[12.1.1]{beuzart2015local})
\begin{num}
\item\label{8.1.1}
Let $y^{\p}\in G^{\p}_{\ss}(F)$ and assume that $G$ and $G^{\p}_{y}$ are both quasi-split. Then, the set 
$$
\{ x\in G_{\ss}(F) | \quad x\sim_{\stab}y  \}
$$
is non-empty.
\end{num}

Continue to fix a pure inner form $(G^{\p},\psi,c)$ of $G$. Say a quasi-character $\theta$ on $G(F)$ is \emph{stable} if for any regular elements $x,y\in G^\rss(F)$ that are stably conjugate, $\theta(x) = \theta(y)$. Let $\theta$ and $\theta^{\p}$ be stable quasi-characters on $G(F)$ and $G^{\p}(F)$ respectively and assume moreover that $G$ is quasi-split. Then $\theta^{\p}$ is called a \emph{transfer} of $\theta$ if for any regular points $x\in G^\rss(F)$ and $y\in G^{\p\rss}(F)$ that are stably conjugate, $\theta^{\p}(y) = \theta(x)$. Note that if $\theta^{\p}$ is a transfer of $\theta$ then $\theta^{\p}$ is entirely determined by $\theta$. 

\begin{thm}\label{thm:replacegermstable}
Let $(V,q_V)$ be a quasi-split quadratic space of even dimension, $\theta$ a stable quasi-character on $G(F) = \SO(V)(F)$. Then 
$$
D^{G}(x)^{1/2}c_{\theta}(x) = 
|W(G_{x},T_{x})|^{-1}\lim_{x^{\p}\in T_{x}(F)\to x}
D^{G}(x^{\p})^{1/2}\theta(x^{\p}),
$$
In particular, for a stable quasi-character $\theta$ on $G(F)$, the definition $c_\theta$ introduced in \ref{sec:linearformgeom} coincides with the definition introduced in \cite[Section~4.5]{beuzart2015local}.
\end{thm}
\begin{proof}
When $G(F)$ contains only one regular nilpotent orbit, there is nothing to prove. Therefore we focus on the case when $G(F)$ has more than one regular nilpotent orbits.

When $F$ is $p$-adic, it follows from the same argument as \cite[Section~13.4]{waldspurger10} and \cite[Section~13.6]{waldspurger10} (one may replace the character distribution $\theta_{\pi}$ by any quasi-character $\theta$ to arrive at the conclusion). 

When $F=\BR$. By the description of regular nilpotent orbits in Section \ref{sub:ggptriples:definitions}, when $d\geq 4$, $|\nil_{\reg}(\Fg)| = |F^{\times}/F^{\times 2}|=2$. From \cite[Proposition~4.4.1 (vi)]{beuzart2015local} and \cite[1.8.1]{beuzart2015local}, we have the following identity for any $x\in G_{\ss}(F)$ and $Y\in \Fg^{\rss}(F)$,
\begin{align*}
&\lim_{t\in F^{\times 2}, t\to 0}D^{G}(xe^{tY})^{1/2}\theta(xe^{tY}) 
=
\lim_{t\in F^{\times 2},t\to 0} 
D^{G}(xe^{tY})^{1/2}\sum_{\CO\in \nil_{\reg}(\Fg_{x})}
c_{\theta,\CO}(x)\wh{j}(\CO,tY)
\\
&=
\lim_{t\in F^{\times 2},t\to 0}
D^{G}(xe^{tY})^{1/2}
|t|^{-\del_{G_{x}}/2}
\sum_{\CO\in \nil_{\reg}(\Fg_{x})}
c_{\theta,\CO}(x)
\wh{j}(\CO,Y)
\end{align*}
where $\del_{G_{x}} = \dim G_{x} - \dim T_{x}$ and $\dim T_{x}$ is the dimension of maximal torus in $G_{x}$.

In particular, since $\theta$ is a stable quasi-character, 
$$
\sum_{\CO\in \nil_{\reg}(\Fg_{x})}c_{\theta,\CO}(x)\wh{j}(\CO,Y)
$$
is a stable distribution supported on the regular nilpotent elements inside $\Fg_{x}(F)$ (Recall that a distribution $T$ on $\Fg(k)$ is called stable if for $f\in \CC^\infty_c(\Fg)$, $T(f) = 0$ whenever the stable orbital integral of $f$ at any regular semi-simple elements is equal to $0$).

Note that whenever $x$ is not regular semi-simple (otherwise there is nothing to prove), the set of regular nilpotent orbits in $\Fg_x(F)$ can also be parametrized by $\CN^V$. Write $\CN^{V} = \{ \pm 1\}$, and choose regular semi-simple elements $\{X^+_{F_1}, X^-_{F_1}\}$ as in appendix \ref{sub:germcal:explicitexample}, which is determined by the fact that $\Gam_{\CO_{+}}(X^{+}_{F_{1}}) = 1$ (resp. $\Gam_{\CO_{-}}(X^{-}_{F_{1}})= 1$) and zero otherwise. 

Since $\wh{j}(X^{+}_{F_{1}},\cdot)+\wh{j}(X^{-}_{F_{1}},\cdot)$ is a stable distribution on $\Fg_{x}(F)$, applying the limit formula \cite[p.98, Section~4.5]{beuzart2015local} to $\wh{j}(X^{+}_{F_{1}},\cdot)+\wh{j}(X^{-}_{F_{1}},\cdot)$, we get that $\wh{j}(\CO_{+},\cdot)+\wh{j}(\CO_{-},\cdot)$ is a stable distribution on $\Fg_{x}(F)$. But for a single distribution $\wh{j}(\CO_+,\cdot)$ (resp. $\wh{j}(\CO_-,\cdot)$), it is not a stable distribution, it follows that $c_{\theta,\CO_{+}} = c_{\theta,\CO_{-}}$. 

It follows that we have established the fact.
\end{proof}

We need the following fact (\cite[12.1.2]{beuzart2015local}).

\begin{num}
\item\label{8.1.2}
Let $\theta$ and $\theta^{\p}$ be stable quasi-characters on $G(F)$ and $G^{\p}(F)$ respectively and assume that $\theta^{\p}$ is a transfer of $\theta$. Then, for any $x\in G_{\ss}(F)$ and $y\in G^{\p}_{\ss}(F)$ that are strongly stably conjugate,
$$
c_{\theta^{\p}}(y) = c_{\theta}(x)
$$
\end{num}

Now assume that $G$ is quasi-split. Following Kottwitz (\cite{kottsign}), we may associate to any class of pure inner forms $\alp\in H^{1}(F,G)$ a sign $e(G_{\alp})$.  When $F$ is either $p$-adic or real, let $Br_{2}(F) = H^{2}(F,\{ \pm 1 \}) = \{  \pm 1\}$ be the $2$-torsion subgroup of the Brauer group of $F$. The sign $e(G_{\alp})$ will more naturally be an element of $Br_{2}(F)$. To define it, we need to introduce a canonical algebraic central extension
\begin{equation}\label{8.1.3}
1\to \{ \pm 1\}\to \wt{G}\to G\to 1
\end{equation}
Recall that a quasi-split connected group over $F$ is classified up to conjugation by its (canonical) based root datum $\Psi_{0}(G)=  (X_{G},\Del_{G},X^{\vee}_{G},\Del^{\vee}_{G})$ together with the natural action of $\Gam_{F}$ on $\Psi_{0}(G)$. For any Borel pair $(B,T)$ of $G$ that is defined over $F$, we have a canonical $\Gam_{F}$-equivariant isomorphism $\Psi_{0}(G) \simeq (X^{*}(T),\Del(T,B), X_{*}(T), \Del(T,B)^{\vee})$ where $\Del(T,B)\subset X^{*}(T)$ denotes the set of simple roots of $T$ in $B$ and $\Del(T,B)^{\vee}\subset X_{*}(T)$ denotes the corresponding sets of simple coroots. Fix such a Borel pair and set
$$
\rho  = \frac{1}{2}\sum_{\bet\in R(G,T)}\bet \in X^{*}(T)\otimes \BQ
$$
for the half sum of the roots of $T$ in $B$. The image of $\rho$ in $X_{G}\otimes \BQ$ does not depend on the choice of $(B,T)$ chosen and we still denote by $\rho$ this image. Consider now the following based root datum 
\begin{equation}\label{8.1.4}
(\wt{X}_{G},\Del_{G}, \wt{X}^{\vee}_{G},\Del^{\vee}_{G})
\end{equation}
where $\wt{X}_{G} = X_{G}+\BZ \rho\subset X_{G}\otimes \BQ$ and $\wt{X}_{G} = \{ \lam^{\vee}\in X^{\vee}_{G} | \quad \langle \lam^{\vee},\rho \rangle \in \BZ \}$. Note that we have $\Del^{\vee}_{G}\subset \wt{X}^{\vee}_{G}$ since $\langle \alp^{\vee},\rho \rangle  = 1$ for any $\alp^{\vee}\in \Del^{\vee}_{G}$. The based root datum (\ref{8.1.4}) with its natural $\Gam_{F}$-action, it the base root datum of a unique quasi-split group $\wt{G}_{0}$ over $F$ well-defined up to conjugacy. Moreover, we have a natural central isogeny $\wt{G}_{0}\to G$, well-defined up to $G(F)$-conjugacy, whose kernel is either trivial or $\{ \pm 1\}$. If the kernel is $\{ \pm 1 \}$, we set $\wt{G} = \wt{G}_{0}$ otherwise we simply set $\wt{G} = G\times \{\pm 1 \}$. In any case, we obtain a short exact sequence like (\ref{8.1.3}) well-defined up to $G(F)$-conjugacy. The last term of the long exact sequence associated to (\ref{8.1.3}) yields a canonical map
\begin{equation}\label{8.1.5}
H^{1}(F,G)\to H^{2}(F,\{ \pm 1 \}) = Br_{2}(F) \simeq \{ \pm 1 \}
\end{equation}
We now define the sign $e(G_{\alp})$ for $\alp\in H^{1}(F,G)$, simply to be the image of $\alp$ by this map. We will need the following fact 
(c.f. \cite[12.1.6]{beuzart2015local})
\begin{num}
\item\label{8.1.6}
Let $T$ be a (not necessarily maximal) subtorus of $G$. Then, the composition of (\ref{8.1.5}) with the natural map $H^{1}(F,T)\to H^{1}(F,G)$ is a group homomorphism $H^{1}(F,T)\to Br_{2}(F)$. Moreover, if $T$ is anisotropic this morphism is onto if and only if the inverse image $\wt{T}$ of $T$ in $\wt{G}$ is a torus (i.e. is connected).
\end{num}

\subsubsection{Pure inner forms of a GGP triple}\label{sec:8.2}
Let $V$ be an quadratic space. There is the following explicit description of the pure inner forms of $\SO(V)$. The cohomology set $H^{1}(F,\SO(V))$ naturally classifies the isomorphism classes of quadratic spaces of the same dimension and same discriminant as $V$. Let $\alp\in H^{1}(F,\SO(V))$ and choose an quadratic space $V_{\alp}$ in the isomorphism class corresponding to $\alp$. Set $V_{\wb{F}} = V\otimes_{F}\wb{F}$ and $V_{\alp,\wb{F}} = V_{\alp}\otimes_{F}\wb{F}$. Fix an isomorphism $\phi_{\alp}:V_{\wb{F}}\simeq V_{\alp,\wb{F}}$ of $\wb{F}$-quadratic spaces. Then, the triple $(\SO(V_{\alp}), \psi_{\alp}, c_{\alp})$, where $\psi_{\alp}$ is the isomorphism $\SO(V)_{\wb{F}}\simeq \SO(V_{\alp})_{\wb{F}}$ given by $\psi_{\alp}(g) = \phi_{\alp}\circ g\circ \phi^{-1}_{\alp}$ and $c_{\alp}$ is the $1$-cocycle given by $\sig\in \Gam_{F}\to \phi^{-1\sig}_{\alp}\phi_{\alp}$, is a pure inner form of $\SO(V)$ in the class of $\alp$. Moreover, the $2$-cover $\wt{\SO}(V)$ has the following description
\begin{itemize}
\item $\wt{\SO}(V) = \mathrm{Spin}(V)$, when $\dim V$ is odd;

\item $\wt{\SO}(V) = \SO(V)\times \{\pm 1 \}$, when $\dim V$ is even.

\end{itemize}

We now return to the GGP triple $(G,H,\xi)$ that we have fixed. Recall that the GGP triple comes from an admissible pair $(V,W)$ of quadratic spaces and that we are assuming in this chapter that $G$ and $H$ are quasi-split. Let $\alp\in H^{1}(F,H)$. We are going to associate to $\alp$ a new GGP triple $(G_{\alp},H_{\alp},\xi_{\alp})$ well-defined up to conjugacy. Since $H^{1}(F,H) = H^{1}(F,\SO(W))$, the cohomology class $\alp$ corresponds an isomorphism class quadratic space of the same dimension and same discriminant as $W$. Let $W_{\alp}$ be an quadratic space in this isomorphism class and set $V_{\alp} = W_{\alp}\oplus^{\perp}Z$. Then the pair $(V_{\alp},W_{\alp})$ is an admissible pair and hence there is a GGP triple $(G_{\alp},H_{\alp},\xi_{\alp})$ associated to it. This GGP triple is well-defined up to conjugacy. We call such a GGP triple a \emph{pure inner form} of $(G,H,\xi)$. By definition, these pure inner forms are parametrized by $H^{1}(F,H)$. Note that for any $\alp\in H^{1}(F,H)$, $G_{\alp}$ is a pure inner form of $G$ in the class corresponding to the image of $\alp$ in $H^{1}(F,G)$ and that the natural map $H^{1}(F,H)\to H^{1}(F,G)$ is injective.

\subsection{The local Langlands correspondence}\label{sec:8.3}

In this section, we recall the local Langlands correspondence in a form that will be used in the following. Let $G$ be a quasi-split connected reductive group over $F$ and denote by $\LG = \wh{G}\times W_{F}$ its Langlands dual, where $W_{F}$ denotes the Weil group of $F$. Recall that a \emph{Langlands parameter} for $G$ is a homomorphism from the group
$$
L_{F} = 
\bigg\{
\begin{matrix}
W_{F}\times \SL_{2}(\BC) & \textit{ if $F$ is $p$-adic} \\
W_{F} & \textit{ if $F$ is archimedean}
\end{matrix}
$$
to $\LG$ satisfying the usual conditions of continuity, semi-simplicity, algebraicity and compatibility with the projection $\LG\to W_{F}$. A Langlands parameter is said to be \emph{tempered} if $\vphi(W_{F})$ is bounded. By the hypothetical local Langlands correspondence, a tempered Langlands parameter $\vphi$ for $G$ should give rise to a finite set $\Pi^{G}(\vphi)$, called a \emph{$L$-packet}, of (isomorphism classes of) tempered representations of $G(F)$. Actually, such a parameter $\vphi$ should also give rise to tempered $L$-packets $\Pi^{G_{\alp}}(\vphi)\subset \mathrm{Temp}(G_{\alp})$ for any $\alp\in H^{1}(F,G)$. Among them, we expect the following properties to hold for every tempered Langlands parameter $\vphi$ of $G$:

(STAB) for any $\alp\in H^{1}(F,G)$, the character 
$$
\theta_{\alp,\vphi} = \sum_{\pi\in \Pi^{G_{\alp}}(\vphi)}\theta_{\pi}
$$
is stable.

See Section \ref{sec:8.1} for the definition of stable;
also notice that our different notion of "strongly stable conjugate" does not affect this property since it only involves the values of $\theta_{\alp,\vphi}$ at regular semi-simple elements (for which the two notions of stable conjugacy coincides). For $\alp = 1 \in H^{1}(F,G)$, in which case $G_{\alp} = G$, we shall simply set $\theta_{\vphi} = \theta_{1,\vphi}$.

(TRANS) for any $\alp\in H^{1}(F,G)$, the stable character $\theta_{\alp,\vphi}$ is the transfer of $e(G_{\alp})\theta_{\vphi}$ where $e(G_{\alp})$ is the Kottwitz sign whose definition has been recalled in Section \ref{sec:8.1}.

(WHITT) For every $\CO\in \nil_{\reg}(\Fg)$, there exists exactly one representation in the $L$-packet $\Pi^{G}(\vphi)$ admitting a Whittaker model of type $\CO$.

Notice that these conditions are far from characterizing the composition of the $L$-packets uniquely. However, by the linear independence of characters, conditions (STAB) and (TRANS) uniquely characterize the $L$-packets $\Pi^{G_{\alp}}(\vphi)$, $\alp\in H^{1}(F,G)$, in terms of $\Pi^{G}(\vphi)$.

When $F$ is archimedean, the local Langlands correspondence has been constructed by Langlands himself (\cite{langlandsclassify}) building on previous results of Harish-Chandra. This correspondence indeed satisfies the three conditions stated above. That (STAB) and (TRANS) hold is a consequence of early work of Shelstad (\cite[Lemma~6.2, Theorem~6.3]{shelsteadinner}. The property (WHITT) for its part, follows from the result of Kostant (\cite[Theorem~6.7.2]{kostantwhittaker}) and Vogan (\cite[Theorem~6.2]{vogandim}). When $F$ is $p$-adic, the local Langlands correspondence is known in a variety of cases. In particular, for special orthogonal groups, the existence of the Langlands correspondence is now fully established thanks to Arthur (\cite{ar13}) for quasi-split case, with supplement by \cite{atobegan} and a conjectural description for the inner forms, which in principle should follow from the last chapter of \cite{ar13}. That the tempered $L$-packets constructed in these references verify the conditions (STAB) and (TRANS) follows from \cite{ar13}. Moreover, the $L$-packets on the quasi-split form $G$ satisfy condition (WHITT) by \cite{ar13}.

\subsection{The theorem}\label{sec:8.4}
Recall that we have fixed a GGP triple $(G,H,\xi)$ with the requirement that $G$ and $H$ being quasi-split. Also, we have fixed in Section \ref{sec:8.2} the pure inner forms $(G_{\alp},H_{\alp},\xi_{\alp})$ of $(G,H,\xi)$. These are also GGP triples, they are parametrized by $H^{1}(F,H)$ and $G_{\alp}$ is a pure inner form of $G$ corresponding to the image of $\alp$ in $H^{1}(F,G)$ via the map $H^{1}(F,H)\to H^{1}(F,G)$.

\subsubsection{Stable conjugacy classes inside $\Gam(G,H)$}\label{sec:8.5}
Recall that in Section \ref{sec:linearformgeom}, we have defined a set $\Gam(G,H)$ of semi-simple conjugacy classes in $G(F)$. It consists in the $G(F)$-conjugacy classes of elements $x\in \SO(W)_{\ss}(F)$ such that 
$$
T_{x} : = \SO(W^{\p\p}_{x})_{x}
$$
is an anisotropic torus (where we recall that $W^{\p\p}_{x}$ denotes the image of $x-1$ in $W$). Two elements $x,x^{\p}\in \SO(W)_{\ss}(F)$ are $G(F)$-conjugate if and only if they are $\SO(W)(F)$-conjugate and moreover if it is so, any element $g\in \SO(W)(F)$ conjugating $x$ to $x^{\p}$ induces an isomorphism 
$$
\SO(W^{\p\p}_{x})_{x}\simeq \SO(W^{\p\p}_{x^{\p}})_{x^{\p}}
$$
Moreover, this isomorphism depends on the choice of $g$ only up to inner automorphism. From this it follows that any conjugacy class $x\in \Gam(G,H)$ determines the anisotropic torus $T_{x}$ up to a unique isomorphism so that we can speak of "the torus" $T_{x}$ associated to $x$.

These considerations apply verbatim to the pure inner forms $(G_{\alp},H_{\alp},\xi_{\alp}),\alp\in H^{1}(F,H)$, of the GGP triple $(G,H,\xi)$ that were introduced in Section \ref{sec:8.2}. In particular, for any $\alp\in H^{1}(F,H)$, we have a set $\Gam(G_{\alp},H_{\alp})$ of semisimple conjugacy classes in $G_{\alp}(F)$ and to any $y\in \Gam(G_{\alp},H_{\alp})$ is associated an anisotropic torus $T_{y}$.

\begin{pro}\label{pro:8.5.1}
\begin{enumerate}
\item Let $\alp\in H^{1}(F,H)$ and $y\in \Gam(G_{\alp},H_{\alp})$ be such that $G_{\alp,y}$ is quasi-split. Then, the set 
$$
\{ x\in \Gam(G,H) | \quad x\sim_{\stab}y \}
$$
is non-empty.

\item Let $\alp\in H^{1}(F,H)$, $x\in \Gam(G,H)$ and $y\in \Gam(G_{\alp},H_{\alp})$ be such that $x\sim_{\stab}y$. Choose $g\in G_{\alp}(\wb{F})$ such that $g\psi_{\alp}g^{-1} =y$ and $\Ad(g)\circ \psi_{\alp} : G_{x}\simeq G_{y}$ is defined over $F$. Then, $\Ad(g)\circ \psi_{\alp}$ restricts to an isomorphism
$$
T_{x}\simeq T_{y}
$$
that is independent of the choice of $g$.

\item Let $x\in \Gam(G,H)$. Then, for any $\alp\in H^{1}(F,H)$ there exists a natural bijection between the set
$$
\{ y\in \Gam(G_{\alp},H_{\alp}) | \quad x\sim_{\stab} y \}
$$
and the set
$$
q^{-1}_{x}(\alp)
$$
where $q_{x}$ denotes the natural map $H^{1}(F,T_{x})\to H^{1}(F,G)$.

\item Let $x\in \Gam(G,H), x\neq 1$. Then the composition of the map $\alp\in H^{1}(F,G)\to e(G_{\alp})\in Br_{2}(F)$ with the natural map $H^{1}(F,T_{x})\to H^{1}(F,G)$ gives a surjective morphism of groups $H^{1}(F,T_{x})\to Br_{2}(F)$.
\end{enumerate}
\end{pro}
\begin{proof}
(1) and (2): The proof of \cite[Proposition~12.5.1]{beuzart2015local} works verbatim.

(3) : Besides the fact that the set $H^1(F,\SO(V^\p_x))$ classifies the (isomorphism classes of) quadratic spaces of the same dimension and discriminant as $V^\p_x$, the proof of \cite[Proposition~12.5.1]{beuzart2015local} works verbatim.

(4) Let us denote by $\wt{G}$ the $2$-cover of $G$ at the end of Section \ref{sec:8.1} and let $\wt{T}_{x}$ be the inverse image of $T_{x}$ in this $2$-cover. Then, by (\ref{8.1.6}), it suffices to check that $\wt{T}_{x}$ is connected. By the precise description of $\wt{\SO}(V)$ and $\wt{\SO}(W)$ given at the beginning fo Section \ref{sec:8.2} and since exactly one of the quadratic spaces $V$ and $W$ is odd dimensional, we have 
$
\wt{T}_{x}
$
is equal to the inverse image of $T_{x}$ in the associated product of the spin group with another special orthogonal group, which is connected.
\end{proof}

Now we are ready to state our main theorem. The proof in \cite[12.6]{beuzart2015local} works verbatim. We refer the details to \cite{thesis_zhilin}.
\begin{thm}\label{thm:8.4.1}
Let $\vphi$ be a tempered Langlands parameter for $G$. Then, there exists a unique representation $\pi$ in the disjoint union of $L$-packets
$$
\bigsqcup_{\alp\in H^{1}(F,H)}\Pi^{G_{\alp}}(\vphi)
$$
such that $m(\pi) = 1$.
\end{thm}

\appendix

\section{A formula for the regular nilpotent germs}\label{sec:germformula}
In this section, we are going to give a formula for the germs of Lie algebra orbital integrals associated to regular nilpotent orbits in a quasi-split connected reductive algebraic group $G$ over any local field $F$ of characteristic zero in terms of endoscopic invariants. The formula was first proved by Shelstad over $p$-adic fields (\cite{regulargermshelstad}). We also discuss the relation between the formula and Kostant's sections.

The organization of the section is as follows. We first review the definition of Lie algebra endoscopic transfer factors, following the work of Langlands-Shelstad (\cite{transferfactorLS}), Waldspurger (\cite{waldtransfert}) and Kottwitz (\cite{kottwitztransfer}), and recall a theorem in \cite[Theorem~5.5.A.]{transferfactorLS} which relates the transfer factors for regular semi-simple elements and regular nilpotent orbits. Then we establish the formula. 
The definition of germ expansion for Lie algebra orbital integrals is known over $p$-adic fields (\cite{MR323957}). Over archimedean local fields, there are also asymptotic expansions (\cite{MR1248081}, \cite{MR3675163}) relating the orbital integrals for regular semi-simple elements and distributions supported on the nilpotent cone. Since we only care about the germs associated to the regular nilpotent orbits, we will simply use the results from \cite[p.98, Section~4.5]{beuzart2015local}. Finally we establish the relation between the formula and Kostant's sections based on \cite[Theorem~5.1]{kottwitztransfer}.
\newline

\paragraph{\textbf{Notation and conventions}}
Let $F$ be a local field of characteristic zero with fixed valuation $|\cdot|$ and Galois group $\Gam_F= \Gal(\wb{F}/F)$, where $\wb{F}$ is a fixed algebraic closure of $F$. Fix an additive character $\psi$ of $F$ and a Haar measure on $F$ that is self-dual w.r.t. $\psi$.

Fix a quasi-split connected reductive algebraic group $G$ defined over $F$. Over $F$, fix a maximal torus $T$ of $G$ and a Borel subgroup $B_0$ of $G$ containing $T$ with Levi decomposition $B_0 = TN_0$, where $N_0$ is the unipotent radical of $B_0$. Let $B_\infty$ be the unique Borel subgroup of $G$ containing $T$ that is opposite to $B_0$. Gothic letters are used to denote the Lie algebras of the associated algebraic groups.

Let $W=W(G,T)$ be the associated Weyl group. Let $R_G=R(T,G)$ be the set of roots of $T$ in $G$. For any $\alp\in R_G$, let $\Fg_\alp$ be the root space of $\Ft$ in $\Fg$ corresponding to $\alp$. In particular $\Fn_0 = \bigoplus_{\alp>0}\Fg_\alp$ and $\Fn_\infty = \bigoplus_{\alp<0}\Fg_\alp$. For any $\alp\in R_G$, let $\check{\alp}$ be the coroot dual to $\alp$. Let $\Del = \Del(T,B_0)$ be the set of positive simple roots determined by $B_0$.

Let $\Fg^{\rss}(F)$ be the subset of regular semi-simple elements in $\Fg(F)$. The definition of Weyl discriminant $D^G(X)$ is given by 
$$
D^G(X) = |\det \mathrm{ad}(X)_{\Fg(F)/\Fg_X(F)}|
$$
for any $X\in \Fg^{\rss}(F)$, where $\Fg_X(F)$ is the centralizer of $X$ in $\Fg(F)$.

Let $\CS(\Fg) = \CS(\Fg(F))$ be the space of Schwartz-Bruhat functions on $\Fg(F)$, where $\Fg(F)$ is viewed as a vector space over $F$.

Fix a $G(F)$-invariant non-degenerate bilinear form $\langle\cdot,\cdot\rangle$ on $\Fg(F).$ Endow $\Fg(F)$ with the self-dual measure w.r.t. $\langle\cdot,\cdot\rangle$, which is the unique Haar measure $\ud X$ on $\Fg(F)$ such that the Fourier transform 
$$
\CF(f)(Y)=  \wh{f}(Y) = \int_{\Fg(F)}
f(X)\psi(\langle X,Y\rangle)\ud X,\quad f\in \CS(\Fg)
$$ 
satisfies 
$$
\CF({\CF(f)})(X)=  f(-X).
$$
Equip $G(F)$ with the unique Haar measure $\ud g$ such that the exponential map has Jacobian equal to $1$ at identity.

Let $\nil_{\reg}(\Fg) = \nil_\reg(\Fg(F))$ be the set of regular nilpotent orbits in $\Fg(F)$ under the adjoint action of $G(F)$. For any $\CO\in \nil_{\reg}(\Fg)$ and $X\in \CO$, the bilinear map $(Y,Z)\to B(Y,[X,Z])$ yields a non-degenerate symplectic form on $\Fg(F)/\Fg_{X}(F)$, which can be viewed as the tangent space of $\CO$ at $X$. This gives $\CO$ a structure of symplectic $F$-analytic manifold. By the Haar measure on $F$, $\CO$ can be equipped with a natural measure that is $G(F)$-invariant.

For $X\in \Fg^{\rss}(F)$, there is the normalized orbital integral at $X$ defined by
$$
J_{G}(X,f) = D^{G}(X)^{1/2}\int_{G_{X}\bs G}
f(g^{-1}Xg)\ud g,\quad f\in \CS(\Fg).
$$
Similarly, for $\CO\in \nil_{\reg}(\Fg)$,
$$
J_{\CO}(f) = \int_{\CO}f(X)\ud X, \quad f\in \CS(\Fg).
$$
The distribution $J_{G}(X,\cdot)$ (resp. $J_{\CO}(\cdot)$) is tempered. Denote the Fourier transform by $\wh{j}(X,\cdot)$ (resp. $\wh{j}(\CO, \cdot)$). It is locally integrable on $\Fg(F)\times \Fg(F)$ (resp. $\Fg(F)$) and  smooth on $\Fg^{\rss}(F)\times \Fg^\rss(F)$ (resp. $\Fg^{\rss}(F)$) (c.f. \cite[1.8]{beuzart2015local}).

\subsection{Transfer factors}\label{sub:germformula:transferfactors}
In this subsection, we review the definition of Lie algebra endoscopic transfer factors. 

\begin{defin}\label{defin:splitting}
The triple $\textbf{\emph{spl}} :=(B_{0},T,\{ X_{\alp} \})$ is called an $F$-splitting for $G$. 
\end{defin}

More precisely, by the Jacobson-Morozov theorem (\cite[Theorem~3]{MR49882}), for any $\alp\in \Del$, it can be associated with a standard $\Fs\Fl(2)$ triple $\{X_{\alp},H_{\alp}, X_{-\alp} \}$. Fix such a triple for any $\alp\in \Del$, and let $X_{+} = \sum_{\alp\in \Del}X_{\alp}$, $X_{-} =\sum_{\alp\in \Del}X_{-\alp}$. Both $X_{+}$ and $X_{-}$ are regular nilpotent elements in $\Fg(F)$.

Let $(H,\CH,s,\xi)$ be an endoscopic data for $G$ (\cite[p.9]{transferfactorLS}). For a quasi-split connected reductive algebraic group $G$ defined over $F$, Langlands and Shelstad defined the notion of transfer factor $\Del_{0}(\gam_{H},\gam_{G})$ for any $\gam_{H}\in H(F)$ that is $G(F)$-regular semi-simple (\cite[1.3]{transferfactorLS}) and $\gam_{G}\in G(F)$ that is regular-simple. The transfer factor $\Del_{0}(\gam_{H},\gam_{G})$ depends on the choice of the $F$-splitting \textbf{spl}. The corresponding Lie algebra variant of $\Del_{0}(\gam_{H},\gam_{G})$ introduced below, which is denoted by $\Del^{\p}_{0}(X_{H},X_{G})$, is analogous to the one defined by Langlands and Shelstad, with the factor $\Del_{\mathrm{IV}}$ (Weyl discriminant) removed.

To introduce $\Del^{\p}_{0}(X_{H},X_{G})$, a set of \emph{$a$-data} and \emph{$\chi$-data} need to be fixed. We recall the notions from \cite[2.3]{waldtransfert}.

\subsubsection{$a$-data.}
Let $T_{H}$ be a maximal torus of $H$ defined over $F$. From the definition of endoscopic data, there exists a canonical $G$-conjugacy class of embeddings $T_{H}\hookrightarrow G$. Moreover we can fix an embedding $T_{H}\hookrightarrow G$ that is defined over $F$. Let $T_{G}$ be the image of $T_{H}$ in $G$, where $T_{G}$ is a maximal torus of $G$ defined over $F$. Let $R_{G} = R(T_{G},G)$ be the set of roots of $T_{G}$ in $G$. Similarly the set $R_{H}$ can be defined. Identify $T_{H}$ with $T_{G}$ so that $R_{H}$ becomes a subset of $R_{G}$.

\begin{defin}\label{defin:adata}
An $a$-data for $T_{G}$ is a collection of elements $\{a_{\alp} \}_{\alp\in R_{G}}$ with $a_{\alp}\in \wb{F}^{\times}$ such that 
\begin{itemize}
\item $a_{\sig\alp} = \sig(a_{\alp})$ for any $\alp\in R_{G}$ and $\sig\in \Gam_{F}$;

\item $a_{-\alp} = -a_{\alp}$ for any $\alp\in R_{G}$.
\end{itemize}
\end{defin}

Fix an $a$-data for $T_{G}$.

\subsubsection{$\chi$-data}
For any $\alp\in R_{G}$, let $F_{\alp}$ (resp. $F_{\pm \alp}$) be the field of definition of $\alp$ (resp. the set $\{ \pm \alp\}$). Then $F\subset F_{\pm \alp}\subset F_{\alp}\subset \wb{F}$, and $[F_{\alp}:F_{\pm \alp}] = 1$ or $2$. Following \cite[2.5]{transferfactorLS}, $\alp$ (and its $\Gam_{F}$-orbit in $R_{G}$) is called \emph{symmetric} if $[F_{\alp}:F_{\pm \alp}] =2$ and let $\chi_{\alp}$ be the quadratic character on $F^{\times}_{\pm \alp}$ associated to the quadratic extension $F_{\alp}/F_{\pm \alp}$ via local class field theory (\cite[Theorem~2]{MR0220701}). Simpler than the definition in \cite[2.5]{transferfactorLS}, it is not needed to extend $\chi_{\alp}$ to a character of $F^{\times}_{\alp}$.

In the following the Lie algebra transfer factors is introduced. Fix $X_{H}\in \Ft_{H}(F)$ and assume that its image $X_{G}$ in $\Ft_{G}(F)$ is regular semi-simple.

\subsubsection{$\Del_{\mathrm{I}}(X_{H},X_{G})$}
First consider the case when $G$ is semi-simple and simply connected. With the fixed Borel pair $(B_{0},T)$, we can fix a canonical section $n:W = W(G,T)\to \mathrm{Norm}_{G}(\wb{F})(T)$ (\cite[2.1]{transferfactorLS}). Fix an element $x\in G(\wb{F})$ such that $xTx^{-1}=T_{G}$. Then the conjugation action of $x$ yields an ordering on $R_{G} =R(T_{G},G)$. For any $\sig\in \Gam_{F}$, the element $x^{-1}\sig(x)$ normalizes $T$, hence it provides an element $w_{\sig}\in W$. Let $n_{\sig} = n(w_{\sig})$. For any $\alp\in R_{G}$, let $\check{\alp}$ be the associated coroot. Set 
$$
a_{\sig} = \prod_{\alp\in R_{G}, \alp>0, \sig^{-1}(\alp)<0}\check{\alp}\otimes a_{\alp}\in X_{*}(T_{G})\otimes_{\BZ}\wb{F}^{\times}\simeq T_{G}(\wb{F}),\quad \sig\in \Gam_{F},
$$
and define
$$
\lam(T_{G})(\sig) = a_{\sig}xn_{\sig}\sig(x^{-1})\in T_{G}(\wb{F}), \quad \sig\in \Gam_{F}.
$$
It can be shown that $\sig\to \lam(T_{G})(\sig)$ is a $1$-cocycle of $\Gam_{F}$ valued in $T_{G}(\wb{F})$ (\cite[2.3]{transferfactorLS}). We use the same notation $\lam(T_{G})$ to denote the associated cohomology class in $H^{1}(F,T_{G}(\wb{F}))$.

In general, let $G_{\mathrm{sc}}$ be the simply connected cover of the derived group of $G$, and $T^{\mathrm{sc}}_{G}$ be the inverse image of $T_{G}$ under the canonical morphism $G_{\mathrm{sc}}\to G$. Denote $\lam(T_{G})$ the image of $\lam(T^{\mathrm{sc}}_{G})$ under the map induced by the canonical morphism $T^{\mathrm{sc}}_{G}\to T_{G}$.
\begin{rmk}
To make the notation in consistent with \cite{regulargermshelstad}, we will also denote the invariant $\lam(T_{G})$ by $\inv(T_{G})$.
\end{rmk}

Let $\wh{T}_{G}$ be the complex torus dual to $T_{G}$. The element $s$ appearing in the endoscopic data is a $\Gam_{F}$-fixed element in the center of the Langlands dual group $\wh{H}$ of $H$, and thus can be viewed as a $\Gam_{F}$-fixed element $\textbf{s}_{T_{G}}$ in $\wh{T}_{H} = \wh{T}_{G}$. 

Recall the Tate-Nakayama pairing (\cite{MR858284})
$$
\langle\cdot,\cdot\rangle:H^{1}(F,T_{G}(\wb{F}))\times \wh{T}_{G}^{\Gam_{F}}\to \BC^{\times}.
$$

\begin{defin}\label{defin:transferfactorI}
Define the transfer factor $\Del_{\mathrm{I}}(X_{H},X_{G})$ to be 
$$
\langle \inv(T_{G}), \textbf{s}_{T_{G}}\rangle.
$$
\end{defin}

\subsubsection{$\Del_{\mathrm{II}}(X_{H},X_{G})$}

\begin{defin}\label{defin:transferfactorII}
Define
$$
\Del_{\mathrm{II}}(X_{H},X_{G}):=\prod_{\alp}
(\frac{\alp(X_{G})}{a_{\alp}}),
$$
where the product is taken over a set of representatives for the symmetric orbits of $\Gam_{F}$ in the set $R_{H}\bs R_{G}$.
\end{defin}

Following \cite[Lemma~2.2.B, 2.2.C, 3.2.D]{transferfactorLS}, \cite{regulargermshelstad}, the factor $\Del_{\mathrm{II}}$ also admits the following interpretation. The morphism
$$
\inv(X_{G}):\sig\in \Gam_{F}\to 
\prod_{\alp\in R_{G}, \alp>0, \sig^{-1}\alp<0}
\check{\alp}\circ 
(\frac{\exp(\frac{\alp(X_{G})}{2})-\exp(-\frac{\alp(X_{G})}{2})}{a_{\alp}})
$$
is a $1$-cocycle of $\Gam_{F}$ in $T_{G}(\wb{F})$. From \cite[Lemma~3.2.D]{transferfactorLS}, there is the equality
$$
\Del_{\mathrm{II}}(X_{H},X_{G}) = \langle \inv(X_{G}), \textbf{s}_{T_{G}}\rangle
$$
whenever $X_{G}$ is sufficiently close to $0$.

Now the Lie algebra transfer factor is defined as the product of $\Del_{\mathrm{I}}$ and $\Del_{\mathrm{II}}$.

\begin{defin}\label{defin:transferfactorwhole}
Define
$$
\Del^{\p}_{0}(X_{H},X_{G}) = 
\Del_{\mathrm{I}}(X_{H},X_{G})
\Del_{\mathrm{II}}(X_{H},X_{G}).
$$
\end{defin}

\begin{rmk}\label{rmk:transferscalarinvariant}
Following \cite[Lemma~3.2.C]{transferfactorLS}, we can show that $\Del^{\p}_{0}(X_{H},X_{G})$ is independent of the choice of $a$-data. In particular $\Del^{\p}_{0}(X_{H},X_{G})$ depends only on the choice of the $F$-splitting.

Set $\Del^{\p}_{0}(\gam_{H},\gam_{G}) :=\Del_{0}(\gam_{H}, \gam_{G})\cdot \Del_{\mathrm{IV}}(\gam_{H},\gam_{G})^{-1}$ where the terms are group version of the transfer factors defined in \cite{transferfactorLS}. Then by definition, whenever $X_{G}$ is sufficiently close to $0$,
$$
\Del^{\p}_{0}(X_{H},X_{G}) = \Del^{\p}_{0}(\exp(X_{H}), \exp(X_{G})).
$$
Moreover, from the above definitions, the following equality holds for any $a\in F^{\times}$,
$$
\Del^{\p}_{0}(a^{2}X_{H}, a^{2}X_{G}) = 
\Del^{\p}_{0}(X_{H},X_{G}).
$$
\end{rmk}

\begin{rmk}\label{rmk:transferstableconjugate}
Suppose that $X^{\p}_{G}\in \Fg(F)$ is stably conjugate to $X_{G}$, i.e. there exists $h\in G(\wb{F})$ such that $\Ad(h)(X^{\p}_{G}) = X_{G}$. Then $\sig\to h\sig(h)^{-1}$ is a $1$-cocycle of $\Gam_{F}$ in $T_{G}(\wb{F})$ whose cohomology class will be denoted by $\inv(X_{G},X^{\p}_{G})$. From \cite[Lemma~3.2.B, 3.4.A]{transferfactorLS}, the following equality holds
$$
\Del^{\p}_{0}(X_{H},X^{\p}_{G}) \cdot 
\langle \inv(X_{G}, X^{\p}_{G}), \textbf{s}_{T_{G}} \rangle^{-1} = \Del^{\p}_{0}(X_{H},X_{G}).
$$
\end{rmk}

\subsubsection{$\Del(\CO)$}
Finally we recall the transfer factor associated to the regular nilpotent (unipotent) elements in $G(F)$. Following \cite[Section~5.1]{transferfactorLS}, for any regular nilpotent conjugacy class $\CO\in \nil_{\reg}(\Fg)$, we can attach to it an $F$-splitting $\textbf{spl}(\CO)$. Moreover, the correspondence $\CO\to \textbf{spl}(\CO)$ induces a bijection between $\nil_{\reg}(\Fg)$ and the $G$-conjugacy classes of $F$-splittings of $G$. Following Langlands and Shelstad, if $\textbf{spl}^{g} = \textbf{spl}_{\infty}$, where $\textbf{spl}_{\infty}$ is the $F$-splitting \emph{opposite} to the fixed one in the beginning, and $g\in G_{\mathrm{sc}}(\wb{F})$, then $\inv(\CO):\sig\to g\sig(g)^{-1}$ is a $1$-cocycle of $\Gam_{F}$ in $Z_{sc}(\wb{F})$, where $Z_{sc}$ is the center of $G_{\mathrm{sc}}$. After composing the canonical morphisms $Z_{\mathrm{sc}}\to T^{\mathrm{sc}}_{G}\to T_{G}$, we obtain a cohomology class $\inv_{T_{G}}(\CO)$ in $H^{1}(F,T_{G})$. Set 
$$
\Del(\CO) = \langle \inv_{T_{G}}(\CO), \textbf{s}_{T_{G}}\rangle.
$$

There is the following theorem connecting the transfer factors introduced above, which can be obtained from \cite[Theorem~5.5.A]{transferfactorLS} through descending to Lie algebra directly.

\begin{thm}\label{thm:transferfactorrelation}
The following identity holds
$$
\lim_{X_{H}\to 0}
\sum_{X_{G}}
\Del^{\p}_{0}(X_{H},X_{G})J_{G}(X_{G},f) = 
\sum_{\CO\in \nil_{\reg}(\Fg(F))}\Del(\CO)J_{\CO}(f), \quad f\in \CS(\Fg).
$$
The sum $X_{G}$ runs over elements in $\Fg^{\rss}(F)$ such that $\Del^{\p}_{0}(X_{H},X_{G})\neq 0$. In particular it is a finite sum.
\end{thm}

\subsection{The formula}\label{sub:germformula:theformula}

We are going to establish the following theorem. Over $p$-adic fields, it was first proved by Shelstad (\cite{regulargermshelstad}).

\begin{thm}\label{thm:germformula}
For any $X\in \Fg^{\rss}(F)$ and $\CO\in \nil_{\reg}(\Fg)$, let $T_{G}=G_{X}$, then 
$$
\Gam_{\CO}(X) =
\bigg\{
\begin{matrix}
1,	& \text{ \rm{if} $\inv(X)\inv(T_{G}) = \inv_{T_{G}}(\CO)$},\\
0,	& \textit{\rm{otherwise}}.
\end{matrix}
$$
\end{thm}

From \cite[p.98, Section~4.5]{beuzart2015local}, the following asymptotic expansion for any $Y\in \Fg^{\rss}(F)$ holds
\begin{equation}\label{eq:germdefin}
\lim_{t\in F^{\times 2}, t\to 0}
D^{G}(tY)^{1/2}\wh{j}(X,tY) = D^{G}(Y)^{1/2}
\sum_{\CO\in \nil_{\reg}(\Fg)}
\Gam_{\CO}(X)\wh{j}(\CO,Y).
\end{equation}
The constants $\Gam_{\CO}(X)$ appearing in the statement of Theorem \ref{thm:germformula} are exactly the terms showing up in (\ref{eq:germdefin}).

We first establish the following lemma.

\begin{lem}\label{lem:pickupgerm}
For any $f\in \CS(\Fg)$, 
$$
\lim_{t\in F^{\times 2}, t\to 0}J_{G}(X, \wh{f}_{t}) = \sum_{\CO\in \nil_{\reg}(\Fg)}\Gam_{\CO}(X)J_{\CO}(f),
$$
where $f_{t}(Y) = |t|^{\del_{G}/2-\dim G}f(t^{-1}Y)$, and $\del_{G} =\dim G-\dim T_G$.
\end{lem}
\begin{proof}
From (\ref{eq:germdefin}), 
$$
\lim_{t\in F^{\times 2}, t\to 0}
\frac{D^{G}(tY)^{1/2}}{D^{G}(Y)^{1/2}}\wh{j}(X,tY) = \sum_{\CO\in \nil_{\reg}(\Fg)}
\Gam_{\CO}(X)\wh{j}(\CO,Y).
$$
Integrate both sides against a function $f(Y)\in \CS(\Fg)$ on $\Fg^{\rss}(F)$. Then 
$$
\mathrm{RHS} = 
\int_{\Fg^{\rss}(F)}
\sum_{\CO\in \nil_{\reg}(\Fg)}
\Gam_{\CO}(X)\wh{j}(\CO,Y)f(Y)\ud Y.
$$
Since $\wh{j}(\CO,Y)$ is locally integrable on $\Fg(F)$,
$$
\mathrm{RHS} = \sum_{\CO\in \nil_{\reg}(\Fg)}
\Gam_{\CO}(X)J_{\CO}(\wh{f}).
$$
For the LHS, using the fact that $D^{G}(Y)^{1/2}\wh{j}(X,Y)$ is globally bounded on $\Fg^{\rss}(F)\times \Fg^{\rss}(F)$ (c.f. \cite[1.8]{beuzart2015local}), and the function $D^{G}(Y)^{-1/2}$ defines a tempered distribution on $\CS(\Fg)$ (c.f. \cite[1.7.1]{beuzart2015local}), by the dominated convergence theorem (\cite[p.67]{MR2129625}), the LHS can be written as 
$$
\lim_{t\in F^{\times 2}, t\to 0}
\int_{\Fg^{\rss}(F)}
\frac{D^{G}(tY)^{1/2}}{D^{G}(Y)^{1/2}}
\wh{j}(X,tY)f(Y)\ud Y.
$$
Since $Y\in \Fg^{\rss}(F)$, $\frac{D^{G}(tY)^{1/2}}{D^{G}(Y)^{1/2}} = |t|^{\del_{G}}$. After a change of variable $Y\to t^{-1}Y$,
$$
\mathrm{LHS} = \lim_{t\in F^{\times 2}, t\to 0}
\int_{\Fg^{\rss}(F)}
|t|^{\del_{G}/2-\dim G}
\wh{j}(X,Y)f(t^{-1}Y)\ud Y,
$$
which is equal to
$$
\lim_{t\in F^{\times 2}, t\to 0}
J_{G}(X,\wh{f_{t}})
$$
by the local integrability of $\wh{j}(X,Y)$ in variable $Y$. 

It follows that we have established the lemma.
\end{proof}

As a corollary, we can show that $\Gam_{\CO}(X)$ is invariant under the scaling action of $F^{\times 2}$. Over $p$-adic fields it is already known (\cite[2.6]{waldspurger10}).

\begin{cor}\label{cor:invariantscaling}
$$
\Gam_{\CO}(aX) = \Gam_{\CO}(X)
$$
for any $a\in F^{\times 2}$, $X\in \Fg^{\rss}(F)$ and $\CO\in \nil_{\reg}(\Fg)$.
\end{cor}
\begin{proof}
For any $f\in \CS(\Fg)$, 
$$
\wh{f_{t}}(Y) = 
\int_{\Fg(F)}
|t|^{\del_{G}/2-\dim G}f(t^{-1}X)\psi(\langle X,Y\rangle)\ud X.
$$
After a change of variable $X\to tX$,
$$
\wh{f_{t}}(Y) = 
\int_{\Fg(F)}
|t|^{\del_{G}/2}f(X)\psi(\langle X, tY\rangle)\ud X = |t|^{\del_{G}/2}\wh{f}(tY).
$$
Therefore 
\begin{align}\label{eq:scaleonX}
J_{G}(X, \wh{f_{t}}) 
&= |t|^{\del_{G}/2}D^{G}(X)^{1/2}\int_{G_{X}\bs G}\wh{f}(gtXg^{-1})\ud g
\\
\nonumber 
&=
D^{G}(tX)^{1/2}
\int_{G_{X}\bs G}
\wh{f}(gtXg^{-1})\ud g = J_{G}(tX,\wh{f}).
\end{align}
In particular Lemma \ref{lem:pickupgerm} can be reformulated as 
\begin{equation}\label{eq:limitformula}
\lim_{t\in F^{\times 2}, t\to 0}
J_{G}(tX,\wh{f}) = \sum_{\CO\in \nil_{\reg}(\Fg)}\Gam_{\CO}(X)J_{\CO}(\wh{f}).
\end{equation}
From (\ref{eq:limitformula}) we get the desired identity.
\end{proof}

We also have the following lemma.

\begin{lem}\label{lem:reggerminv}
For any $f\in \CS(\Fg)$, $t\in F^{\times 2}$, and $\CO\in \nil_{\reg}(\Fg)$,
$$
J_{\CO}(\wh{f_{t}}) = J_{\CO}(\wh{f}).
$$
\end{lem}
\begin{proof}
By definition, 
$$
J_{\CO}(\wh{f_{t}}) = 
\int_{\Fg(F)}|t|^{\del_{G}/2-\dim G}\wh{j}(\CO,Y)f(t^{-1}Y)\ud Y.
$$
After a change of variable $Y\to tY$,
$$
J_{\CO}(\wh{f_{t}}) = \int_{\Fg(F)}
|t|^{\del_{G}/2}\wh{j}(\CO, tY)f(Y)\ud Y.
$$
Since $t\in F^{\times 2}$, from \cite[1.8.1]{beuzart2015local},
$$
\wh{j}(\CO, tY) = |t|^{-\dim (\CO)/2}\wh{j}(\CO,Y) = |t|^{-\del_{G}/2}\wh{j}(\CO,Y).
$$
It follows that we have established the desired equality.
\end{proof}

For any $f\in \CS(\Fg)$ and $t\in F^{\times 2}$, plug the function $\wh{f_{t}}$ into Theorem \ref{thm:transferfactorrelation}. Then the following identity holds
$$
\lim_{X_{H}\to 0}
\sum_{X_{G}}
\Del^{\p}_{0}(X_{H},X_{G})J_{G}(X_{G}, \wh{f_{t}}) = 
\sum_{\CO\in \nil_{\reg}(\Fg)}
\Del(\CO)J_{\CO}(\wh{f_{t}}).
$$
Applying (\ref{eq:scaleonX}) and Lemma \ref{lem:reggerminv} to the above identity, the following identity holds,
\begin{equation}\label{eq:eq1germ}
\lim_{X_{H}\to 0}
\sum_{X_{G}}
\Del^{\p}_{0}(X_{H},X_{G})J_{G}(tX_{G},\wh{f}) = 
\sum_{\CO\in \nil_{\reg}(\Fg)}
\Del(\CO)J_{\CO}(\wh{f}).
\end{equation}
On the other hand, from Lemma \ref{lem:pickupgerm} and (\ref{eq:scaleonX}), the following identity holds
\begin{equation}\label{eq:eq2germ}
\lim_{t\in F^{\times 2},t\to 0}
\sum_{X_{G}}
\Del^{\p}_{0}(X_{H},X_{G})
J_{G}(tX_{G}, \wh{f}) = \sum_{X_{G}}\Del^{\p}_{0}(X_{H},X_{G})\sum_{\CO\in \nil_{\reg}(\Fg)}
\Gam_{\CO}(X_{G})J_{\CO}(\wh{f}).
\end{equation}
It is natural to expect that the RHS of the two equations are equal to each other, which is going to be established in the next lemma.
\begin{lem}\label{lem:eqholds}
For any $f\in \CS(\Fg)$,
$$
\sum_{\CO\in \nil_{\reg}(\Fg)}
\Del(\CO)J_{\CO}(\wh{f}) = \sum_{X_{G}}
\Del^{\p}_{0}(X_{H},X_{G})\sum_{\CO\in \nil_{\reg}(\Fg)}\Gam_{\CO}(X_{G})J_{\CO}(\wh{f}).
$$
\end{lem}
\begin{proof}
For any $a\in F^{\times 2}$, by Remark \ref{rmk:transferscalarinvariant}, $\Del^{\p}_{0}(aX_{H},aX_{G}) = \Del^{\p}_{0}(X_{H},X_{G})$. From Lemma \ref{lem:reggerminv}, $\Gam_{\CO}(X_{G}) = \Gam_{\CO}(aX_{G})$.

Consider the limit,
\begin{align*}
&\lim_{a\in F^{\times 2},a\to 0}\lim_{t\in F^{\times 2}, t\to 0}
\sum_{aX_{G}}\Del^{\p}_{0}(aX_{H},aX_{G})J_{G}(atX_{G}, \wh{f})
\\
&=
\lim_{a\in F^{\times 2},a\to 0}
\lim_{t\in F^{\times 2}, t\to 0}
\sum_{X_{G}}\Del^{\p}_{0}(X_{H},X_{G})J_{G}(atX_{G},\wh{f})
\\
&=
\sum_{X_{G}}\Del^{\p}_{0}(X_{H},X_{G})
\sum_{\CO\in \nil_{\reg}(\Fg(F)}
\Gam_{\CO}(aX_{G})J_{\CO}(\wh{f})
\\
&=
\sum_{X_{G}}
\Del^{\p}_{0}(X_{H},X_{G})
\sum_{\CO\in \nil_{\reg}(\Fg(F))}
\Gam_{\CO}(X_{G})J_{\CO}(\wh{f}).
\end{align*}
In particular, the limit is uniform in $a\in F^{\times}$. Hence by Moore-Osgood theorem (\cite[p139]{MR824243}), the order of the limit can be switched, from which we get
$$
\sum_{\CO\in \nil_{\reg}(\Fg(F))}
\Del(\CO)J_{\CO}(\wh{f}).
$$

It follows that we have establish the desired identity.
\end{proof}

Using the linear independence of the distributions $\{ J_{\CO}|\quad \CO\in \nil_{\reg}(\Fg)\}$ (c.f. \cite[1.8.2]{beuzart2015local}), the following corollary holds.

\begin{cor}\label{cor:eqfactor}
For any $\CO\in \nil_{\reg}(\Fg)$,
$$
\Del(\CO) = \sum_{X_{G}}
\Del^{\p}_{0}(X_{H},X_{G})\Gam_{\CO}(X_{G}).
$$
\end{cor}

Now we are ready to establish Theorem \ref{thm:germformula}. We follow the argument of \cite{regulargermshelstad}. 

For any character $\kappa$ of $\CE(T) = \mathrm{Im}(H^{1}(F,T^{\mathrm{sc}}_{G}(\wb{F}))\to H^{1}(F,T_{G}(\wb{F})))$, we can attach to it an endoscopic group $H=H(T_{G},\kappa)$, together with an admissible embedding $T_{H}\to T_{G} = G_{X}$ whose underlying Lie algebra homomorphism sends $X_{H}$ to $X$. For convenience we may assume that ${}^L H$ embeds admissibly into $\LG$.

By Remark \ref{rmk:transferstableconjugate}, for the identity in Corollary \ref{cor:eqfactor}, we may write each $X_{G}$ appearing in the summation as $X_{G} = g^{-1}Xg = X(\inv(X_{G},X))$, where $g\in G(\wb{F})$ and $\sig\to \sig(g)g^{-1}$ represents the element $\inv(X_{G},X)$ in $\CE(T)$. Then $\Del^{\p}_{0}(X_{H},X_{G}) = \langle \inv(X_{G},X),\kappa\rangle \cdot\Del^{\p}_{0}(X_{H},X)$. It follows that we arrive at the following identity
$$
\sum_{X_{G}}
\langle \inv(X,X_{G}),\kappa\rangle \Gam_{\CO}(X_{G}) = \frac{\Del(\CO)}{\Del^{\p}_{0}(X_{H},X)}.
$$ 
On the other hand, we may write
$$
\Del(\CO) = \langle  \inv_{T_{G}}(\CO),\kappa\rangle, 
\Del_{\mathrm{I}}(X_{H},X) = \langle \inv(T_{G}),\kappa \rangle, \Del_{\mathrm{II}}(X_{H},X) = \langle \inv_{T_{G}}(X), \kappa \rangle.
$$
Hence we get 
$$
\sum_{X_{G}}
\langle \inv(X,X_{G}),\kappa\rangle 
\Gam_{\CO}(X_{G}) = 
\langle \frac{\inv_{T_{G}}(\CO)}{\inv(T_{G})\inv(X)},\kappa \rangle.
$$
After summing over $\kappa$, we obtain Theorem \ref{thm:germformula}.

\subsection{Relation with the Kostant's sections}\label{sub:germformula:kostantsection}
We are going to point out the relation between Theorem \ref{thm:germformula} and Kostant's sections.

For the presentation of Kostant's sections, we follow \cite{kottwitztransfer}, \cite{drinfeldnotes} and \cite{MR158024}. In particular we assume that $G$ is split over $F$ and $T$ is a fixed split maximal torus in $G$ defined over $F$.

We first recall the Chevalley's isomorphism (\cite[6.7]{MR1433132}). Under the adjoint action of $G$ on $\Fg$ and $W$ on $\Ft$, the restriction morphism $F[\Fg]\to F[\Ft]$ yields an isomorphism $F[\Fg]^{G}\simeq F[\Ft]^{W}$. The associated morphism $u:\Fg\to \Ft/W$ sends $Z\in \Fg(F)$ to the $W$-orbit in $\Ft(F)$ consisting of elements that are $G(F)$-conjugate to the semi-simple part $Z_{s}$ of the Jordan decomposition $Z=Z_{s}+Z_{n}$. Here $Z_{s}$ is semi-simple, $Z_{n}$ is nilpotent, and $[Z_{s},Z_{n}] = 0$.

\begin{defin}\label{defin:regularelement}
An element $Z\in \Fg(F)$ is called \emph{regular} if the dimension of its centralizer in $\Fg(F)$ is equal to the dimension of $\Ft$.
\end{defin}

It is known that the set of regular elements is open and dense in $\Fg$ (\cite[I.3]{MR0180554}).

In \cite{MR158024}, over the algebraic closure $\wb{F}$, B. Kostant showed that $Z$ is regular if and only if the nilpotent part $Z_{n}$ of $Z$ is a regular element in the centralizer of $Z_{s}$ in $\Fg$. The map $Z\to Z_{s}$ induces a bijection between the set of regular $\Ad(G)$-orbits in $\Fg$ and the set of semi-simple $\Ad(G)$-orbits in $\Fg$. Moreover, using the morphism $u$, both sets of orbits can be identified with $\Ft/W$.

\subsubsection{Kostant's sections}
Let $B_{\infty}=TN_{\infty}$ be the Borel subgroup of $G$ defined over $F$ that is opposite to $B_{0}$. 

Kostant proved that every element in the $F$-points of the affine space $\Fb_{\infty}(F)+X_{+}\subset \Fg(F)$ is regular. For any $H\in \Ft(F)\subset \Fb_{\infty}(F)$, the semi-simple part of $H+X_{+}$ is conjugate to $H$. In particular, over the algebraic closure the set $\Ft+X_{+}$ meets every regular $\Ad(G)$-orbit in $\Fg$.

Let $\Fa = \Cent_{\Fg}(X_{-})$ (the choice of $\Fa$ has more freedom, see \cite[2.4]{kottwitztransfer}). Then over the algebraic closure $\wb{F}$, $\Fa+X_{+}$ meets every regular $\Ad(G)$-orbit exactly once. Over the rational field $F$, the composition of the closed embedding $\Fa+X_{+}\hookrightarrow \Fg$ and the morphism $u:\Fg\to \Ft/W$ is an isomorphism of algebraic varieties. For any $Z\in \Ft(F)+X_{+}$, there exists a unique element $n(Z)\in N_{\infty}(F)$ such that $\Ad(n(Z))(Z)\in \Fa(F)+X_{+}$, and the map $Z\to n(Z)$ gives a morphism of algebraic varieties $\Ft+X_{+}\to N_{\infty}$. It can be deduced from \cite{MR158024} that for any $Z\in \Fb_{\infty}(F)+X_{+}$ there exists a unique element $n(Z)\in N_{\infty}(F)$ such that $\Ad(n(Z))\in \Fa(F)+X_{+}$, and the map $Z\to n(Z)$ gives a morphism of algebraic varieties $\Fb_{\infty}+X_{+}\to N_{\infty}$. It follows that the map $(n,Y)\to \Ad(n)Y$ gives an isomorphism of algebraic varieties from $N_{\infty}\times (\Fa+X_{+})$ to $\Fb_{\infty}+X_{+}$.

In summary, we obtain the following commutative diagram over $F$.
\begin{align}\label{diagram:kostant}
\xymatrix{
\Fa+X_{+} \ar[dr]^{\simeq}\ar[r] & X_{+}+\Fb_{\infty} \ar[d]^{\text{$\Ad$-$N_{\infty}$ trivial bundle}}\ar[rr] & & \Fg \ar[d] \\
	& \Ft/W \ar[rr]^{\simeq} && \Fg \sslash G
}.
\end{align}

\subsubsection{Submersion}
Following the diagram (\ref{diagram:kostant}), let $\Sig$ be the image of the following morphism
\begin{align*}
G\times (\Fa+X_{+})&\to \Fg
\\
(g,Y)&\to g^{-1}Yg.
\end{align*}
Then there is the following commutative diagram
$$
\xymatrix{
G\times (\Fa+X_{+})\ar[r] \ar[dr]^{p_{2}} &\Sig\ar[r] \ar[d]^{p}	& \Fg \ar[d]	\\
 &(\Fa+X_{+}) \ar[r]^{\simeq}	& \Ft/W
}.
$$
Here $p_{2}$ is the canonical projection onto the second variable, and $p$ is the canonical morphism induced from $p_{2}.$ After composition, there is a morphism $\textbf{c}:\Sig\to \Ft/W$. After passing to the $F$-rational points, we get an $F$-analytical smooth submersion between two smooth $F$-analytic manifolds
$$
c_{F}:\Sig(F)\to (\Ft/W)(F).
$$
The $F$-analytical structure of $\Sig(F)$ is inherited from $\Fg(F)$, together with the fact that the set of regular elements in $\Fg$ are open and dense. Denote the measure on $\Sig(F)$ induced from $\Fg(F)$ by $\mu_{\Sig(F)}$. Similarly there is the canonical measure on $(\Ft/W)(F)$ which is denoted by $\mu_{(\Ft/W)(F)}$ (\cite[Proposition~3.29]{MR2779866}). The fibers of the morphism $c_{F}$ are given by the $G(F)$-orbits associated to the elements in $(\Fa+X_{+})(F)$ via the isomorphism $(\Fa+X_{+})\simeq \Ft/W$. In particular, the fibers of $c_{F}$ are all of the same dimension. By the theory of integration along fibers (\cite[p.61]{MR658304} for archimedean, \cite[7.6]{MR1743467} for $p$-adic), for any $Y\in (\Ft/W)(F)$, there exists a canonical chosen measure $\mu_{Y}$ on the $G(F)$-orbit associated to $Y$, such that the measure $\mu_{Y}$ varies $F$-analytically in $Y$ with the the following identity
$$
\int_{\Sig(F)}f(X)\mu_{\Sig(F)}(X) = 
\int_{(\Ft/W)(F)}F_{f}(Y)\mu_{(\Ft/W)(F)}(Y).
$$
From \cite[3.30]{MR2779866}, when the associated element in $(\Fa+X_{+})(F)$ is regular semi-simple, which is still denoted by $Y$ by abuse of notation,
$$
F_{f}(Y) = J_{G}(Y,f).
$$
It follows that if we choose a sequence $\{Y_{i} \}_{i\geq 1}$, $Y_{i}\in (\Fa+X_{+})^{\rss}(Y)$, such that $\lim_{i\to \infty}Y_{i} = N\in (\Fa+X_{+})(F)$ where $N$ is the unique regular nilpotent element in $(\Fa+X_{+})(F)$, then
$$
\lim_{i\to \infty}\mu_{Y_{i}} = \mu_{N},
$$
i.e.
$$
\lim_{i\to \infty}J_{G}(Y_{i},f) = J_{\CO_{N}}(f), \quad f\in C^{\infty}_{c}(\Sig(F)),
$$
where $\CO_{N}\in \nil_{\reg}(\Fg)$ is the regular nilpotent element associated to $N$.

\begin{rmk}
There is an alternative proof of the fact that $\lim_{i\to \infty}\mu_{Y_{i}} = \mu_{N}$ given in \cite[Lemme~11.4]{waldspurger10}. It is worth pointing out that in Waldspurger's proof, he chose an arbitrary sequence $\{ Y_{i}\}_{i\geq 1}$ where $Y_{i}\subset \Fg(F)^{\rss}$ and $\lim_{i\to \infty}Y_{i} = N$. But we can observe that the sequence satisfies the property that $Y_{i}\subset \Sig(F)$ whenever $i$ is sufficiently large. This follows from the fact that $\Sig(F)$ is open in $\Fg(F)$.
\end{rmk}

\subsubsection{The relation}
Now we discuss the relation between Theorem \ref{thm:germformula} and the Kostant sections. Recall that we have fixed an endoscopic data $(H,\CH, s,\xi)$ for $G$ and an $F$-splitting $\textbf{spl} = (B_{0},T,\{ X_{\alp}\})$.

We recall the following theorem and corollary proved by Kottwitz (\cite[Theorem~5.1, Corollary~5.2]{kottwitztransfer}).

\begin{thm}\label{thm:kottwitztransfer}
The transfer factor $\Del^{\p}_{0}(X_{H},X_{G})$ is equal to $1$ whenever $X_{G}$ lies in the set of $F$-rational points of $\Fb_{0}+X_{-}$.
\end{thm}
\begin{cor}\label{cor:kottwitztransfer}
$$\Del^{\p}_{0}(X_{H},X_{G}) = \langle \inv(X_{G},X^{\p}_{G}),\textbf{s}_{T_{G}}\rangle
$$
where $X^{\p}_{G}$ is any $F$-rational element in $\Fb_{0}+X_{-}$ that is stably conjugate to $X_{G}$.
\end{cor}

Following the construction of Kostant's sections, we can define a linear subspace $\Fa$ of $\Fb_{0}$ such that $\Fa+X_{-}$ yields a section for the adjoint quotient map $\Fg\to \Ft/W$.

Now as before let $\Sig$ be the image of $G\times (\Fa+X_{-})\to \Fg$ where $G$ acts on $\Fa+X_{-}$ via adjoint action. Then Corollary \ref{cor:kottwitztransfer} says that $\Del^{\p}_{0}(X_{H},X_{G}) = 1$ if and only if $X_{G}$ lies in $\Sig(F)$.

By Theorem \ref{thm:germformula}, for any $\CO\in \nil_{\reg}(\Fg)$ and $X\in \Fg^{\rss}(F)$, 
$$
\Gam_{\CO}(X)=  
\bigg\{
\begin{matrix}
1,	&	\inv(X)\inv(T_{G}) = \inv_{T_{G}}(\CO),\\
0,	& \text{\rm{otherwise}}.
\end{matrix}
$$
Note that the invariant $\inv_{T_{G}}(\CO)$ measures the difference between $\textbf{spl}_{\infty}$, the \emph{opposite} $F$-splitting to $\textbf{spl}$, and the $F$-splitting $\textbf{spl}(\CO)$ determined by $\CO$. In particular, up to $G(\wb{F})$-conjugation, for convenience we may set $\textbf{spl}(\CO) = \textbf{spl}_{\infty}$, i.e. $X_{-}\in \CO$ and $\inv_{T_{G}}(\CO ) =1$. Then the orbit $\CO$ lies in $\Sig(F)$, and the formula can be simplified to say that $\Gam_{\CO}(X) =1$ if and only if $\inv(X)\inv(T_{G}) = 1$, which, is equivalent to say that $\Del^{\p}_{0}(X_{H},X) =1$ identically. On the other hand, by Kottwitz's theorem \ref{thm:kottwitztransfer}, it is also equivalent to say that $X$ lies in $\Sig(F)$. In particular, both the $G(F)$-orbit of $X$ and $\CO$ are contained in $\Sig(F)$, i.e. they lie in the image of the $G(F)$-orbit of a common Kostant's section.

Conversely, assume both the $G(F)$-orbit of $X$ and $\CO$ lie in the image of the $G(F)$-orbit of a common Kostant's section, which again is denoted by $\Sig(F)$. Up to $G(\wb{F})$-conjugation for $\textbf{spl}_{\infty}$, again we may assume that $\inv_{T_{G}}(\CO ) =1$. Then Theorem \ref{thm:kottwitztransfer} says that $\Del^{\p}_{0}(X_{H},X_{G})$ is identically $1$, which is equivalent to say that $\inv(X)\inv(T_{G}) = 1$. Hence $\Gam_{\CO}(X) = 1$.

In conclusion, the following theorem has been proved.

\begin{thm}
With the above notations, $\Gam_{\CO}(X)=  1$ if and only if the $G(F)$-orbit of $X$ and $\CO$ lie in the $G(F)$-orbit of a common Kostant's section.
\end{thm}

\section{Calculation of germs}\label{sec:germcal}
In this section, we are going to compute the endoscopic invariants defined in Section \ref{sub:germformula:transferfactors} for even dimensional quasi-split special orthogonal Lie algebras. We follow the strategy of \cite[Chapitre~I,X]{wald01nilpotent}. Using the regular germ formula proved in Section \ref{sub:germformula:theformula}, we give explicit formulas for the regular nilpotent germs for some regular semi-simple conjugacy classes in an even dimensional quasi-split special orthogonal Lie algebra. Finally using the explicit formulas we can determine the intersection of $\nil_{\reg}(\Fg)$ and $\Gam(\Sig)$ following the strategy of \cite[11.5,11.6]{waldspurger10}. The results established in the section will be important determining the nilpotent orbit support of the geometric expansion of the trace formula.

Throughout the section, fix an even dimensional quasi-split quadratic space $(V,q_{V})$ of dimension $d$ defined over $F$, with associated (special) orthogonal group $\mathrm{O}(V)$ ($\SO(V)$) and corresponding Lie algebra $\Fs\Fo(V)$.

\subsection{Some regular semi-simple conjugacy classes}\label{sub:germcal:conjugacy}
We are going to review the parametrization of some regular semi-simple conjugacy classes (without eigenvalue $0$) in $\Fs\Fo(V)(F)$, and prove some basic properties for the parametrization.

The following data is fixed.
\begin{itemize}
\item A finite set $I$;

\item For any $i\in I$, fix a finite field extension $F^{\#}_{i}$ of $F$ together with a $2$-dimensional $F^{\#}_{i}$ commutative algebra $F_{i}$. Denote $\tau_{i}$ the unique nontrivial automorphism of $F_{i}$ over $F^{\#}_{i}$.

\item For any $i\in I$, fix constants $a_{i},c_{i}\in F^{\times}_{i}$.
\end{itemize}

For fix data as above, we make further assumptions as follows.
\begin{itemize}
\item For any $i\in I$, $a_{i}$ generates $F_{i}$ over $F$;

\item For any $i,j\in I$ with $i\neq j$, there do not exist any $F$-linear isomorphisms between $F_{i}$ and $F_{j}$ that send $a_{i}$ to $a_{j}$;

\item For any $i\in I$, $\tau_{i}(a_{i}) = -a_{i}$, and $\tau_{i}(c_{i}) = c_{i}$;
\end{itemize}
For any $i\in I$, note $\sgn_{F_{i}/F^{\#}_{i}}$ the quadratic character of $F^{\#\times}_{i}$ associated to $F_{i}$ via local class field theory (\cite[Theorem~2]{MR0220701}). Let $I^{*}$ be the subset of $I$ consisting of $i\in I$ such that $F_{i}$ is a field, i.e. $\sgn_{F_{i}/F^{\#}_{i}}$ is nontrivial.

With the above assumptions set $W = \bigoplus_{i\in I}F_{i}$ and define a quadratic space $(W,q_{W})$ via 
$$
q_{W}(\sum_{i\in I}w_{i}, \sum_{i\in I}w^{\p}_{i}) = \sum_{i\in I}[F_{i}:F]^{-1}\tr_{F_{i}/F}
(\tau_{i}(w_{i})w^{\p}_{i}c_{i}), \quad w_{i},w^{\p}_{i}\in F_{i}.
$$
Using the fact that $\tau_{i}(c_{i}) = c_{i}$ we immediately find that the bilinear form $q_{W}$ is indeed symmetric. 

Note $X_{W}$ the element in $\End_{F}(W)$ defined by 
$$
X_{W}(\sum_{i\in I}w_{i}) = 
\sum_{i\in I}a_{i}w_{i}, \quad w_{i}\in F_{i}.
$$
Then $X_{W}\in \Fs\Fo(W)(F)$ which follows directly from the fact that $\tau_{i}(a_{i}) = -a_{i}$. 

We establish some basic properties for the quadratic space $(W,q_{W})$.

\begin{lem}\label{lem:gemcal:basisform}
For any finite-dimensional $F$-algebra $E$, let $\Sig(E)$ be the set of nonzero $F$-algebra homomorphisms from $E$ to $\wb{F}$. Then there exists a basis $\{f_{\sig}\}_{\sig\in \Sig(E)}$ of $\wb{F}\otimes_{F}E$, such that for any $v\in E\subset \wb{F}\otimes_{F}E$,
$$
v=\sum_{\sig\in \Sig(E)}\sig(v)f_{\sig}.
$$
\end{lem}
\begin{proof}
Let $k=[E:F]$, which is also the cardinality of the set $\Sig(E)$. We fix any $F$-basis $\{ b_{\sig}\}_{\sig\in \Sig(E)}$ of $E$, and assume that we have 
$$
f_{\sig} = \sum_{\tau\in \Sig(E)}c^{\tau}_{\sig}\otimes b_{\tau}
$$
for some constant $c^{\tau}_{\sig}\in \wb{F}$. Then the equality
$$
v = \sum_{\sig\in \Sig(E)}\sig(v)f_{\sig}
$$
is equivalent to the following identities 
\begin{align*}
&\sum_{\sig\in \Sig(E)}\sig(b_{\tau})c^{\tau}_{\sig} = 1,\quad \forall \tau\in \Sig(E),
\\
&\sum_{\sig\in \Sig(E)}\sig(b_{\tau})c^{\xi}_{\sig} = 0, \quad \forall \xi\in \Sig(E), \tau \neq \xi.
\end{align*}
Let $B=  (\sig(b_{\tau}))\in \RM_{k\times k}(\wb{F})$, where the rows are indexed by $\sig\in \Sig(E)$, and columns are indexed by $\tau\in \Sig(E)$. Similarly let $C=  ((c^{\tau}_{\sig}))\in \RM_{k\times k}(\wb{F})$ where the rows are indexed by $\tau$ and the columns are indexed by $\tau$. Then the above equalities are equivalent to the following matrix identity
$$
BC=\RI_{k}.
$$
Using the fact that $\{ b_{\tau}\}_{\tau\in \Sig(E)}$ is an $F$-basis of $E$, we realize that the matrix $B$ is invertible, hence we can simply let $C=B^{-1}$ and the lemma follows.
\end{proof}

As a corollary, we can establish the following fact.

\begin{cor}\label{cor:germcal:galoispermute}
For any $\zet\in \Gam_{F}$, we have 
$$
\zet(f_{\sig}) = f_{\zet\sig}.
$$
\end{cor}
\begin{proof}
We use the same notation as in Lemma \ref{lem:gemcal:basisform}. We write
$$
f_{\sig} = \sum_{\tau\in \Sig(E)}c^{\tau}_{\sig}\otimes b_{\tau}.
$$
Then for any $\zet\in \Gam_{F}$, $\zet(f_{\sig}) = \sum_{\tau\in \Sig(E)}\zet(c^{\tau}_{\sig})\otimes b_{\tau}$.

Since 
$$
\zet(B)\zet(C) = \zet(\RI_{k}) = \RI_{k},
$$
and we know that $\zet(B) =  (\zet \sig(b_{\tau}))$, we get $\zet(C)=  ((c^{\tau}_{\zet\sig}))$. It follows that
$$
\zet(f_{\sig}) = f_{\zet \sig}.
$$
\end{proof}

Now for any $i\in I$, we define $\Sig(F_{i})$ to be the set of nonzero $F$-algebra homomorphisms from $F_{i}$ to $\wb{F}$. We let $\Sig(W) = \{ (i,\sig)|\quad i\in I, \sig\in \Sig(F_{i}) \}$. From Lemma \ref{lem:gemcal:basisform}, there exists a basis for $\wb{F}\otimes_{F}W$, which we denote by $\{f_{i,\sig}|\quad (i,\sig)\in \Sig(W) \}$, such that for any $i\in I$ and $v_{i}\in F_{i}\subset V$,
$$
v_{i} = \sum_{\sig\in \Sig(F_{i})}\sig(v_{i})f_{i,\sig}.
$$

\begin{lem}\label{lem:germcal:basisproperty}
The following properties hold,
\begin{enumerate}
\item For any $(i,\sig)\in \Sig(W)$,
$$
X(f_{i,\sig}) = \sig(a_{i})f_{i,\sig}.
$$

\item For any $(i,\sig)$ and $(i^{\p},\sig^{\p})\in \Sig(W)$,
$$
q_{W}(f_{i^{\p},\sig^{\p}}, f_{i,\sig}) = 
\bigg\{
\begin{matrix}
0,	& (i^{\p},\sig^{\p})\neq (i,\sig\tau_{i}),\\
\sig(c_{i})[F_{i}:F]^{-1},	& (i^{\p},\sig^{\p}) = (i,\sig\tau_{i}).
\end{matrix}
$$
\end{enumerate}
\end{lem}
\begin{proof}
We establish (1). By definition, $Xv_{i} = a_{i}v_{i}$. We write
$$
v_{i} = \sum_{\sig\in \Sig(F_{i})}\sig(v_{i})f_{i,\sig}.
$$
Then we get the following equality
$$
\sum_{\sig\in \Sig(F_{i})}
\sig(v_{i})X(f_{i,\sig}) = \sum_{\sig\in \Sig(F_{i})}\sig(a_{i}v_{i})f_{i,\sig}.
$$
It follows that we have $X(f_{i,\sig}) = \sig(a_{i})f_{i,\sig}$.

Now for (2), when $i\neq i^{\p}$, $q_{W}(f_{i^{\p},\sig^{\p}},f_{i,\sig}) =0$ holds automatically. We only consider the case when $i=i^{\p}$. For any $v_{i},w_{i}\in F_{i}$, by definition,
$$
q_{W}(v_{i},w_{i}) = [F_{i}:F]^{-1}
\tr_{F_{i}/F}(\tau_{i}(v_{i})w_{i}c_{i}) = 
[F_{i}:F]^{-1}
\sum_{\sig\in \Sig(F_{i})}
(\sig\tau_{i})(v_{i})\sig(w_{i})\sig(c_{i}).
$$
On the other hand, by linearity, we have 
$$
q_{W}(v_{i},w_{i}) = 
\sum_{\xi,\sig\in \Sig(F_{i})}
\xi(v_{i})\sig(w_{i})q_{W}(f_{i,\xi}, f_{i\sig}).
$$
Comparing the two identities, we get the desired result.
\end{proof}

We establish the following lemma.

\begin{lem}\label{lem:germcal:centralizer}
The element $X_{W}$ is regular semi-simple, and the centralizer of $X_{W}$ in $\SO(W)$ $(\text{actually $\mathrm{O}(W)$})$ is isomorphic to
$$
\prod_{i\in I\bs I^{*}}
F^{\#\times}_{i}
\prod_{i\in I^{*}}\ker(\mathrm{N}_{F_{i}/F_{i}^{\#}}:F_{i}^{\times}\to F^{\#\times}_{i}).
$$
\end{lem}
\begin{proof}
Over the algebraic closure, under the fixed basis $\{ f_{i,\sig}|\quad (i,\sig)\in \Sig(W)\}$, we have 
$$
X(f_{i,\sig}) = \sig(a_{i})f_{i,\sig},
$$
therefore $X$ is semi-simple.

Now by definition, for any $i,j\in I$, there do not exist any $F$-linear isomorphisms between $F_{i}$ and $F_{j}$ sending $a_{i}$ to $a_{j}$, and $a_{i}$ generates $F_{i}$ over $F$. Hence the eigenvalues of $X$, which we denote by $\{ \sig(a_{i})\}_{(i,\sig)\in \Sig(W)}$, are pairwise different and do not contain $0$. It follows that $X_{W}$ is regular semi-simple, and the centralizer of $X_{W}$ in $\mathrm{O}(W)$ coincides with its centralizer in $\SO(W)$.

For any $g\in \GL(W)$ that commutes with $X_{W}$, we have 
$$
X_{W}(\sum_{i\in I}g(w_{i})) = \sum_{i\in I}a_{i}g(w_{i}),\quad w_{i}\in F_{i}.
$$
Since there are no $F$-linear isomorphisms between $F_{i}$ and $F_{j}$ sending $a_{i}$ to $a_{j}$, we get 
$$
g(w_{i})\in F_{i}.
$$
It follows that $g\in \prod_{i\in I}\GL(F_{i})$. Moreover, since $g\in \SO(W)$, we realize that $g\in \prod_{i\in I}\mathrm{O}(F_{i})$. Here $\mathrm{O}(F_{i})$ is the orthogonal group associated to the quadratic form $(F_{i}, q_{F_{i}})$, where 
$$
q_{F_{i}}(w_{i},w^{\p}_{i}) = [F_{i}:F]^{-1}\tr_{F_{i}/F}(\tau_{i}(w_{i})w^{\p}_{i}c_{i}),\quad w_{i},w^{\p}_{i}\in F_{i}.
$$.

For any $i\in I$, $F^{\times}_{i}$ is a maximal torus in $\GL(F_{i})$ that commutes with $X_{W}$ by definition. Hence we only need to compute $F_{i}^{\times}\cap \mathrm{O}(F_{i})$. For any $a\in F^{\times}_{i}$, the identity
$$
q_{F_{i}}(aw_{i},aw_{i}^{\p}) = 
q_{F_{i}}(w_{i},w_{i}^{\p}), \quad w_{i},w_{i}^{\p}\in F_{i}
$$
is equivalent to 
$$
\tau_{i}(a)a = 1,
$$
which is equivalent to say that 
\begin{align*}
a\in 
\bigg\{
\begin{matrix}
\ker(\RN_{F_{i}/F^{\#}_{i}}:F^{\times}_{i}\to F^{\#\times}_{i}),	& \text{ if $F_{i}/F$ is a field extension},\\
\{(z,z^{-1})|\quad z\in F^{\#\times}_{i}\}, & \text{ if $F_{i}/F$ is not a field extension}.
\end{matrix}
\end{align*}
It follows that we have established the lemma.
\end{proof}

Now we assume that $d=\sum_{i\in I}[F_{i}:F]=\dim W=\dim V$. If there exists an isomorphism between $(W,q_{W})$ and $(V,q_{V})$, we fix such an isomorphism, and then $X_{W}$ can be identified with a regular semi-simple element $X$ in $\Fs\Fo(V)(F)$. We note that the $\SO(V)$-orbit of $X$ depends on the choice of the isomorphism between $(W,q_{W})$ and $(V,q_{V})$. We note $\mathrm{O}(V)\simeq \SO(V)\ltimes \{1,O \}$, where $O$ is the outer automorphism. Then as we have seen in the proof of Lemma \ref{lem:germcal:centralizer}, the centralizer of $X$ in $\mathrm{O}(V)$ coincides with the centralizer of $X$ in $\SO(V)$. Hence conjugation by the outer automorphism $O$ yields another $\SO(V)$-orbit, which can be obtained from $X_{W}$ as well through composing the original isomorphism between $(W,q_{W})$ and $(V,q_{V})$ with the outer automorphism $O$. In particular we notice that the two orbits are not $\SO(V)$ stable conjugate.

For any triple $(I,(a_{i}),(c_{i}))$, assume that there exists an isomorphism $(W,q_{W})\simeq (V,q_{V})$. Then there are two possible $F$-rational $\SO(V)$-orbits permuted by the outer automorphism $O$, which are denoted by $\CO^{+}(I, (a_{i}),(c_{i}))$ and $\CO^{-}(I, (a_{i}),(c_{i}))$. We denote the associated $F$-rational $\SO(V)$ stable conjugacy classes by $\CO^{\mathrm{st,+}}(I,(a_{i}),(c_{i}))$ and $\CO^{\mathrm{st,+}}(I,(a_{i}),(c_{i}))$. Finally let the union of the two (stable) conjugacy classes by $\CO(I, (a_{i}),(c_{i}))$ ($\CO^{\mathrm{st}}(I,(a_{i}),(c_{i})$)).

We have the following lemma.

\begin{lem}\label{lem:germcal:stableconjugate}
Given two triples $(I, (a_{i}),(c_{i}))$ and $(I^{\p},(a^{\p}_{i}),(c_{i}^{\p}))$ together with isomorphisms between $(W,q_{W})$, $(W^{\p},q_{W^{\p}})$ and $(V,q_{V})$. After fixing isomorphisms between them, we have the associated orbits in $\SO(V)$.

Then 
$\CO^{\mathrm{st}}(I,(a_{i}),(c_{i})) = \CO^{\mathrm{st}}(I^{\p},(a^{\p}_{i}),(c^{\p}_{i}))$ if and only if the following properties hold.
\begin{itemize}
\item There exists a bijection $\phi:I\to I^{\p}$;

\item For any $i\in I$, an $F$-linear isomorphism $\sig_{i}:F^{\p}_{\phi(i)}\to F_{i}$ satisfying
$$
\sig_{i}(a^{\p}_{\phi(i)}) = a_{i}, \quad \forall i\in I.
$$
\end{itemize}

Furthermore, $\CO(I,(a_{i}),(c_{i})) = \CO(I^{\p},(a^{\p}_{i}),(c^{\p}_{i}))$ if the following additional property is satisfied.
\begin{itemize}
\item For any $i\in I$, we have 
$$
\sgn_{F_{i}/F^{\#}_{i}}
(c_{i}\sig_{i}(c^{\p}_{\phi(i)})^{-1}) = 1.
$$
\end{itemize}
\end{lem}
\begin{proof}
We notice that $\CO^{\mathrm{st}}(I,(a_{i}),(c_{i})) = \CO^{\mathrm{st}}(I^{\p},(a^{\p}_{i}),(c^{\p}_{i}))$ if and only if $X_{W}$ and $X_{W^{\p}}$ share the same characteristic polynomial. Using the notation from Lemma \ref{lem:germcal:centralizer}, the roots of the characteristic polynomial of $X_{W}$ (resp. $X_{W^{\p}}$) are given by the set $\{\sig(a_{i}) \}_{(i,\sig)\in \Sig(W)}$ (resp. $\{\sig^{\p}(a^{\p}_{i^{\p}}) \}_{(i^{\p},\sig^{\p})\in \Sig(W^{\p})}$). It follows that the if part follows immediately. For the only if part, we notice that by definition $a_{i}$ generates $F_{i}$ over $F$, and for any $i,j\in I$, there do not exist $F$-linear isomorphism between $F_{i}$ and $F_{j}$ sending $a_{i}$ to $a_{j}$. Moreover, for any $i\in I$, the roots related to $F_{i}$ are exactly given by $\{ \sig(a_{i})\}_{i\in \Sig(F_{i})}$. Then the only if part is also straightforward.

Now to determine the conjugacy classes within the stable conjugacy classes, we may assume that $I=I^{\p}$, $a_{i} = a^{\p}_{i}$, and $\sig_{i} = \Id_{F_{i}}$. Then for the two orbits $\CO(I,(a_{i}),(c_{i}))$ and $\CO(I,(a_{i}), (c_{i}^{\p}))$, again from Lemma \ref{lem:germcal:centralizer}, over the algebraic closure $\wb{F}$, under the basis $\{f_{i,\sig} \}_{(i,\sig)\in \Sig(W)}$ both $X_{W}$ and $X_{W^{\p}}$ are diagonal with same eigenvalues. Hence the two orbits are the same if and only if there exists an element in the centralizer of $X_{W}$ giving rise to an isomorphism between the two quadratic spaces $(W,q_{W})\simeq (W^{\p},q^{\p}_{W})$. But this follows directly from Lemma \ref{lem:germcal:centralizer}. 
\end{proof}

\subsection{Computation of the endoscopic invariants}\label{sub:germcal:computeendoscopic}
We are going to compute the endoscopic invariants defined in Section \ref{sub:germformula:transferfactors} for the conjugacy classes defined in previous section. We will freely use notations from Section \ref{sub:germformula:transferfactors}.

\subsubsection{Coordinates and notations}
We first fix the coordinates for the quadratic space $(V,q_{V})$.

When $(V,q_{V})$ is split, we fix $\eta\in F^{\times}$ and an $F$-basis $\{e_{j}|\quad j=1,...,d \}$ of $V$ such that for any $j,k\in \{1,...,d \}$ we have 
$$
q_{V}(e_{j},e_{k}) = 
\bigg\{
\begin{matrix}
0,	& \text{ if $j+k\neq d+1$},\\
(-1)^{j+1}\eta/2, & \text{ if $j+k=d+1,j\leq d/2$}.
\end{matrix}
$$
Let $B_{0}=TN$ be the $F$-rational Borel subgroup stabilizing the totally isotropic flag
$$
\langle e_{1} \rangle \subset...\subset \langle e_{1},...,e_{d/2} \rangle.
$$
Then $T_{0}$ acts on the basis $\{e_{i} \}_{i=1}^{d}$ via scaling.

With the choice of the Borel pair $(B_{0},T)$, we have a natural set of simple roots $\Del = \Del(B_{0},T) = \{\alp_{j}|\quad j=1,...,d/2 \}$ as follows. For $j\in \{1,...,d/2-1 \}$, we let $X_{\alp_{j}}$ be the element in $\Fn$ that annihilates $e_{k}$ for $k\neq j+1,...,d+1-j$ and 
$$
X_{\alp_{j}}(e_{j+1}) = e_{j}, \quad X_{\alp_{j}}(e_{d+1-j}) = e_{d-j}.
$$
The element $X_{\alp_{d/2}}\in \Fn$ annihilates $e_{k}$ for $k\neq d/2+1,d/2+2$ and 
$$
X_{\alp_{d/2}}(e_{d/2+1}) = e_{d/2-1}, \quad 
X_{\alp_{d/2}}(e_{d/2+2}) = e_{d/2}.
$$
We can identify the Weyl group $W=W(G,T)$ as the group of permutations $w$ of the set $\{1,...,d \}$ corresponding to the basis $\{e_{j}|\quad j=1,...,d \}$ such that 
\begin{itemize}
\item for any $j\in \{1,...,d \}$, $w(j)+w(d+1-j) = d+1$;

\item the set $\{j|\quad 1\leq j\leq d/2, d/2+1\leq w(j)\leq d \}$ has even cardinality.
\end{itemize}

We also define a function 
$$
r: \{1,...,d \}\times W\to \BN
$$
via 
$$
r(j,w) = |\{ k\in \BN|\quad j<k\leq d, w(j)>w(k), j+k\neq d+1 \}|,\quad j\in \{1,...,d \}, w\in W.
$$

When $(V,q_{V})$ is quasi-split but not split, following the notation in Section \ref{sub:ggptriples:definitions} there exists a unique quadratic extension $E$ of $F$ such that $V\otimes_{F}E,q_{V})$ is split. The group $\Gal(E/F) = \{1,\tau \}$ acts on $V\otimes_{F}E$ in the following way. As in the split case, we may fix $\eta\in F^{\times}$ and a basis $\{e_{j} |\quad j=1,...,d \}$ of $V\otimes_{F}E$ with prescribed properties as above, then 
\begin{itemize}
\item for any $j\in \{1,...,d \}\bs \{ d/2,d/2+1 \}$, $e_{j}$ is fixed by $\tau$;

\item $\tau_{e_{d/2}} = e_{d/2+1}$ and $\tau(e_{d/2+1}) = e_{d/2}$.
\end{itemize}

We first establish the following lemma.

\begin{lem}\label{lem:germcal:weylsign}
For any $w\in W$ and $j\in \{1,...,d \}$, we have 
$$
n(w)(e_{j}) = (-1)^{r(j,w)}e_{w(j)}.
$$
\end{lem}
\begin{proof}
For any $\alp\in \Del= \Del(B_{0},T)$, we let $X_{-\alp}$ be the unique element in $\wb{\Fn}$ which is the radical space associated to $-\alp$ such that 
$$
[X_{\alp},X_{-\alp}] = d\check{\alp}(1)
$$
and $w_{\alp}$ the simple reflection associated to $\alp$.

We note that the Weyl group $W$ has a length function $\ell$ (\cite[5.2]{MR1066460}). Every Weyl element in $W$ has a reduced expression which can be written as the product of simple reflections.

The section $n$ is characterized uniquely by the following two properties (\cite[2.1]{transferfactorLS}).
\begin{itemize}
\item For any $\alp\in \Del$, 
$$
n(w_{\alp}) = \exp(X_{\alp})\exp(-X_{-\alp})\exp(X_{\alp});
$$

\item For any $w,w^{\p}\in W$ satisfying $\ell(w)+\ell(w^{\p}) = \ell(ww^{\p})$, we have 
$$
n(ww^{\p}) = n(w)n(w^{\p}).
$$
\end{itemize}

Therefore we only need to establish the lemma for simple reflections and the following fact.

\begin{num}
\item\label{num:ggpcal:lengthfact}
For any $j\in \{1,...,d\}$, and $w_{1},w_{2}\in W$ satisfying $\ell(w_{1})+\ell(w_{2}) = \ell(w_{1}w_{2})$ and $w_{2}$ being a simple reflection, we have 
$$
r(j,w_{1}w_{2})+r(j,w_{2})+r(w_{2}(j),w_{1})\equiv 0 \mod 2.
$$
\end{num}

We first establish the lemma for simple reflections. For any $t\in \{1,...,d/2 \}$, we let $w_{t}$ be the simple reflection associated to $\alp_{t}$. By explicit computation, we have 
\begin{align*}
&n(w_{t})(e_{j}) = 
\Bigg\{
\begin{matrix}
e_{t},	& j=t+1, & r(j,w) = |\emptyset| = 0,\\
-e_{t+1}, & j=t, & r(j,w) = |\{t+1 \}| = 1,\\
e_{d-t}, & j=d+1-t,& r(j,w) = |\emptyset| = 0,\\
-e_{d+1-t}, & j=d-t,, & r(j,w) = |\{ d+1-t \}| =1,\\
e_{j}, &\text{others}, & r(j,w) = |\emptyset| = 0.
\end{matrix}
,\quad t\in \{ 1,...,d/2-1 \}
\\
\\
&n(w_{d/2})(e_{j}) = 
\Bigg\{
\begin{matrix}
-e_{d/2+1}, & j = d/2-1, & r(j,w) = |\{ d/2+1 \}| = 1,\\
-e_{d/2+2}, & j = d/2, & r(j,w) = |\{ d/2+2 \}| = 1,\\
e_{d/2-1},& j=d/2+1, & r(j,w) = |\emptyset| = 0,\\
e_{d/2}, & j=d/2+2, & r(j,w) = |\emptyset| = 0,\\
e_{j}, & \text{others}, & r(j,w) = |\emptyset| =0.
\end{matrix}
\end{align*}
It follows that the lemma holds for any simple reflections.

Now we are going to establish (\ref{num:ggpcal:lengthfact}).

By definition,
\begin{align*}
&r(j,w_{1}w_{2}) = |\{\alp\in \BN|\quad j<\alp, w_{1}w_{2}(j)>w_{1}w_{2}(\alp), j+\alp\neq d+1 \}|,
\\
&r(j,w_{2}) = |\{\bet\in \BN|\quad j<\bet, w_{2}(j)>w_{2}(\bet),j+\bet \neq d+1 \}|,
\\
&r(w_{2}(j),w_{1}) = |\{ \gam\in \BN|\quad w_{2}(j)<\gam, w_{1}w_{2}(j)>w_{1}(\gam), w_{2}(j)+\gam\neq d+1 \}|.
\end{align*}
For the set appearing in the definition of $r(w_{2}(j),w_{1})$, we let $\xi = w^{-1}_{2}(\gam)$, then 
$$
r(w_{2}(j),w_{1}) = |
\{\xi\in \BN|\quad w_{2}(j)<w_{2}(\xi), w_{1}w_{2}(j)>w_{1}w_{2}(\xi), j+\xi\neq d+1  \}|.
$$
We can separate the set 
$$
\{\xi\in \BN|\quad w_{2}(j)<w_{2}(\xi), w_{1}w_{2}(j)>w_{1}w_{2}(\xi),j+\xi \neq d+1 \}
$$
to two disjoint parts, for abbreviation we denote them by $\{\xi<j \}$ and $\{\xi>j \}$. Then $\{\xi<j \}$ is a subset of the one appearing in the definition of $r(j,w_{1}w_{2})$. After subtracting $\{\xi<j \}$ from the set in the definition of $r(j,w_{1}w_{2})$, we are reduced to prove the following identity 
\begin{align*}
&|\{\alp\in \BN|\quad j<\alp, w_{2}(j)>w_{2}(\alp), w_{1}w_{2}(j)>w_{1}w_{2}(\alp), j+\alp\neq d+1 \}|
\\
&\equiv |\{ \bet\in \BN|\quad j<\bet, w_{2}(j)>w_{2}(\bet), j+\bet \neq d+1 \}|
\\
&+|\{ \xi\in \BN|\quad j>\xi, w_{2}(j)<w_{2}(\xi), w_{1}w_{2}(j)>w_{1}w_{2}(\xi), j+\xi\neq d+1 \}| \mod 2.
\end{align*}
The above set involving $\alp$ is a subset of the one involving $\bet$. Hence after subtraction we are reduced to prove the following identity whenever $\ell(w_{1})+\ell(w_{2}) = \ell(w_{1}w_{2})$,
\begin{align*}
&|\{\alp\in \BN|\quad j<\alp, w_{2}(j)>w_{2}(\alp), w_{1}w_{2}(j)<w_{1}w_{2}(\alp), j+\alp \neq d+1 \}|
\\
&\equiv
|\{ \xi\in \BN|\quad j>\xi, w_{2}(j)<w_{2}(\xi), w_{1}w_{2}(j)>w_{1}w_{2}(\xi), j+\xi\neq d+1 \}| \mod 2.
\end{align*}
For notational convenience, we let 
\begin{align*}
&A_{\alp} = \{\alp\in \BN|\quad j<\alp, w_{2}(j)>w_{2}(\alp), w_{1}w_{2}(j)<w_{1}w_{2}(\alp), j+\alp \neq d+1 \},
\\
&B_{\xi} = \{ \xi\in \BN|\quad j>\xi, w_{2}(j)<w_{2}(\xi), w_{1}w_{2}(j)>w_{1}w_{2}(\xi), j+\xi\neq d+1 \}.
\end{align*}
We assume that $w_{2} = w_{t}$ is a simple reflection, where $t\in \{ 1,...,d/2 \}$.

When $t\in \{1,...,d/2-1 \}$, the sets $A_{\alp}$ and $B_{\xi}$ have the chance to be nonempty only when $j=t,t+1, d-t,d-t+1$. When $j=t$, 
$A_{\alp}= \{ \alp\in \BN|\quad \alp = t+1 \text{ and } w_{1}(t+1)<w_{1}(t)\}$, and $B_{\xi} = \emptyset$. In particular, $A_{\alp}$ is non-empty (and singleton) only when $w_{1}(t+1)<w_{1}(t)$. But when $w_{1}(t+1) <w_{1}(t)$, we have $w_{1}w_{2}(t+1)>w_{1}w_{2}(t)$, therefore $\ell(w_{1}w_{2}) = \ell(w_{1}) - 1$, since the length $\ell(w)$ of a Weyl element $w$ is equal to the number of positive roots sent to negative roots by $w$ \cite[5.6]{MR1066460}, which contradicts the fact that $\ell(w_{1}w_{2}) = \ell(w_{1})+\ell(w_{2})$. It follows that $A_{\alp} = \emptyset$ as well. Hence we have the equality $|A_{\alp}| = |B_{\xi}|$. Similar arguments adapt to the situation when $j=t+1,d-t,d-t+1$.

When $t=d/2$, the sets $A_{\alp}$ and $B_{\xi}$ have the chance to be nonempty only when $j=d/2-1, d/2, d/2+1,d/2+2$. When $j=d/2-1$, $B_{\xi} = \emptyset$, and $A_{\alp} = \{ \alp\in \BN|\quad \alp = d/2+1 \text{ and }  w_{1}(d/2+1)<w_{1}(d/2-1)  \}$, then the same argument as above implies that $A_{\alp} = \emptyset$, and in particular $|A_{\alp}| = |B_{\xi}|$ as well. Similarly we can obtain the results for other cases.

Hence we have established (\ref{num:ggpcal:lengthfact}) and it follows that we have established the lemma.
\end{proof}

Now we fix a regular semi-simple element $X$ in $\Fs\Fo(V)(F)$ parametrized by a triple $(I,(a_{i}),(c_{i}))$, with a fixed identification $(V,q_{V})\simeq (W,q_{W})$. We let $T_{G}$ be the centralizer of $X$ in $\SO(V)(F)$. We may choose an element $x\in \SO(V)(\wb{F})$ such that $T_{G}=xTx^{-1}$. In particular, up to scaling the element $x$ should send the basis $\{ e_{i} \}_{i=1}^{d}$ to $\{ f_{i,\sig} \}_{(i,\sig)\in \Sig(V)}$. Using Lemma \ref{lem:germcal:basisproperty} we may construct an element $x$ as follows.

We fix a bijection $\del:\Sig(V)\to \{1,...,d \}$
such that for any $(i,\sig)\in \Sig(V)$ we have the equality
$$
\del(i,\sig\tau_{i}) + \del(i,\sig) = d+1.
$$
We also fix a mapping $\mu:\{1,...,d\}\to \wb{F}^{\times}$ such that for any $j\in \{1,...,d/2 \}$ if we put $(i,\sig) =\del^{-1}(j)$, then we have the equality
$$
\mu(j)\mu(d+1-j) = \eta/2 (-1)^{j+1}[F_{i}:F]\sig(c_{i})^{-1}, \quad j\leq d/2.
$$

Now we define $x\in \GL(\wb{F}\otimes_{F}V)$ by 
$$
x(e_{j}) = \mu(j)f_{\del^{-1}(j)}, \quad j\in \{1,...,d \}.
$$
By Lemma \ref{lem:germcal:basisproperty} $x$ lies in $\mathrm{O}(V)(\wb{F})$. Up to composing $\del$ with the simple reflection interchanging $1$ and $d$ stabilizing $\{2,...,d-1 \}$, we can assume that $x\in \SO(V)(\wb{F})$. Then by construction $x$ is a desired element satisfying $T_{G} = xTx^{-1}$.

We recall that we have defined a set of roots $R=R(T_{G},G)$ for $T_{G}$, which has an ordering inherited from $\Del =\Del(T,B_{0})$ via $x$.
We set the $a$-data $\{a_{\alp} \}_{\alp\in R}$ by the equality
$$
\check{\alp}\otimes a_{\alp} = 1\in T_{G}(\wb{F}), \quad \forall \alp\in R.
$$
Then we assume that $X$ is close to $0$, and the endoscopic invariants defined in Section \ref{sub:germformula:transferfactors} can be written as
\begin{align*}
&\inv(X)(\sig) = \prod_{\sig\in R, \alp>0, \sig^{-1}(\alp)<0}
\check{\alp}\circ \alp(X)\in T_{G}(\wb{F}), \quad \sig\in \Gam_{F},
\\
&\inv(T_{G})(\sig) = xn(w_{\sig})\sig(x^{-1})\in T_{G}(\wb{F}), \quad \sig\in \Gam_{F}.
\end{align*}

\subsubsection{Cohomology classes construction}
Before proving explicit formulas for the endoscopic invariants, we given an explicit construction of some $1$-cocycles.

We consider a chain of finite field extensions $F\subset F^{\#}_{0}\subset F_{0}$, where $F_{0}$ is a quadratic extension of $F^{\#}_{0}$. We fix the nontrivial involution $\tau_{0}\in \Gal(F_{0}/F^{\#}_{0})$ and let $T_{0} = \ker(\RN_{F_{0}/F^{\#}_{0}}:F_{0}^{\times}\to F^{\#\times}_{0})$. By local class field theory \cite[Theorem~2]{MR0220701} we know that $H^{1}(F,T_{0}(\wb{F})) \simeq \{\pm 1 \}$. We are going to give an explicit construction of the isomorphism.

We let $\Sig_{F_{0}}$ be the set of nonzero embeddings $\tau:F_{0}\hookrightarrow \wb{F}$ extending the fixed embedding $F\hookrightarrow \wb{F}$. We fix an element $1=1_{F_{0}}\in \Sig(F_{0})$, which we view as the canonical embedding of $F_{0}$ into $\wb{F}$.

\begin{lem}\label{lem:germcal:torusidentification}
We have the following identification
$$
T_{0}(\wb{F}) = \{y\in \Sig(F_{0})\to \wb{F}^{\times}|\quad y(\sig)y(\sig\tau_{0}) = 1 \}.
$$
The group $\Gam_{F}$ acts on $T_{0}(\wb{F})$ in the following way. For any $\tau\in \Gam_{F}$ and $y\in T_{0}(\wb{F})$,
$$
(\tau(y))(\sig) = \tau(y(\tau^{-1}\sig)), \quad \sig\in \Sig(F_{0}).
$$
\end{lem}
\begin{proof}
We first verify that $\Gam_{F}$ acts on $T_{0}(\wb{F})$ using the formula defined in the statement. For any $\tau_{1},\tau_{2}\in \Gam_{F}$, $y_{0}\in T(\wb{F})$ and $\sig\in \Sig(F_{0})$, we have 
\begin{align*}
(\tau_{1}\tau_{2}(y))(\sig) = 
(\tau_{1}\tau_{2})(y(\tau^{-1}_{2}\tau^{-1}_{1}\sig))
=\tau_{1}(\tau_{2}(y)(\tau_{1}^{-1}\sig))
=\tau_{1}(\tau_{2}(y))(\sig).
\end{align*}
Hence $\Gam_{F}$ acts on $T_{0}(\wb{F})$. 

The group structure of $T_{0}(\wb{F})$ is straightforward. We only need to verify that $T_{0}(\wb{F})^{\Gam_{F}} \simeq T_{0} = \{z\in F^{\times}_{0}|\quad z\tau_{0}(z) = 1 \}$. For any $\tau\in \Gam_{F}$ and $y\in T_{0}(\wb{F})^{\Gam_{F}}$,
$$
\tau(y)(1) = \tau(y(\tau^{-1})) = y(1),
$$
hence $y(\tau) = \tau(y(1))$. In particular if $\tau\in \Gam_{F_{0}}$, then $\tau 1_{F_{0}} = 1_{F_{0}}$. Hence $y(1)\in \wb{F}^{\times,\Gam_{F_{0}}} = F^{\times}_{0}$. It follows that 
$$
T(\wb{F})^{\Gam_{F}} = \{ y:\Sig(F_{0})\to F_{0}^{\times}|\quad y(1)\tau_{0}(y(1)) = 1, \sig(y)(1) = y(\sig) \},
$$
which can exactly be identified as the set $\{z\in F^{\times}_{0}|\quad z\tau_{0}(z) = 1 \}$ via the morphism $y\to y(1) = z$.
\end{proof}

Suppose that $\tau\to y_{\tau}$ is a $1$-cocycle of $\Gam_{F}$ valued in $T_{0}(\wb{F})$. Then for any $\sig,\tau\in \Gam_{F}$, we have $y_{\sig\tau} = y_{\sig}\sig(y_{\tau})$. Hence $y_{\sig\tau}(1) = y_{\sig}(1)\sig(y_{\tau})(1) = y_{\sig}(1)\sig(y_{\tau}(\sig^{-1}))$. In particular, when $\sig,\tau\in \Gam_{F_{0}}$, we have $y_{\sig\tau}(1) = y_{\sig}(1)\sig(y_{\tau}(1))$. Hence the morphism $\tau\to y_{\tau}(1)$ is a $1$-cocycle of $\Gam_{F_{0}}$ valued in $\wb{F}^{\times}$. By Hilbert 90 \cite[2.7]{MR0225922}, there exists $z\in \wb{F}^{\times}$ such that $y_{\tau}(1) = \tau(z)z^{-1}$ for any $\tau\in \Gam_{F_{0}}$. We fix such a $z$. For $\theta\in \Gam_{F^{\#}_{0}}$ such that $\theta|_{F_{0}}$ is nontrivial, we consider the element $y_{\theta}(1)z\theta(z)\in \wb{F}^{\times}$.

\begin{lem}\label{lem:germcal:indeptheta}
The element $y_{\theta}(1)z\theta(z)\in \wb{F}^{\times}$ is independent of $\theta$.
\end{lem}
\begin{proof}
For $\theta_{1},\theta_{2}\in \Gam_{F^{\#}_{0}}$ whose restriction to $F_{0}$ are nontrivial, we are going to establish the identity 
$y_{\theta}(1)\theta_{1}(z)z=  y_{\theta_{2}}(1)\theta_{2}(z)z$. Equivalently, we only need to show 
$$
\frac{y_{\theta_{1}}(1)}{y_{\theta_{2}}(1)} = 
\frac{\theta_{2}(z)}{\theta_{1}(z)}.
$$
By the cocycle condition on $y_{\tau}$, we have 
$$
y_{\theta_{1}}(1) = y_{\theta_{2}}(1)
\theta_{2}(y_{\theta_{2}^{-1}\theta_{1}}(\theta_{2}^{-1})).
$$
We notice that $\theta_{2}^{-1}\theta_{1}\in \Gam_{F_{0}}$, hence 
$$
y_{\theta^{-1}_{2}\theta_{1}}(\theta^{-1}_{2}) = y_{\theta_{2}^{-1}\theta_{1}}(1)^{-1} = 
z(\theta^{-1}_{2}\theta_{1})(z)^{-1}.
$$
It follows that 
$$
\frac{y_{\theta_{1}(1)}}{y_{\theta_{2}}(1)} = 
\theta_{2}(z(\theta^{-1}_{2}\theta_{1})(z^{-1})) =
\theta_{2}(z)\theta_{1}(z^{-1})
$$
and we have proved the lemma.
\end{proof}

We note the element $y_{\theta}(1)z\theta(z)\in \wb{F}^{\times}$ by $\wt{y}$.

\begin{lem}\label{lem:germcal:fixedfield}
The element $\wt{y}$ lies in $F^{\# \times}_{0}$.
\end{lem}
\begin{proof}
To show $\wt{y} = y_{\theta}(1)z\theta(z)$ lies in $F^{\#\times}_{0}$, equivalently we are going to show that 
$$
\sig(\wt{y}) = \sig(y_{\theta}(1))(\sig\theta)(z)\sig(z)= \wt{y} = y_{\theta}(1)\theta(z)z, \quad \forall \sig\in \Gam_{F^{\#}_{0}}.
$$ 
By $1$-cocycle condition, 
$$
y_{\sig\theta}(1) = y_{\sig}(1)\sig(y_{\theta}(\sig^{-1})).
$$

When $\sig\in \Gam_{F_{0}}$, we have $\sig(y_{\theta}(\sig^{-1})) = \sig(y_{\theta}(1))$, and $y_{\sig}(1) = \sig(z)z^{-1}$. Hence 
$$
\sig(\wt{y}) = 
\frac{y_{\sig\theta}(1)}{\sig(z)z^{-1}}
(\sig\theta)(z)\sig(z) =
y_{\sig\theta}(1)(\sig\theta)(z)z = \wt{y}.
$$

When $\sig\in \Gam_{F^{\#}_{0}}$ and the restriction to $F_{0}$ is nontrivial, we have $\sig(y_{\theta}(\sig^{-1})) = \sig(y_{\theta}(1)^{-1})$, and $y_{\sig\theta}(1) = (\sig\theta)(z)z^{-1}$. Hence 
$$
\sig(\wt{y}) = 
\frac{y_{\sig}(1)}{(\sig\theta)(z)z^{-1}}
(\sig\theta)(z)\sig(z) = 
y_{\sig}(1)\sig(z)z = \wt{y}.
$$

It follows that $\wt{y}\in F^{\#\times}_{0}$.
\end{proof}

We consider the element $\wt{y}\RN_{F_{0}/F^{\#}_{0}}(F^{\times}_{0})$ in $F^{\#\times}_{0}/\RN_{F_{0}/F^{\#}_{0}}(F^{\times}_{0})$.

\begin{lem}\label{lem:germcal:ind}
The element $\wt{y}\RN_{F_{0}/F^{\#}_{0}}(F^{\times}_{0})$ is independent of the choice of $z$, $1_{F_{0}}$, and the cocycle class $\tau\to y_{\tau}$ in $H^{1}(F,T_{0}(\wb{F}))$.
\end{lem}
\begin{proof}
We first show the element is independent of $z$.
If $\tau(z)z^{-1} = \tau(w)w^{-1}$ for any $\tau\in \Gam_{F_{0}}$, then $zw^{-1}\in F_{0}$, and 
$$
y_{\theta}(1)\theta(z)z = 
y_{\theta}(1)\theta(w)w
(\theta(w^{-1}z)w^{-1}z).
$$
Hence the claim follows.

Then  we show that the element is independent of the cocycle class of $\tau\to y_{\tau}$ in $H^{1}(F,T_{0}(\wb{F}))$. We let $\tau\to z_{\tau}$ be a $1$-coboundary. In other words, there exists $y\in T_{0}(\wb{F})$ such that $z_{\tau} = \tau(y)y^{-1}.$ In particular, $z_{\tau}(1) = (\tau(y))(1)y^{-1}(1) = \tau(y(\tau^{-1}))y^{-1}(1)$. When $\tau\in \Gam_{F_{0}}$, we have $z_{\tau}(1) = \tau(y(1))y^{-1}(1)$.
Hence we notice that the element associated to $y_{\tau}z_{\tau}$ with $\tau\in \Gam_{F_{0}^{\#}}$ whose restriction to $F_{0}$ is nontrivial is given by
$$
y_{\tau}(1)z_{\tau}(1)zy(1)\tau(zy(1)).
$$
Therefore we have
$$
\frac{y_{\tau}(1)z_{\tau}(1)zy(1)\tau(zy(1))}{y_{\tau}(1)z\tau(z)} = 
z_{\tau}(1)y(1)\tau(y(1))
=\tau(y(\tau^{-1}))\tau(y(1)) = 
\tau(y(1)^{-1}y(1)) = 1.
$$
Hence the claim follows.

Finally we show the element is independent of $1_{F_{0}}$. Different choices of $1_{F_{0}}$ yield an automorphism of $T_{0}(\wb{F})$, which, after applying the automorphism, do not affect the cocycle class of $\tau\to y_{\tau}$ in $H^{1}(F,T_{0}(\wb{F}))$. Hence the claim follows as well.

It follows that the lemma holds.
\end{proof}

From the discussion above we realize that we have a well-defined homomorphism from $H^{1}(F,T_{0}(\wb{F}))$ to $F^{\#,\times}_{0}/\RN_{F_{0}/F^{\#}_{0}}(F^{\times}_{0})\simeq \{\pm 1 \}$.

\begin{lem}\label{lem:germcal:isocoh}
The above construction gives an explicit isomorphism between $H^{1}(F,T_{0}(\wb{F}))$ and $F^{\#, \times}_{0}/\RN_{F_{0}/F^{\#}_{0}}(F^{\times}_{0})$.
\end{lem}
\begin{proof}
Following the above notations, we only need to show that if $\wt{y}\in \RN_{F_{0}/F^{\#}_{0}}(F^{\times}_{0})$, then the $1$-cocycle $\tau\to y_{\tau}$ is actually a $1$-cobundary. Now for $\theta\in \Gam_{F_{0}}$ whose restriction to $F_{0}$ is trivial, we have 
$$
\wt{y} = y_{\theta}(1)\theta(z)z,
$$
and by assumption $\wt{y} = \theta(w)w$ for some $w\in F_{0}^{\times}$. On the other hand, since $y_{\tau}(1) = \tau(z)z^{-1}$ for any $\tau\in \Gam_{F_{0}},$ up to replacing $z$ by $zw^{-1}$, we may assume that $w=1$. Hence 
$$
\wt{y} = y_{\theta}(1)\theta(z)z = 1.
$$
In other words, we have 
\begin{align*}
y_{\theta}(1) = 
\bigg\{
\begin{matrix}
\theta(z^{-1})z^{-1}, & \theta\in \Gam_{F_{0}^{\#}} \text{ and } \theta_{|F_{0}} \text{ is nontrivial}.
\\
\theta(z)z^{-1}, & \theta\in \Gam_{F_{0}}.
\end{matrix}
\end{align*}
We define a $1$-coboundary $\tau\to z_{\tau}$ as follows. We let $z_{\tau} = \tau(\xi)\xi^{-1}$, where for any $\theta\in \Sig(F_{0})$,
$$
\xi(\theta) = 
\bigg\{
\begin{matrix}	
z,	& \theta = 1,\\
z^{-1}, & \theta = \tau_{0},\\
1,	& \text{otherwise}.
\end{matrix}.
$$
Then we can verify that 
$y_{\tau}(z_{\tau})^{-1}(1) = 1$. Up to the $1$-coboundary, we may assume that $y_{\tau}(1) = 1$ for any $\tau\in \Gam_{F_{0}^{\#}}$. Similarly, since for any $\alp\in \Sig(F_{0})$, $\tau\to y_{\tau}(\alp)$ is a $1$-coboundary of $\Gam_{F_{0}}$ in $\wb{F}^{\times}$, we can run the same procedure as above, and up to multiplying $1$-coboundaries of the above form, we may assume that $\tau\to y_{\tau}$ is a trivial $1$-cocycle from $\Gam_{F_{0}^{\#}}$ in $T_{0}(\wb{F})$, i.e. $y_{\tau}(\alp) = 1$ for any $\alp\in \Sig(F_{0})$ and $\tau\in \Gam_{F_{0}^{\#}}$. Now since we are in characteristic zero situation, we know that the restriction morphism
$$
H^{1}(F,T_{0}(\wb{F}))\to H^{1}(F_{0}^{\#},T_{0}(\wb{F}))
$$
is an injection \cite[p.15]{serregaloiscoh}, whose morphism is exactly given by restricting cocycles in $H^{1}(F,T_{0}(\wb{F}))$ to $H^{1}(F_{0}^{\#},T_{0}(\wb{F}))$. Since $\tau\to y_{\tau}$ is trivial in $H^{1}(\Gam_{F_{0}^{\#}},T_{0}(\wb{F}))$, it follows that $\tau\to y_{\tau}$ is trivial in $H^{1}(F,T_{0}(\wb{F}))$.

It follows that the lemma has been proved.
\end{proof}

\subsubsection{The computation}
Now we are ready to compute the transfer factors defined in the beginning of the section. We let $P_{X}$ be the characteristic polynomial of $X$ acting on $\End_{\wb{F}}(V)$, and we let $P^{\p}(X)$ be its derivative. For any $i\in I^{*}$, we set 
$$
C_{i} = \eta[F_{i}:F]^{-1}c_{i}^{-1}a^{-1}_{i}
P^{\p}_{X}(a_{i}).
$$

From Lemma \ref{lem:germcal:centralizer}, we know that $\inv(X)\inv(T_{G})\in H^{1}(T_{G})\simeq \{ \pm 1\}^{I^{*}}.$

\begin{pro}\label{pro:germcal:computation}
We have 
$$
\inv(X)\inv(T_{G}) = (\sgn_{F_{i}/F_{i}^{\#}}(C_{i}))_{i\in I^{*}}.
$$
\end{pro}
\begin{proof}
We will use the explicit construction of cohomological invariants above to compute $\inv(X)\inv(T_{G})$. 

We fix an element $1_{F_{i}}$ of $\Sig(F_{i})$, which we use to identify a canonical $F$-embedding of $F_{i}$ in $\wb{F}$. For any $1$-cocycle $\lam$ of $\Gam_{F}$ valued in $T_{G}(\wb{F})$, we define a function
$$
\lam(i,\cdot):\Gam_{F_{i}^{\#}}\to \wb{F}^{\times}
$$
by 
$$
\lam(\tau)f_{i,1_{F_{i}}} = \lam(i,\tau)f_{i,1_{F_{i}}}
$$
for any $\tau\in \Gam_{F_{i}^{\#}}$. Then following the discussion above Lemma \ref{lem:germcal:indeptheta}, there exists $z_{i}\in \wb{F}^{\times}$ such that 
$$
\lam(i,\tau) = \tau(z_{i})z_{i}^{-1}
$$
for any $\tau\in \Gam_{F_{i}}$. Fix such a $z_{i}$ and $\theta_{i}\in \Gam_{F_{i}^{\#}}$ whose restriction to $F_{i}$ is nontrivial. Then from Lemma \ref{lem:germcal:isocoh} we have the following equality 
$$
\lam = (\sgn_{F_{i}/F^{\#}_{i}}(\lam(i,\theta_{i})z_{i}\theta_{i}(z_{i})))_{i\in I^{*}}.
$$ 
In the following, we are going to compute the terms for $\inv(X)$ and $\inv(T_{G})$ explicitly.

For any $i\in I$, we let $b= \del(i,1)$. We do computation for $\inv(T_{G})$ and $\inv(X)$ separately.

We first treat the computation for $\inv(T_{G})$. By definition, for any $\tau\in \Gam_{F_{i}^{\#}}$, 
$$
\inv(T_{G})(\tau)f_{i,1} = (xn(w_{\tau})\tau(x^{-1}))f_{i,1} = (xn(w_{\tau})\tau x^{-1}\tau^{-1})f_{i,1}.
$$
\begin{enumerate}
\item When $\tau\in \Gam_{F_{i}}$, 
\begin{align*}
(xn(w_{\tau})\tau x^{-1}\tau^{-1})f_{i,1} &=
(xn(w_{\tau})\tau x^{-1})f_{i,1} = (xn(w_{\tau})\tau)(\mu(b)^{-1}e_{b})
\\
& = (\tau\circ \mu(b)^{-1})
(xn(w_{\tau}))e_{b}.
\end{align*}
Here we recall that $w_{\tau}$ is the image of $x^{-1}\tau(x)$ under the natural projection from $\mathrm{Norm}_{T}(\SO(V))\to W$. By definition, for any $\alp\in \{1,...,d \}$,
$$
x^{-1}\tau(x)e_{\alp} = 
x^{-1}\tau x\tau^{-1}e_{\alp} = 
x^{-1}\tau xe_{\alp}= x^{-1}\tau(\mu(\alp)f_{\del^{-1}(\alp)}).
$$
When $\del^{-1}(\alp) = (j,\sig)\in \Sig(V)$, we have 
$$
x^{-1}(\tau(\mu(\alp)f_{j,\sig})) = 
\tau\circ \mu(\alp) x^{-1}f_{j,\tau\sig}= 
\tau\circ \mu(\alp)
\mu(\del(j,\tau\sig))^{-1}e_{\del(j,\tau\sig)}.
$$
Hence, if we let $w^{\p}_{\tau}$ be the bijection of $\Sig(V)$ sending $(j,\sig)$ to $(j,\tau\sig)$, then the image of $x^{-1}\tau(x)$ in $W$ is given by $\del\circ w^{\p}_{\tau}\circ \del^{-1}$. In particular $w_{\tau}(b) = b$. By Lemma \ref{lem:germcal:weylsign}, we have 
$$
n(w_{\tau})(e_{b}) = 
(-1)^{r(\tau)}e_{b},
$$
where 
$$
r(\tau) = |\{
(j,\sig)\in \Sig(V)|\quad
\del(j,\tau\sig)<b<\del(j,\sig), \del(j,\sig)+b\neq d+1
\}|.
$$
Therefore finally we get the following identity
$$
\inv(T_{G})f_{i,1} = \mu(b)\tau\circ \mu(b)^{-1}
(-1)^{r(\tau)}f_{i,1},\quad \tau\in \Gam_{F_{i}}.
$$

\item
Similarly, when $\tau\in \Gam_{F_{i}^{\#}}$ and the restriction of $\tau$ to $F_{i}$ is nontrivial, we have 
$$
\inv(T_{G})(\tau)f_{i,1} = 
(xn(w_{\tau})\tau x^{-1})f_{i,\tau^{-1}} = 
(xn(w_{\tau})\tau)(\mu(d+1-b)^{-1}e_{d+1-b}).
$$
By parallel argument as above situation, we have 
$$
n(w_{\tau})(e_{d+1-b}) = (-1)^{r(\tau)}e_{b},
$$
where 
$$
r(\tau) =
|\{
(j,\sig)\in \Sig(V)|\quad
d+1-b<\del(j,\sig), \del(j,\tau\sig)<b, \del(j,b)\neq b	
\}|.
$$
Therefore we get
$$
\inv(T_{G})(\tau)f_{i,1} = 
\mu(b)\tau\circ \mu(d+1-b)^{-1}(-1)^{r(\tau)}f_{i,1}.
$$
\end{enumerate}

Then we turn to the computation for $\inv(X)$. By definition, for any $\tau\in \Gam_{F_{i}^{\#}}$, 
$$
\inv(X)(\tau)f_{i,1} = 
(\prod_{\alp\in R, \alp>0, \tau^{-1}(\alp)<0}\check{\alp}\circ \alp(X))f_{i,1}.
$$
The set of roots $R=R(T_{G},G)$ has the following description. From Lemma \ref{lem:germcal:basisproperty}, any $y\in T_{G}$ acts via scalar on the basis $\{ f_{i,\sig}\}_{(i,\sig)\in \Sig(V)}$. In other words, there exists $y_{i,\sig}\in \wb{F}^{\times}$ such that the following identity holds,
$$
y(f_{i,\sig}) = y_{i,\sig}f_{i,\sig}, \quad (i,\sig)\in \Sig(V).
$$
Moreover we have $y_{i,\sig}y_{i,\sig\tau_{i}} = 1$. For any pair $((j,\sig),(j^{\p},\sig^{\p}))\in \Sig(V)^{2}$ we let $\alp_{((j,\sig),(j^{\p},\sig^{\p}))}$ be the homomorphism from $T_{G}$ to $\wb{F}^{\times}$ sending $y\in T_{G}$ to $y_{j,\sig}y^{-1}_{j^{\p},\sig^{\p}}$. In particular, the pairs $((j,\sig),(j^{\p},\sig^{\p}))$ and $((j^{\p},\sig^{\p}\tau_{j^{\p}}),(j,\sig\tau_{j}))$ give rise to the same homomorphism from $T_{G}$ to $\wb{F}^{\times}$. Under the identification, we notice that $R$ has the following description. We consider the set of pairs $((j,\sig),(j^{\p},\sig^{\p}))\in \Sig(V)^{2}$ such that $(j,\sig)\neq (j^{\p},\sig^{\p})$ and $\del(j,\sig)+ \del(j^{\p},\sig^{\p})\neq d+1$. We equip the set with the following equivalent relation
$$
((j,\sig),(j^{\p},\sig^{\p})) \sim ((j^{\p},\sig^{\p}\tau_{j^{\p}}),(j,\sig\tau_{j})).
$$
Then the set $R$ can be identified with the set of equivalence classes of such pairs. Suppose that $((j,\sig),(j^{\p},\sig^{\p}))$ is a pair associated to a root $\alp\in R$. Then we have the following properties.
\begin{itemize}
\item $\alp(X) = \sig(a_{j}) -\sig^{\p}(a_{j^{\p}})$;

\item $\alp>0 \Leftrightarrow \del(j,\sig)<\del(j^{\p},\sig^{\p})$;

\item For any $z\in \wb{F}^{\times}$ and $(j^{\p\p},\sig^{\p\p})\in \Sig(V)$,
$$
\check{\alp}(z)f_{j^{\p\p},\sig^{\p\p}} = 
\bigg\{
\begin{matrix}
zf_{j^{\p\p},\sig^{\p\p}} ,& \text{ if $(j^{\p\p},\sig^{\p\p}) = (j,\sig)$ or $(j^{\p},\sig^{\p}\tau_{j^{\p}})$},\\
z^{-1}f_{j^{\p\p},\sig^{\p\p}}, & \text{ if $(j^{\p\p},\sig^{\p\p}) = (j,\sig\tau_{j})$ or $(j^{\p},\sig^{\p})$},\\
f_{j^{\p\p},\sig^{\p\p}}, & \text{ otherwise}.
\end{matrix}
$$
\end{itemize}

By the description we find that for any $\tau\in \Gam_{F_{i}^{\#}}$ and $\alp\in \{\alp\in R|\quad \alp>0, \tau^{-1}(\alp)<0 \}$ represented by $((j,\sig), (j^{\p},\sig^{\p}))\in \Sig(V)^{2}$, the relation
$$
\alp>0, \tau^{-1}(\alp)<0
$$
is equivalent to 
$$
\del(j,\sig)<\del(j^{\p},\sig^{\p}), \quad 
\del(j,\tau^{-1}\sig)>\del(j^{\p},\tau^{-1}\sig^{\p}).
$$
Moreover, since
$$
\alp(X)=  \sig(a_{j}) -\sig^{\p}(a_{j^{\p}}),
$$
we find that the inequality
$$
\check{\alp}\circ \alp(X)f_{i,1} \neq f_{i,1}
$$
happens only when $(i,1)\in \{ (j,\sig), (j,\sig\tau_{j}), (j^{\p},\sig^{\p}\tau_{j^{\p}}) , (j^{\p},\sig^{\p})\}$. Modulo the equivalence, we may assume that $(i,1) = (j,\sig)$ or $(j^{\p},\sig^{\p}\tau_{j^{\p}})$.

When $(i,1) = (j,\sig)$, the root associated to the pair $((i,1), (j^{\p},\sig^{\p}))$ gives
$$
\check{\alp}\circ \alp(X)f_{i,1} = 
(a_{i} - \sig^{\p}(a_{j^{\p}}))f_{i,1},
$$
and the possible pairs of $(j^{\p},\sig^{\p})$ are parametrized by the set
$$
R^{+}(\tau) = 
\{
(j,\sig)\in \Sig(V)|\quad 	
\del(i,1)<\del(j,\sig), \del(i,\tau^{-1})>\del(j,\tau^{-1}\sig), \del(j,\sig)+\del(i,1)\neq d+1
\}.
$$

When $(i,1) = (j^{\p},\sig^{\p})$ the root associated to the pair $((j,\sig), (i,1))$ gives 
$$
\check{\alp}\circ \alp(X)f_{i,1} = (\sig(a_{j}) - a_{i})^{-1}f_{i,1},
$$
and the possible pairs of $(j,\sig)$ are parametrized by the set
$$
R^{-}(\tau) = \{
(j,\sig)\in\Sig(V)
|\quad 
\del(j,\sig)<\del(i,1), \del(j,\tau^{-1}\sig)>\del(i,\tau^{-1}),
\del(j,\sig)+\del(i,1)\neq d+1
\}.
$$
It follows that for any $\tau\in \Gam_{F_{i}^{\#}}$, we have 
$$
\inv(X)(\tau)f_{i,1} = 
\prod_{(j,\sig)\in R^{+}(\tau)}
(a_{i} - \sig(a_{j}))
\prod_{(j,\sig)\in R^{-}(\tau)}
(\sig(a_{j}) - a_{i})^{-1}f_{i,1}.
$$

Based on above calculations, we find that whenever $\tau\in \Gam_{F_{i}}$, we have 
\begin{align*}
&\inv(T_{G})(\tau)\inv(X)(\tau)f_{i,1} 
\\
&= 
\mu(b)\tau\circ \mu(b)^{-1}
(-1)^{r(\tau)}
\prod_{(j,\sig)\in R^{+}(\tau)}
(a_{i} - \sig(a_{j}))
\prod_{(j,\sig)\in R^{-}(\tau)}(\sig(a_{j})-a_{i})^{-1}f_{i,1},
\end{align*}
where 
\begin{align*}
&r(\tau) = |\{(j,\sig)\in \Sig(V)|\quad \del(j,\tau\sig)<b<\del(j,\sig), \del(j,\sig)+\del(i,1)\neq d+1 \}|,
\\
&R^{+}(\tau) = \{(j,\sig)\in \Sig(V)|\quad \del(j,\tau^{-1}\sig)<\del(i,1)<\del(j,\sig), \del(j,\sig)+\del(i,1)\neq d+1 \},
\\
&R^{-}(\tau) = 
\{(j,\sig)\in \Sig(V)|\quad \del(j,\sig)<\del(i,1)<\del(j,\tau^{-1}\sig), \del(j,\sig)+\del(i,1)\neq d+1 \}.
\end{align*}

In the following we are going to find an element $z_{i}\in \wb{F}^{\times}$ such that for any $\tau\in \Gam_{F_{i}}$, we have 
$$
\tau(z_{i})z^{-1}_{i} = 
\inv(X)(\tau)\inv(T_{G})(\tau)f_{i,1}.
$$
We let 
\begin{align*}
&S = \{ (j,\sig)\in \Sig(V)|\quad \del(i,1)<\del(j,\sig), \del(j,\sig)+\del(i,1)\neq d+1\},
\\
&z_{i} = \mu(b)^{-1}\prod_{(j,\sig)\in S}(a_{i}-\sig(a_{j}))^{-1}.
\end{align*}
Then by definition, we have
$$
\tau(z_{i})z_{i}^{-1}
=\mu(b)\tau\circ \mu(b)^{-1}
\prod_{(j,\sig)\in S}(a_{i}-\sig(a_{j}))
\prod_{(j,\sig)\in S(\tau)}(a_{i}-\sig(a_{j}))^{-1},
$$
where 
$$
S(\tau) = \{(j,\sig)\in \Sig(V)|\quad \del(i,1)<\del(j,\tau^{-1}\sig), \del(j,\tau^{-1}\sig)+\del(i,1)\neq d+1 \}.
$$

We notice that by definition,
\begin{align*}
&S-S\cap S(\tau) 
= 
\{ (j,\sig)\in \Sig(V)|\quad \del(j,\tau^{-1}\sig)\leq \del(i,1)<\del(j,\sig),\del(j,\sig)+\del(i,1)\neq d+1 \}
\\
&\bigsqcup
\{ (j,\sig)\in \Sig(V)|\quad \del(i,1)<\del(j,\sig), \del(j,\sig)+\del(i,1)\neq d+1, 
\del(j,\tau^{-1}\sig)+\del(i,1) =d+1 \}.
\end{align*}
But when $\del(i,1) = \del(j,\tau^{-1}\sig)$ we have $j=i, \sig = 1$, which contradicts the fact that $\del(i,1)<\del(j,\sig)$, and when $\del(j,\tau^{-1}\sig)+\del(i,1) = d+1$ we have $(j,\sig) = (i,\tau_{i})$, which also contradicts the fact that $\del(j,\sig)+\del(i,1)\neq d+1$. It follows that
$$
S-S\cap S(\tau) = R^{+}(\tau).
$$
Similarly, we find that
\begin{align*}
&S(\tau)-S\cap S(\tau) 
= 
\{ (j,\sig)\in \Sig(V)|\quad \del(j,\sig)\leq \del(i,1)<\del(j,\tau^{-1}\sig),\del(j,\tau^{-1}\sig)+\del(i,1)\neq d+1 \}
\\
&\bigsqcup
\{ (j,\sig)\in \Sig(V)|\quad \del(i,1)<\del(j,\tau^{-1}\sig), \del(j,\tau^{-1}\sig)+\del(i,1)\neq d+1, 
\del(j,\sig)+\del(i,1) =d+1 \}.
\end{align*}
When $\del(i,1) = \del(j,\sig)$, we have $(j,\sig) = (i,1) = (j,\tau^{-1}\sig)$, hence we also have a contradiction. Similarly, when $\del(j,\sig)+\del(i,1) = d+1$, we have $(j,\sig) = (i,\tau_{i})$ which also contradicts the fact that $\del(j,\tau^{-1}\sig)+\del(i,1)\neq d+1$. Hence 
$$
S(\tau) - S\cap S(\tau) = \{
(j,\sig)\in \Sig(V)|\quad 
\del(j,\sig)<\del(i,1)<\del(j,\tau^{-1}\sig), \del(j, \tau^{-1}\sig)+\del(i,1)\neq d+1
\}.
$$
But the condition $\del(j,\tau^{-1}\sig)+\del(i,1)\neq d+1$ is equivalent to $(j,\tau^{-1}\sig)\neq (i,\tau_{i})$, which in turn is equivalent to $(j,\sig)\neq (i,\tau_{i})$, hence we realize that 
$$
S(\tau) - S\cap S(\tau) =R^{-}(\tau).
$$
It follows that
$$
\tau(z_{i})z_{i}^{-1} 
=\mu(b)\tau\circ \mu(b)^{-1}(-1)^{|R^{-}(\tau)|}
\prod_{(j,\sig)\in R^{+}(\tau)}
(a_{i}-\sig(a_{j}))
\prod_{(j,\sig)\in R^{-}(\tau)}
(\sig(a_{j})-a_{i})^{-1}.
$$
Finally by the bijective map $(j,\sig)\to (j,\tau\sig)$ and the fact that $\del(j,\sig)+\del(i,1)\neq d+1$ is equivalent to $\del(j,\tau\sig)+\del(i,1) \neq d+1$, we arrive at the equality
$$
r(\tau) = |R^{-}(\tau)|.
$$ 
It follows that
$$
\tau(z_{i})z_{i}^{-1}f_{i,1} = \inv(T_{G})(\tau)\inv(X)(\tau)f_{i,1}, \quad \forall \tau\in \Gam_{F_{i}}.
$$

Now by Lemma \ref{lem:germcal:isocoh}, we are going to compute the following element
$$
(\inv(T_{G})(\tau)\inv(X)(\tau)f_{i,1})z_{i}\tau(z_{i})
$$
for some $\tau\in \Gam_{F_{i}^{\#}}$ whose restriction to $F_{i}$ is nontrivial.

By the computation that we have done above, whenever $\tau\in \Gam_{F_{i}^{\#}}$ whose restriction to $F_{i}$ is nontrivial, we have
$$
\inv(T_{G})(\tau)f_{i,1} = 
\mu(b)\tau\circ \mu(d+1-b)^{-1}(-1)^{r(\tau)}f_{i,1},
$$
where 
$$
r(\tau) = |\{(j,\sig)\in \Sig(V)|\quad \del(i,\tau_{i})<\del(j,\sig), \del(j,\tau \sig)<b,\del(j,\sig)\neq \del(i,1)\}|,
$$
and 
$$
\inv(X)(\tau)f_{i,1} = 
\prod_{(j,\sig)\in R^{+}(\tau)}
(a_{i}-\sig(a_{j}))\prod_{(j,\sig)\in R^{-}(\tau)}
(\sig(a_{j})-a_{i})^{-1}f_{i,1},
$$
where 
\begin{align*}
&R^{+}(\tau) = \{
(j,\sig)\in \Sig(V)|\quad 
\del(i,1)<\del(j,\sig), \del(j,\tau^{-1}\sig)<\del(i,\tau_{i}), \del(j,\sig)+\del(i,1)\neq d+1
\},
\\
&R^{-}(\tau) = \{
(j,\sig)\in \Sig(V)|\quad
\del(j,\sig)<\del(i,1), \del(j,\tau^{-1}\sig)>\del(i,\tau_{i}), \del(j,\sig)+\del(i,1)\neq d+1
\}.
\end{align*}
Moreover, by definition,
$$
\tau(z_{i}) = 
\tau\circ \mu(b)^{-1}
\prod_{(j,\sig)\in S(\tau)}
(-a_{i}-\sig(a_{j}))^{-1}.
$$
We set 
$$
S^{\p}(\tau) = \{
(j,\sig)\in \Sig(V)|\quad b<\del(j,\tau^{-1}\sig\tau_{j}), \del(j,\tau^{-1}\sig \tau_{j})+\del(i,1)\neq d+1
 \}.
$$
Using the fact that $\tau_{j}(a_{j}) = -a_{j}$, we realize that 
$$
\tau(z_{i}) = 
(-1)^{|S^{\p}(\tau)|}
\tau\circ \mu(b)^{-1}
\prod_{(j,\sig)\in S^{\p}(\tau)}
(a_{i}-\sig(a_{j}))^{-1}.
$$
It follows that 
\begin{align*}
&(\inv(T_{G})(\tau)\inv(X)(\tau)f_{i,1})z_{i}\tau(z_{i}) 
\\
&= 
\tau(\mu(b)\mu(d+1-b))^{-1}
(-1)^{r(\tau)+|R^{-}(\tau)|+|S^{\p}(\tau)|}
\prod_{(j,\sig)\in \Sig(V)}(a_{i}-\sig(a_{j}))^{-m(j,\sig)}f_{i,1}.
\end{align*}
Here we set 
$$
m= 1_{R^{-}(\tau)}+1_{S^{\p}(\tau)}+1_{S}-1_{R^{+}(\tau)},
$$
where for a subset $U$ of $\Sig(V)$, $1_{U}$ is the characteristic function of $U$.

By definition, $R^{+}(\tau)\subset S$, and 
$$
S-R^{+}(\tau) = \{(j,\sig)\in \Sig(V)|\quad 
\del(i,1)<\del(j,\sig), \del(i,\tau_{i})
\leq \del(j,\tau^{-1}\sig), \del(j,\sig)+\del(i,1)\neq d+1
\}.
$$
Now since 
$$
R^{-}(\tau)= \{
(j,\sig)\in \Sig(V)|\quad
\del(j,\sig)<\del(i,1), \del(j,\tau^{-1}\sig)>\del(i,\tau_{i}), \del(j,\sig)+\del(i,1)\neq d+1
\},
$$
we notice that when $\del(j,\tau^{-1}\sig)=\del(i,\tau_{i})$, we have $(j,\sig) = (i,1)$. Hence the set $R^{-}(\tau)$ can also be written as
$$
R^{-}(\tau)= \{
(j,\sig)\in \Sig(V)|\quad
\del(j,\sig)<\del(i,1), \del(j,\tau^{-1}\sig)\geq \del(i,\tau_{i}), \del(j,\sig)+\del(i,1)\neq d+1
\}.
$$
It follows that the sets $S-R^{+}(\tau)$ and $R^{-}(\tau)$ are disjoint, and their union gives 
$$
\{(j,\sig)\in \Sig(V)|\quad \del(i,1)\neq \del(j,\sig), \del(j,\tau^{-1}\sig)\geq \del(i,\tau_{i}), \del(j,\sig)+\del(i,1)\neq d+1\}.
$$
On the other hand, 
$$
S^{\p}(\tau) = \{
(j,\sig)\in \Sig(V)|\quad 
b<\del(j,\tau^{-1}\sig \tau_{j}), \del(j,\tau^{-1}\sig\tau_{j})+\del(i,1)\neq d+1
\}
$$
can also be written as 
$$
S^{\p}(\tau) = \{
(j,\sig)\in \Sig(V)|\quad 
\del(j,\tau^{-1}\sig)<\del(i,\tau_{i}), (j,\tau^{-1}\sig \tau_{j})\neq (i,\tau_{i})
\}.
$$
Moreover, the relation $(j, \tau^{-1}\sig\tau_{j})\neq (i,\tau_{i})$ is equivalent to $(j,\sig)\neq(i, \tau_{i})$. Hence 
$$
S^{\p}(\tau) = \{
(j,\sig)\in \Sig(V)|\quad 
\del(j,\tau^{-1}\sig)<\del(i,\tau_{i}), \del(j,\sig)+\del(i,1)\neq d+1
\}.
$$

It follows that the sets $S-R^{+}(\tau)$, $R^{-}$ and $S^{\p}(\tau)$ are disjoint to each other, and their union can be written as $\Sig(V)-\{(i,1),(i,\tau_{i})\}$.

It follows that the product is equal to 
$$
\prod_{(j,\sig)\in \Sig(V)-\{ (i,1), (i,\tau_{i}) \}}
(a_{i}-\sig(a_{j}))^{-1},
$$
which is equal to $(P_{X}^{\p}(a_{i})/2a_{i})^{-1}$. 

Using the bijection $(j,\sig)\to (j,\tau^{-1}\sig)$ and the same argument as before we have that $|R^{-}(\tau)| = r(\tau)$. From the construction of the element $x\in T_{G}$, we have 
\begin{align*}
&\mu(b)\mu(d+1-b) =\eta/2(-1)^{b+1}[F_{i}:F]\sig(c_{i})^{-1} = \eta/2(-1)^{b+1}[F_{i}:F]c_{i}^{-1},\quad \text{ if $b\leq d/2$},
\\
&\mu(b)\mu(d+1-b) =\eta/2(-1)^{b}[F_{i}:F]\sig(c_{i})^{-1} = \eta/2(-1)^{b}[F_{i}:F]c_{i}^{-1},\quad \text{ if $b> d/2$}.
\end{align*}

We also have $\tau(c_{i}) = c_{i}$. It follows that the desired invariant is equal to 
\begin{align*}
&\tau(\frac{\eta}{2}(-1)^{b+1}[F_{i}:F]c_{i}^{-1})^{-1}
(-1)^{|S^{\p}(\tau)|}2a_{i}P^{\p}_{X}(a_{i})^{-1}
\\
&=
\bigg\{
\begin{matrix}
(\eta/2(-1)^{b+1}
[F_{i}:F]c_{i}^{-1})^{-1}
(-1)^{|S^{\p}(\tau)|}
2a_{i}P^{\p}_{X}(a_{i})^{-1}, & \text{ if $b\leq d/2$}, \\
(\eta/2(-1)^{b}
[F_{i}:F]c_{i}^{-1})^{-1}
(-1)^{|S^{\p}(\tau)|}
2a_{i}P^{\p}_{X}(a_{i})^{-1}, & \text{ if $b> d/2$},
\end{matrix}
\end{align*}
Finally by the fact that $|S^{\p}(\tau)| = d-b-1$ whenever $b\leq d/2$, and $d-b$ whenever $b>d/2$, it follows that the desired invariant is equal to
$$
4\eta^{-1}[F_{i}:F]^{-1}c_{i}a_{i}P^{\p}_{X}(a_{i})^{-1}.
$$
Up to square class we may ignore the $4$ part. Through taking the inverse and we obtain the eventual formula
$$
\eta[F_{i}:F]c_{i}^{-1}a_{i}^{-1}P_{X}^{\p}(a_{i}).
$$
\end{proof}

\subsection{Explicit examples}\label{sub:germcal:explicitexample}

In this section, following \cite[11.4]{waldspurger10}, we are going to calculate the regular germs for explicit regular semi-simple elements using Theorem \ref{thm:germformula}. We compute the invariants in Theorem \ref{thm:germformula} explicitly via Proposition \ref{pro:germcal:computation}.

We focus on the case when $d$ is even and $d\geq 4$, since only in this case will we have $|\nil_{\reg}(\Fs\Fo(V))|>1$.

In the following we consider the situation when $(V,q)$ is split and quasi-split but not split separately. We are going to pick up particular regular semi-simple elements following the description in Section \ref{sub:germcal:conjugacy}.

\begin{itemize}
\item When $(V,q)$ is split. We fix two quadratic extensions $F_{1},F_{2}$ of $F$ such that the product of the two quadratic characters $\sgn_{F_{1}/F}\sgn_{F_{2}/F}$ is the trivial character on $F^{\times}$, i.e. $F_{1}\simeq F_{2}$ as a field extension of $F$. For any $i=1,2$, we fix $a_{i}\in F^{\times}_{i}$ such that $\tau_{i} (a_{i})= -a_{i}$ and $a_{1}\neq \pm a_{2}$, where $\tau_{i}$ is the nontrivial element in $\Gal(F_{i}/F)$. For any $c\in F^{\times}$, we consider $F_{1}$ (resp. $F_{2}$) equipped with quadratic form $c\RN_{F_{1}/F}$ (resp. $-c\RN_{F_{2}/F}$). Let $\wt{Z}$ be a hyperbolic space of dimension $d-4$. Then we can find that $F_{1}\oplus F_{2}\oplus \wt{Z} \simeq V$ as split quadratic space.  We fix a maximal split subtorus $\wt{T}$ of the special orthogonal group associated to $\wt{Z}$. Let $\wt{\Ft}(F)$ be the $F$-points of the Lie algebra of $\wt{T}$. We fix a regular semi-simple element $\wt{S}$ in $\wt{\Ft}(F)$. 

We consider the regular semi-simple element $X_{F_{1}}\in \Fs\Fo(V)(F)$ that acts by multiplication by $a_{1}$ (resp. $a_{2}$) on $F_{1}$ (resp. $F_{2}$) and by $\wt{S}$ on $\wt{Z}$. From Lemma \ref{lem:germcal:stableconjugate} we know that the conjugacy classes of $X_{F_{1}}$ within the stable conjugacy class is determined by $c\in F^{\times}/\RN_{F_{1}/F}(F^{\times})$. In the following we will note $X^{+}_{F_{1}}$ to be the element corresponding to $c=c^{+}$ such that $\sgn_{F_{1}/F}(2c) = \sgn_{F_{1}/F}(\RN_{F_{1}/F}(a_{1}) - \RN_{F_{2}/F}(a_{2}))$ and $X^{-}_{F_{1}}$ corresponds to $c=c^{-}$ such that $\sgn_{F_{1}/F}(c^{-}) = -\sgn_{F_{1}/F}(c^{+})$.

To keep the notation consistent with the following quasi-split but not split case, we let $\eta = 1$ in our situation.

\item When $(V,q)$ is quasi-split but not split. There exists a quadratic extension $E/F$ and an element $\eta\in F^{\times}$ such that the anisotropic kernel of $q$ is of the form $\eta\cdot\RN_{E/F}$. We fix two quadratic extensions $F_{1},F_{2}$ of $F$ such that the product of two quadratic characters $\sgn_{F_{1}/F}\sgn_{F_{2}/F}$ is equal to $\sgn_{E/F}$. For $i=1,2$, suppose that $a_{i}\in F^{\times}_{i}$ such that $\tau_{F_{i}}(a_{i}) = -a_{i}$ and $a_{1}\neq \pm a_{2}$. We fix $c\in F^{\times}$ such that $\sgn_{E/F}(2\eta c \RN_{F_{1}/F}(a_{1})) = 1$. Then we can find that $F_{1}\oplus F_{2}\oplus \wt{Z}\simeq V$ as quadratic space where we consider $F_{1}$ (resp. $F_{2}$) equipped with quadratic form $c\RN_{F_{1}/F}$ (resp. $-c\RN_{F_{2}/F}$).

Parallel with the split situation, we consider the regular semi-simple element $X_{F_{1}}\in \Fs\Fo(V)(F)$ that acts by multiplication by $a_{1}$ (resp. $a_{2}$) on $F_{1}$ (resp. $F_{2}$) and by $\wt{S}$ on $\wt{Z}$. Again from Lemma \ref{lem:germcal:stableconjugate} we note $X^{+}_{F_{1}}$ the element corresponding to $c=c^{+}$ such that $\sgn_{F_{1}/F}(2c) = \sgn_{F_{1}/F}(\eta) \sgn_{F_{1}/F}(\RN_{F_{1}/F}(a_{1}) - \RN_{F_{2}/F}(a_{2}))$
and $X^{-}_{F_{1}}$ corresponds to $c=c^{-}$ such that $\sgn_{F_{1}/F}(c^{-}) = -\sgn_{F_{1}/F}(c^{+})$.
\end{itemize}

Our goal is to compute the regular germ associated to $X^{\pm}_{F_{1}}$. We follow the proof of \cite[11.4]{waldspurger10}.

\begin{lem}\label{lem:germcal:Xpm}
Suppose that $\CO\in \mathrm{Nil}_{\reg}(\Fs\Fo(V)(F))$.
\begin{enumerate}
\item We have 
$\Gam_{\CO}(X_{\qd}) = 1$ for any $X_{\qd}\in \Ft_{\qd}$.

\item Suppose that $d$ is even and $d\geq 4$. We have the equality
$$
\Gam_{\CO}(X^{+}_{F_{1}}) - \Gam_{\CO}(X^{-}_{F_{1}}) = 
\sgn_{F_{1}/F}(\nu \eta)
$$
if $\CO=\CO_{\nu}$ with $\nu\in \CN^{V}$.
\end{enumerate}
\end{lem}
\begin{proof}
The first statement follows from \cite[Section~3.4]{beuzart2015local}. The second statement follows from Theorem \ref{thm:germformula} and the explicit computation of the invariants following Proposition \ref{pro:germcal:computation} . We will go over the proof of \cite[11.4]{waldspurger10} in the following.

For the sign $\zet = \pm$, we note $T^{\zet}(F)$ the maximal subtorus of $\SO(V)(F)$ such that $X^{\zet}_{F_{1}}\in \Ft^{\zet}(F)$. Suppose that $\nu\in \CN^{V}$ and $N\in \CO_{\nu}$. We are going to evaluate the invariant $\mathrm{inv}(X^{\zet}_{F_{1}})\mathrm{inv}(T^{\zet})/\mathrm{inv}_{T^{\zet}}(N)$. All the elements belong to the cohomology group $H^{1}(T^{\zet})$. From Lemma \ref{lem:germcal:centralizer} we have $H^{1}(T^{\zet}) = \{ \pm 1\} \times \{\pm 1\}$. The invariants depend on the choice of pining. Following the normalization of the basis for quadratic space in the beginning of Section \ref{sub:germcal:computeendoscopic} we make choice following above discussion by making the element $\eta$ as above and multiplying by $2(-1)^{d/2-1}$. Then from Proposition \ref{pro:germcal:computation} we have the following explicit formula for $\mathrm{inv}(X^{\zet}_{F_{1}})\mathrm{inv}(T^{\zet})$. We have
\begin{align*}
\mathrm{inv}(X^{\zet}_{F_{1}})\mathrm{inv}(T^{\zet}) =
& 
(\sgn_{F_{1}/F}((-1)^{d/2-1}\eta (c^{\zet})^{-1}a^{-1}_{1}P^{\p}(a_{1}))
,
\\
&\sgn_{F_{2}/F}((-1)^{d/2-1}\eta (c^{\zet})^{-1}a^{-1}_{2}P^{\p}(a_{2}))
)
\end{align*}
where $P$ is the characteristic polynomial of $X_{F_{1}}$ acting on $V$ and $P^{\p}$ is the derivative. We let $(\pm \wt{s}_{j})_{j=3,...,d/2}$ be the eigenvalues of $\wt{S}$ acting on $\wt{Z}$. They belong to $F^{\times}$ since $\wt{T}$ is split. Then we have
$$
P(T) = (T^{2}+\mathrm{N}_{F_{1}/F}(a_{1}))
(T^{2}+\mathrm{N}_{F_{2}/F}(a_{2}))
\prod_{j=3,...,d/2}(T^{2}- \wt{s}^{2}_{j}).
$$
Then
\begin{align*}
a_{1}^{-1}P^{\p}(a_{1}) 
&= 
2(-\mathrm{N}_{F_{1}/F}(a_{1})+\mathrm{N}_{F_{2}/F}(a_{2}))
\prod_{j=3,...,d/2}
(-\mathrm{N}_{F_{1}/F}(a_{1}) - \wt{s}^{2}_{j})
\\
&=
2(-1)^{d/2-1}
(\mathrm{N}_{F_{1}/F}(a_{1}) - \mathrm{N}_{F_{2}/F}(a_{2}))
\prod_{j=3,...,d/2}
(\mathrm{N}_{F_{1}/F}(a_{1})+\wt{s}_{j}^{2})
\end{align*}
Note here we have $\wt{s}^{2}_{j}+\mathrm{N}_{F_{1}/F}(a_{1}) = \mathrm{N}_{F_{1}/F}(\wt{s}_{j}+a_{1})$, therefore $\sgn_{F_{1}/F}(\wt{s}^{2}_{j}+\mathrm{N}_{F_{1}/F}(a_{1}))= 1$. Similar calculation can be obtained when one switch $F_{1}$ to $F_{2}$. It turns out that we have 
\begin{align*}
\mathrm{inv}(X^{\zet}_{F_{1}})
\mathrm{inv}(T^{\zet}) 
&= 
(\sgn_{F_{1}/F}(2\eta(c^{\zet})^{-1}(\mathrm{N}_{F_{1}/F}(a_{1})- \mathrm{N}_{F_{2}/F}(a_{2}))), 
\\
&
\sgn_{F_{2}/F}
(2\eta(c^{\zet})^{-1}(\mathrm{N}_{F_{1}/F}(a_{1})-\mathrm{N}_{F_{2}/F}(a_{2})))
).
\end{align*}
By the definition of $c^{\zet}$, we get 
$$
\mathrm{inv}(X^{\zet}_{F_{1}})
\mathrm{inv}(T^{\zet}) = (\zet,\zet),
$$
where one can identify $\zet$ as an element of $\{ \pm 1\}$. 

The pining is determined by the standard fixed regular element $N^{*}$, where from the notation in Section \ref{sub:germcal:computeendoscopic}, we have $N^{*} = \sum_{j=1,...,d/2}X_{\alp_{j}}$. Note $\nu^{*}$ the element of $\CN^{V}$ such that $N^{*}\in \CO_{\nu^{*}}$. We can define a cocycle $d_{N}$ of $\Gal(\wb{F}/F)$ in $T^{\zet}$ as follows. If $\nu=\nu^{*}$, $d_{N} = 1$. If $\nu\neq \nu^{*}$ we set $E_{N} = F(\sqrt{\frac{\nu}{\nu^{*}}})$. Then, for $\sig\in \Gal(\wb{F}/F)$, we have $d_{N}(\sig) = 1$ if $\sig\in \Gal(\wb{F}/E_{N})$ and $d_{N}(\sig) = -1$ otherwise. The invariant $\mathrm{inv}_{T^{\zet}}(N)$ is the image in $H^{1}(T)$ of the cocycle $d_{N}$. We can calculate the image 
$$
\mathrm{inv}_{T^{\zet}}(N) = (\sgn_{F_{1}/F}(\nu/\nu^{*}), \sgn_{F_{2}/F}(\nu/\nu^{*})).
$$
In fact, we can force that $\sgn_{E}(\nu/\nu^{*})=1$ following the parametrization of $\nil_{\reg}(\Fs\Fo(V))$ in Section \ref{sub:ggptriples:definitions}. Then
$$
\mathrm{inv}_{T^{\zet}}(N) = (\sgn_{F_{1}/F}(\nu/\nu^{*}), \sgn_{F_{1}/F}(\nu/\nu^{*})).
$$
The element $N^{*}$ stabilize the hyperplane generated by $e_{1},...,e_{d/2-1}, e_{d/2}+e_{d/2+1},e_{d/2+2},...,e_{d}$.  The anisotropic kernel of the restriction of the form $q$ to the hyperplane is the restriction of $q$ to the line $e_{d/2}+e_{d/2+1}$, or say $q(e_{d/2}+e_{d/2+1}) = \eta$, therefore we have $\nu^{*} = \eta$. Finally
$$
\mathrm{inv}(X^{\zet}_{F_{1}})
\mathrm{inv}(T^{\zet})/\mathrm{inv}_{T^{\zet}}(N) = (\zet \sgn_{F_{1}/F}(\nu \eta), \zet \sgn_{F_{1}/F}(\nu \eta)).
$$
It follows that we have 
$$
\Gam_{\CO_{\nu}}
(X^{\zet}_{F_{1}}) = 
\bigg\{
	\begin{matrix}
	1, & \sgn_{F_{1}/F}(\nu \eta) = \zet,\\
	0, & \sgn_{F_{1}/F}(\nu \eta) = -\zet.
	\end{matrix}
$$
This is exactly the identity that we want.
\end{proof}

\bibliographystyle{amsalpha}
\bibliography{zhilin}

\end{document}